\documentclass[12pt]{article}
\usepackage{amsmath}
\usepackage{graphicx,psfrag,epsf}
\usepackage{enumerate}
\usepackage{url} 
\usepackage{graphicx,amsmath,amssymb,epsfig}

\usepackage{colortbl,xr}

\newcommand{\blind}{1}

\newtheorem{theorem}{Theorem}
\newtheorem{algorithm}{Algorithm}[section]

\newtheorem{corollary}{Corollary}
\newtheorem{lemma}{Lemma}

\newtheorem{remark}{Comment}[section]
\numberwithin{remark}{section}

\newcommand{\be}{\begin{eqnarray}}
\newcommand{\ee}{\end{eqnarray}}

\newcommand{\ba}{\begin{array}}
\newcommand{\ea}{\end{array}}
\newcommand{\bs}{\begin{align}\begin{split}\nonumber}
\newcommand{\bsnumber}{\begin{align}\begin{split}}
\newcommand{\es}{\end{split}\end{align}}

\renewcommand{\(}{\left(}
\renewcommand{\)}{\right)}
\renewcommand{\[}{\left[}
\renewcommand{\]}{\right]}
\renewcommand{\hat}{\widehat}

\newcommand{\cc}{\mathbf{c}}

\newcommand{\Gn}{\mathbb{G}_n}

\newcommand{\Ep}{{\mathrm{E}}}
\newcommand{\barEp}{\bar \Ep}
\newcommand{\En}{{\mathbb{E}_n}}

\renewcommand{\Pr}{{\mathrm{P}}}

\def\RR{ {\Bbb{R}}}

\def\z{{{\iota}}}

\def\supp{{\rm supp}}

\newcommand{\ceil}[1]{\left\lceil #1 \right\rceil}
\newcommand{\semin}[1]{\phi_{{\rm min}}(#1)}
\newcommand{\semax}[1]{\phi_{{\rm max}}(#1)}
\renewcommand{\hat}{\widehat}
\renewcommand{\leq}{\leqslant}
\renewcommand{\geq}{\geqslant}
\newcommand{\sign}{ {\rm sign}}

\newcommand{\underf}{{\underline{f}}}
\newcommand{\mtau}{{\theta \tau}}
\newcommand{\gtau}{{\tau}}

\DeclareMathOperator{\diag}{diag}

\begin{document}

\def\spacingset#1{\renewcommand{\baselinestretch}%
{#1}\small\normalsize} \spacingset{1}


\if1\blind
{
  \title{\bf Valid Post-Selection Inference in High-Dimensional Approximately Sparse Quantile Regression Models}
  \author{Alexandre Belloni \hspace{.2cm}\\
    The Fuqua School of Business, Duke University\\
    and \\
    Victor Chernozhukov \hspace{.2cm} \\
    Department of Economics, Massachusetts Institute of Technology\\
        and \\
    Kengo Kato\hspace{.2cm} \\
    Graduate School of Economics, University of Tokyo}

  \maketitle
} \fi

\if0\blind
{
  \bigskip
  \bigskip
  \bigskip
  \begin{center}
    {\LARGE\bf Valid Post-Selection Inference in High-Dimensional Approximately Sparse Quantile Regression Models}
\end{center}
  \medskip
} \fi

\vspace{-1cm}

\bigskip
\begin{abstract}
This work proposes new inference methods for a regression coefficient of interest in a (heterogeneous) quantile regression model. We consider a high-dimensional model where the number of regressors potentially exceeds the sample size but a subset of them suffice to construct a reasonable approximation to the conditional quantile function. The proposed methods are  (explicitly or implicitly)  based on orthogonal score functions that protect against moderate model selection mistakes, which are often inevitable in the approximately sparse model considered in the present paper. We establish the uniform validity of the proposed confidence regions for the quantile regression coefficient. Importantly, these methods directly apply to more than one variable and a continuum of quantile indices. In addition, the performance of the proposed methods is illustrated through  Monte-Carlo experiments and an empirical example, dealing with risk factors in childhood malnutrition.
\end{abstract}

\noindent%
\vfill

\newpage
\spacingset{1.45} 
\section{Introduction}

Many applications of interest require the measurement of the distributional impact of a policy (or treatment) on the relevant outcome variable. Quantile treatment effects have emerged as an important concept for measuring such distributional impacts (see, e.g., \cite{Lehmann,koenker:book}). In this work we focus on the quantile treatment effect $\alpha_\tau$ on a policy/treatment variable $d$ of an outcome of interest $y$ in the (heteroskedastic) partially linear model:
$$
\tau\textrm{-quantile}(y\mid z, d) = d\alpha_\tau + g_\tau(z),
$$
where $g_\tau$ is the (unknown) confounding function of the other covariates $z$ which can be well approximated by a linear combination of $p$ technical controls. Large $p$ arises due to the existence of many features (as in genomic studies and econometric applications) and/or the use of basis expansions in non-parametric approximations. When $p$ is comparable or larger than the sample size $n$, this brings forth the need to perform model selection or regularization.

We propose methods to construct estimates and confidence regions for the coefficient of interest $\alpha_\tau$ based upon robust post-selection procedures.  We establish the (uniform) validity of the proposed methods in  a non-parametric setting. Model selection in those settings generically leads to a moderate mistake and traditional arguments based on perfect model selection do not apply. Therefore, the proposed methods are developed to be robust to  such model selection mistakes. Furthermore, they are directly applicable to construction of simultaneous confidence bands when $d$ is multivariate and a continuum of quantile indices is of interest.

Broadly speaking the main obstacle to construct confidence regions with (asymptotically) correct coverage is the estimation of the confounding function $g_\tau$ that is (generically) non-regular because of the high-dimensionality. To overcome this difficulty we construct (explicitly or implicitly) an orthogonal score function which leads to a moment condition that is immune to first-order mistakes in the estimation of the confounding function $g_\tau$. The construction is based on preliminary estimation of the confounding function and properly partialling out the confounding factors $z$ from the policy/treatment variable $d$. The former can be achieved via $\ell_1$-penalized quantile regression \cite{BC-SparseQR,kato} or post-selection quantile regression based on $\ell_1$-penalized quantile regression \cite{BC-SparseQR}. The latter is carried out by heteroscedastic post-Lasso \cite{T1996,BellChenChernHans:nonGauss} applied to a density-weighted equation.
Then we propose two estimators for $\alpha_{\tau}$ based on: (i) a moment condition based on an orthogonal score function; (ii) a density-weighted quantile regression with all the variables selected in the previous steps. The latter method is reminiscent of the ``post-double selection" method proposed in \cite{BCH2011:InferenceGauss,BelloniChernozhukovHansen2011}. Explicitly or implicitly the last step estimates $\alpha_\tau$ by minimizing a Neyman-type score statistic \cite{Neyman1979}.\footnote{We mostly focus on selection as a means of regularization, but certainly other regularizations (e.g. the use of $\ell_1$-penalized fits per se) are possible, although they perform no better than the methods we focus on.}

Under mild moment conditions and approximate sparsity assumptions, we establish that the estimator $\check\alpha_\tau$, as defined by either method (see Algorithms \ref{Alg:1} and \ref{Alg:2} below), is root-$n$ consistent and asymptotically normal,
\begin{equation}\label{Eq:InferentialAlpha}
\sigma_n^{-1} \sqrt{n} (\check \alpha_\tau - \alpha_\tau) \rightsquigarrow N (0,1),
\end{equation}
where $\rightsquigarrow$ denotes convergence in distribution; in addition the estimator $\check{\alpha}_{\tau}$ admits a (pivotal) linear representation. Hence the confidence region defined by
\begin{equation}
\label{Eq:InferenceDirect}
\mathcal{C}_{\xi,n}:=\{ \alpha \in \RR : |\alpha - \check\alpha_\tau| \leq \hat\sigma_n \Phi^{-1}(1-\xi/2)/\sqrt{n} \}
\end{equation}
has asymptotic coverage probability of $1-\xi$ provided that the estimate $\hat \sigma^2_n$ is consistent for $\sigma_n^2$, namely,  $\hat\sigma_n^2/\sigma_n^2 = 1+ o_{\Pr}(1)$. In addition, we establish that a Neyman-type score statistic
$L_n(\alpha)$ is asymptotically distributed as the chi-squared distribution with one degree of freedom when evaluated at the true value $\alpha=\alpha_\tau$, namely,
\begin{equation}
\label{Eq:InferenceLn}
n L_n(\alpha_\tau)  \rightsquigarrow \chi^2(1),
\end{equation}
which in turn allows the construction of another confidence region:
\begin{equation}
\label{Def:Inv} \mathcal{I}_{\xi,n}:=\{ \alpha \in \mathcal{A}_\tau:  n L_n(\alpha) \leq (1-\xi) \mbox{-quantile of} \ \chi^2(1)\},
\end{equation}
which has asymptotic coverage probability of $1-\xi$. These convergence results hold under array asymptotics, permitting the data-generating process $\Pr = \Pr_n$ to change with $n$, which implies that these convergence results hold uniformly over large classes of data-generating processes. In particular, our results do not require separation of regression coefficients away from zero (the so-called ``beta-min" conditions) for their validity. Importantly, we discuss how the procedures naturally allow for construction of simultaneous confidence bands for many parameters based on a (pivotal) linear representation of the proposed estimator.


Several recent papers study the problem of constructing confidence regions after model selection while allowing $p\gg n$.
In the case of linear mean regression, \cite{BCH2011:InferenceGauss} proposes a double selection inference in a parametric setting with homoscedastic Gaussian errors; \cite{BelloniChernozhukovHansen2011} studies a double selection procedure in a non-parametric setting with heteroscedastic errors; \cite{c.h.zhang:s.zhang} and  \cite{vandeGeerBuhlmannRitov2013} propose methods based on $\ell_1$-penalized estimation combined with ``one-step" correction in parametric models. Going beyond mean regression models, \cite{vandeGeerBuhlmannRitov2013} provides high level conditions for the one-step estimator applied to smooth generalized linear problems,  \cite{BelloniChernozhukovKato2013a} analyzes confidence regions for a parametric homoscedastic LAD regression model under primitive conditions based on the instrumental LAD regression, and \cite{BelloniChernozhukovWei2013} provides two post-selection procedures to build confidence regions for the logistic regression.    None of the aforementioned papers deal with the problem of the present paper.

Although related in spirit with
our previous work \cite{BelloniChernozhukovHansen2011,BelloniChernozhukovKato2013a,BelloniChernozhukovWei2013},
new tools and major departures from the previous works are required. First is the need to  accommodate the non-differentiability of the loss function (which translates into discontinuity of the score function) and the non-parametric setting. In particular, we establish new finite sample bounds for the prediction norm on the estimation error of $\ell_1$-penalized quantile regression in nonparametric models that extend results of \cite{BC-SparseQR,kato}. Perhaps more importantly, the use of post-selection methods in order to reduce bias and improve finite sample performance requires sparsity of the estimates. Although sharp sparsity bounds for $\ell_1$-penalized methods are available for smooth loss functions, those are not available for quantile regression precisely due to the lack of differentiability. This led us to developing sparse estimates with provable guarantees by suitable truncation while preserving the good rates of convergence despite of possible additional model selection mistakes. To handle heteroscedsaticity, which is a common feature in many applications, consistent estimation of the conditional density is necessary whose analysis is new in high-dimensions.
Those estimates are used as weights in the weighted Lasso estimation for the auxiliary regression ((\ref{Eq:indirect}) below). In addition, we develop new finite sample bounds for Lasso with estimated weights as the zero mean condition is not assumed to hold for each observation but rather to the average across all observations. Because the estimation of the conditional density function is at a slower rate, it affects penalty choices, rates of convergence, and sparsity of the Lasso estimates.

This work and some of the papers cited above achieve an important uniformity guarantee with respect to the (unknown) values of the parameters. These uniform properties translate into more reliable finite sample performance of the proposed inference procedures because they are robust with respect to (unavoidable) model selection mistakes. There is now substantial theoretical and empirical evidence on the potential poor finite sample performance of inference methods  that rely on perfect model selection when applied to models without separation from zero of the coefficients (i.e., small coefficients). Most of the criticism of these procedures are consequence of negative results established in \cite{LeebPotscher2005}, \cite{leeb:potscher:hodges}, and the references therein. 


\textbf{Notation.}
We work with triangular array data $\{ \omega_{i,n} : i=1,\dots,n; n=1,2,3,\dots\}$ where for each $n$, $\{ \omega_{n,i} ; i=1,\dots,n \}$ is defined on
 the probability space $(\Omega, \mathcal{S}, \Pr_n)$. Each  $\omega_{i,n}= (y_{i,n}', z_{i,n}', d_{i,n}')'$
is a vector which are  i.n.i.d., that is, independent across $i$ but not necessarily identically distributed. Hence all parameters that characterize the distribution of  $\{\omega_{i,n} : i=1,\dots,n\}$ are
implicitly indexed by $\Pr_n$ and thus by $n$.  We omit this dependence from the notation for the sake of simplicity.  We use $\En$ to abbreviate the notation $n^{-1}\sum_{i=1}^n$; for example, $\En[f] := \En[f(\omega_{i})] := n^{-1} \sum_{i=1}^n f(\omega_{i})$. We also use the following notation:
$\barEp[f] := \Ep \left [ \En[f] \right] =  \Ep \left [ \En[f(\omega_{i})] \right ]= n^{-1} \sum_{i=1}^n \Ep[f(\omega_{i})]$.
The $\ell_2$-norm is denoted by
$\|\cdot\|$; the $\ell_0$-``norm'' $\|\cdot\|_0$ denotes the number of non-zero components of a vector; and the $\ell_{\infty}$-norm $\| \cdot \|_{\infty}$ denotes the maximal absolute value in the components of a vector.
Given a vector $\delta \in \RR^p$ and a set of
indices $T \subset \{1,\ldots,p\}$, we denote by $\delta_T \in \RR^p$ the vector in which $\delta_{Tj} = \delta_j$ if $j\in T$ and $\delta_{Tj}=0$ if $j \notin T$.

\section{Setting}
\label{Sec:Model}

For a quantile index $\tau \in (0,1)$, we consider a partially linear conditional quantile model
\begin{equation}
\label{Eq:direct}
y_i = d_i\alpha_\tau + g_\tau(z_i) + \epsilon_i,   \ \  \tau\textrm{-quantile}(\epsilon_i\mid d_i, z_i) = 0, \ \ i=1,\ldots,n,
\end{equation}
where $y_i$ is the outcome variable, $d_i$ is the policy/treatment variable,  and confounding factors are represented by the variables $z_i$ which impact the equation through an unknown function $g_\tau$.  We shall use a large number $p$ of technical controls $x_i=X(z_i)$  to achieve an accurate approximation to the function $g_\tau$ in (\ref{Eq:direct}) which takes the form:
\begin{equation}
\label{ApproxSparse}
g_\tau(z_i) = x_i'\beta_\tau + r_{\gtau i}, \ \ i=1,\ldots,n, \
\end{equation} where $r_{\gtau i}$ denotes an approximation error. We view $\beta_\tau$ and $r_{\gtau}$ as nuisance parameters while the main parameter of interest is
$\alpha_\tau$ which describes the  impact of the treatment on the conditional quantile (i.e., quantile treatment effect).

In order to perform robust inference with respect to model selection mistakes, we construct a moment condition based on a score function that satisfies an additional orthogonality property that makes them immune to first-order changes in the value of the nuisance parameter. 
Letting $f_i = f_{\epsilon_i}(0\mid d_i, z_i)$ denote the conditional density at 0 of the disturbance term $\epsilon_i$ in (\ref{Eq:direct}), the construction of the orthogonal moment condition is based on the linear projection of the regressor of interest $d_i$ weighted by $f_i$ on the $x_i$ variables weighted by $f_i$
\begin{equation}
\label{Eq:indirect}
f_i d_i = f_i x_i'\theta_{0\tau} + v_i, \ \ \ i=1,\ldots,n, \ \ \barEp[f_ix_iv_i]=0,
\end{equation}
where $\theta_{0\tau} \in \arg\min \barEp[ f_i^2(d_i-x_i'\theta)^2]$. The orthogonal score function $\psi_i(\alpha):=(1\{y_i\leq d_i\alpha+ x_i'\beta_\tau + r_{\gtau}\}-\tau)v_i$ leads to a moment condition to estimate $\alpha_\tau$,
\begin{equation}
\label{Eq:Id}
\Ep[ (1\{y_i\leq d_i\alpha_\tau + x_i'\beta_\tau + r_{\gtau i}\} - \tau)v_i ]=0,
\end{equation}
and satisfies the following orthogonality condition with respect to first-order changes in the value of the nuisance parameters $\beta_\tau$ and $\theta_{0\tau}$:
\begin{equation}
\label{Eq:OrthoI}
\begin{array}{l}
\left.\partial_\beta \barEp[ (1\{y_i\leq d_i\alpha_\tau + x_i'\beta + r_{\gtau i}\} - \tau)f_i(d_i-x_i'\theta_{0\tau}) ] \right|_{\beta=\beta_\tau} = 0, \ \ \mbox{and}\\ \left.\partial_\theta \barEp[ (1\{y_i\leq d_i\alpha_\tau + x_i'\beta_\tau + r_{\gtau i}\} - \tau)f_i(d_i-x_i'\theta) ] \right|_{\theta=\theta_{0\tau}} = 0.
\end{array}
\end{equation}

In order to handle the high-dimensional setting, we assume that $\beta_\tau$ and $\theta_{0\tau}$ are approximately sparse, namely, it is possible to choose sparse vector $\beta_\tau$ and $\theta_\tau$ such that:
\begin{equation}
\| \theta_\tau\|_0 \leq s, \ \ \|\beta_\tau\|_0 \leq s,  \ \
\barEp [(x_i'\theta_{0\tau}-x_i'\theta_\tau)^2]\lesssim s/n \text{ and } \barEp [(g(z_i)-x_i'\beta_\tau)^2]\lesssim s/n.
\end{equation}
The latter equation requires that it is possible to choose the sparsity index $s$ so that the mean squared approximation error is of no larger order than the variance of the oracle estimator for estimating the coefficients in the approximation. See \cite{chen:Chapter} for a detailed discussion of this notion of approximate sparsity.

\subsection{Methods}
\label{Sec:KnownDensity}


The methodology based on the orthogonal score function (\ref{Eq:Id}) can be used for construction of many different estimators that have the same first-order asymptotic properties but potentially different finite sample behaviors. In the main part of the paper we present two such procedures in detail (the discussion on additional variants can be found in Subsection \ref{Sec:Variants} of the Supplementary Appendix). Our procedures use $\ell_1$-penalized quantile regression and $\ell_1$-penalized weighted least squares as intermediate steps (we collect the recommended choices of the user-chosen parameters in Remark \ref{Remark:Param} below).  The first procedure stated in Algorithm \ref{Alg:1} is based on the explicit construction of the orthogonal score function.

\begin{algorithm}\label{Alg:1}{\rm (Orthogonal score function.)} \\
\enspace \emph{Step 1}. Compute $(\hat\alpha_\tau,\hat\beta_\tau)$ from $\ell_1$-penalized quantile regression of $y$ on  $d$ and $x$.\\
\enspace \emph{Step 2}. Compute $(\widetilde\alpha_\tau,\widetilde\beta_\tau)$ from quantile regression of $y$ on  $d$ and $\{x_j : | \hat\beta_{\tau j} | \geq \lambda_\tau / \{\En[x_{ij}^2]\}^{1/2}\}$.\\
\enspace \emph{Step 3}. Estimate the conditional density $\hat f$ via (\ref{Eq:hatf}) or (\ref{Eq:hatfsecond}).\\
\enspace \emph{Step 4}. Compute $\widetilde\theta_\tau$ from the post-Lasso estimator of $\hat f d$ on $\hat f x$.\\
\enspace \emph{Step 5}. Construct the score function $\hat \psi_i(\alpha)= (\tau - 1\{y_i\leq d_i\alpha + x_i'\widetilde \beta_\tau\}) \hat f_i(d_i-x_i'\widetilde\theta_\tau)$. \\
\enspace \emph{Step 6}. For $L_n(\alpha) = |\En[\hat \psi_i(\alpha)]|^2/\En[\hat \psi_i^2(\alpha)]$, set $\check\beta_\tau = \widetilde\beta_\tau$  and $\check\alpha_\tau^{OS} \in \arg\min_{\alpha \in \mathcal{A}_\tau} L_n(\alpha)$.
\end{algorithm}

Step 6 of Algorithm \ref{Alg:1} solves the empirical analog of (\ref{Eq:Id}). We will show validity of the confidence regions for $\alpha_\tau$ defined in (\ref{Eq:InferenceDirect}) and (\ref{Def:Inv}). We note that the truncation in Step 2 for the solution of the penalized quantile regression (provably) induces a sparse solution with the same rate of convergence as the original estimator. This is required because the post-selection methods  exhibit better finite sample behaviors  in our simulations.

The second algorithm is based on selecting relevant variables from equations (\ref{Eq:direct}) and (\ref{Eq:indirect}), and running a weighted quantile regression.

\begin{algorithm}\label{Alg:2}{\rm (Weighted Double Selection.)} \\
\enspace \emph{Step 1}.  Compute $(\hat\alpha_\tau,\hat\beta_\tau)$ from $\ell_1$-penalized quantile regression of $y$ on  $d$ and $x$.\\
\enspace \emph{Step 2}. Estimate the conditional density $\hat f$ via (\ref{Eq:hatf}) or (\ref{Eq:hatfsecond}).\\
\enspace \emph{Step 3}. Compute $\hat\theta_\tau$ from the Lasso estimator of $\hat f d$ on $\hat f x$.\\
\enspace \emph{Step 4}. Compute $(\check\alpha_\tau^{DS},\check\beta_\tau)$ from quantile regression of $\hat f y$ on $\hat f d$ and $\{\hat fx_{j}:  j\in \supp(\hat \theta_\tau)\}\cup \{\hat fx_{j}: | \hat\beta_{\tau j} | \geq \lambda_\tau / \{\En[x_{ij}^2]\}^{1/2}\}$.
\end{algorithm}

Although the orthogonal score function is not explicitly constructed  in Algorithm \ref{Alg:2}, inspection of the proof reveals that an orthogonal score function is constructed implicitly via the optimality conditions of the weighted quantile regression in Step 4.

\begin{remark}[Choices of User-Chosen Parameters]\label{Remark:Param}
For  $\gamma = 0.05/n$, we set the penalty levels for the heteroscedastic Lasso and the $\ell_1$-penalized quantile regression as
\begin{equation}
\label{Parameter:lambda}
\lambda := 1.1\sqrt{n}2\Phi^{-1}(1-\gamma/2p) \quad \text{and} \quad \lambda_\tau := 1.1\sqrt{n\tau(1-\tau)}\Phi^{-1}(1-\gamma/2p).
\end{equation}
The penalty loading $\widehat \Gamma_\tau= \diag [ \hat \Gamma_{\tau jj} , j =1,\dots,p]$
 is a diagonal matrix defined by the following procedure: (1) Compute the post-Lasso estimator $\widetilde \theta_\tau^{0}$ based on $\lambda$ and initial values $\hat \Gamma_{\tau jj}  =   \max_{1 \leq i\leq n} \|\hat f_i x_i\|_\infty \{\En[\hat f_i^2d_i^2]\}^{1/2}$. (2)  Compute the residuals $\widehat v_i = \hat f_i(d_i - x_i'\widetilde\theta_\tau^{0})$ and update
\begin{equation}
\label{Parameter:Gamma}
\hat \Gamma_{\tau jj}  =   \sqrt{\En[\hat f_i^2 x^2_{ij} \widehat v_i^2]}, \ j=1,\ldots,p.
\end{equation}
In Algorithm \ref{Alg:1} we have used the following parameter space for $\alpha$:
\begin{equation}
\label{Parameter:Atau}
\mathcal{A}_\tau = \{ \alpha \in \RR : |\alpha - \widetilde \alpha_\tau| \leq 10\{\En[d_i^2]\}^{-1/2} / \log n \}.
\end{equation}
\end{remark}

\begin{remark}[Estimating Standard Errors]
There are different possible choices of estimators for $\sigma_n$:
\begin{equation}
\label{Def:EstSigma}
\begin{array}{l}
\hat\sigma_{1n}^2 :=\tau(1-\tau) \left ( \En[\widetilde v_i^2] \right )^{-1}, \quad \hat\sigma_{2n}^2 := \tau(1-\tau) \left [ \{ \En[\hat f_i^2( d_i, x_{i\check T}')'(d_i,x_{i\check T}')] \}^{-1} \right ]_{11}, \\
\hat\sigma_{3n}^2 := \left ( \En[\hat f_i d_i \widetilde v_i]\right)^{-2}\En[(1\{y_i\leq d_i\check\alpha_\tau+x_i'\check\beta_\tau\}-\tau)^2\widetilde v_i^2],
\end{array}
\end{equation}
where $\check T = \supp(\check \beta_\tau) \cup \supp(\hat \theta_\tau)$ is the set of controls used in the double selection quantile regression. Although all three estimates are consistent under similar regularities conditions, their finite sample behaviors might differ. Based on the small-sample performance in computational experiments, we recommend the use of $\hat\sigma_{3n}$ for the orthogonal score estimator and $\hat\sigma_{2n}$ for the double selection estimator.
\end{remark}

\subsection{Estimation of Conditional Density Function}\label{Comment:EstDensity}

The implementation of the algorithms in Section \ref{Sec:KnownDensity} requires an estimate of the conditional density function $f_i$ which is typically unknown under heteroscedasticity.
Following \cite{koenker:book},  we shall use the observation that $1/ f_i = \partial Q(\tau\mid d_i, z_i)/\partial \tau$ to estimate $f_i$ where $Q(\cdot\mid d_i,z_i)$ denotes the conditional quantile function of the outcome. Let  $\hat Q(u\mid z_i,d_i)$ denote an estimate of the conditional $u$-quantile function $Q(u\mid z_i,d_i)$,
 based on  either $\ell_1$-penalized quantile regression or an associated post-selection method,  and let $h =h_n \to 0$ denote a bandwidth parameter. Then an estimator of $f_i$ can be constructed as
\begin{equation}
\label{Eq:hatf} \hat f_i = \frac{2h}{\hat Q(\tau+h\mid z_i,d_i)-\hat Q(\tau-h\mid z_i,d_i)}.
\end{equation}
When the conditional quantile function is three times continuously differentiable, this estimator is based on the first order partial difference of the estimated conditional quantile function, and so it has the bias of order $h^2$. Under additional smoothness assumptions, an estimator that has a  bias of order $h^4$ is given by
{\small
\begin{equation}
\label{Eq:hatfsecond}
\hat f_{i} = \frac{h}{ \frac{3}{4}\{\hat Q(\tau+h\mid z_i,d_i)-\hat Q(\tau-h\mid z_i,d_i)\} - \frac{1}{12} \{\hat Q(\tau+2h\mid z_i,d_i)-\hat Q(\tau-2h\mid z_i,d_i)\} }.
\end{equation}
}

\begin{remark}[Implementation of the estimates $\hat f_i$]\label{Implementation:hatf}
There are several possible choices of tuning parameters to construct the estimates $\hat f_i$. In particular the bandwidth choices set in the R package `quantreg' from \cite{koenker2016quantreg} exhibits good empirical behavior. In our theoretical analysis we coordinate the bandwidth choice with the choice of the penalty level of the density weighted Lasso.  In Subsection \ref{Sec:Bandwidth} of the Supplementary Appendix we discuss in more detail the requirements associated with different choices for penalty level $\lambda$ and bandwidth $h$. Together with the recommendations made in  Remark \ref{Remark:Param}, we suggest to construct $\hat f_i$ as in (\ref{Eq:hatf}) with bandwidth $h:= \min\{ n^{-1/6}, \tau(1-\tau)/2\}$.
\end{remark}

\section{Theoretical Analysis}
\label{Sec:MainResults}

\subsection{Regularity Conditions}

In this section we provide regularity conditions that are sufficient for validity of the main estimation and inference results. In what follows, let $c,C$, and $q$ be given (fixed) constants with $c > 0, C \geq 1$ and $q \geq 4$, and let $\ell_n \uparrow \infty, \delta_n \downarrow 0$, and $\Delta_n \downarrow 0$ be given sequences of positive constants. We  assume that  the following condition holds for the data generating process $\Pr = \Pr_n$ for each $n$.

\textbf{Condition AS($\Pr$)}. \textit{(i) Let $\{(y_i,d_i,x_i=X(z_i)) : i=1,\ldots,n\}$ be independent random vectors that obey the model described in (\ref{Eq:direct}) and (\ref{Eq:indirect}) with $\|\theta_{0\tau}\|+\|\beta_\tau\|+|\alpha_\tau|\leq C$. (ii) There exists $s \geq 1$ and vectors $\beta_{\tau}$ and $\theta_\tau$ such that $x_i'\theta_{0\tau} = x_i' \theta_{\tau} + r_{\mtau i}, \ \|\theta_{\tau}\|_0 \leq s, \   \barEp [r_{\mtau i}^2] \leq C s/n$, $\|\theta_{0\tau}-\theta_\tau\|_1 \leq s\sqrt{\log(pn)/n}$, and $g_\tau(z_i) = x_i' \beta_{\tau} + r_{\gtau i}, \  \|\beta_\tau\|_0 \leq s,  \  \barEp[r_{\gtau i}^2] \leq  C s/n$.
(iii) The conditional distribution function of $\epsilon_i$ is absolutely continuous with continuously differentiable density $f_{\epsilon_i\mid d_i, z_i}(\cdot\mid d_i,z_i)$ such that $0<\underline{f} \leq f_i \leq \sup_t f_{\epsilon_i\mid d_i, z_i}(t\mid d_i,z_i) \leq \bar f \leq C$ and $\sup_{t} |f_{\epsilon_i\mid d_i, z_i}'(t\mid d_i, z_i)|\leq \bar f'\leq C$.}

Condition AS(i) imposes the setting discussed in Section \ref{Sec:Model} in which the  error term $\epsilon_i$ has zero conditional $\tau$-quantile. The approximate sparsity on the high-dimensional parameters is stated in Condition AS(ii). Condition AS(iii) is a standard assumption on the conditional density function in the quantile regression literature (see \cite{koenker:book}) and the instrumental quantile regression literature (see \cite{ch:iqrWeakId}). Next we summarize the moment conditions we impose.

\textbf{Condition M($\Pr$)}.  \textit{(i) We have $\barEp[\{(d_{i},x_{i}')\xi\}^2] \geq c\|\xi\|^2$ and $\barEp[\{(d_i,x_i')\xi\}^4] \leq C\|\xi\|^4$ for all $\xi \in \RR^{p+1}$, $c \leq \min_{1 \leq j\leq p} \barEp[ |f_ix_{ij}v_i-\Ep[f_ix_{ij}v_i]|^2]^{1/2} \leq \max_{1\leq j\leq p}  \barEp[|f_i x_{ij} v_i|^3]^{1/3} \leq C$. (ii) The approximation error satisfies $|\barEp[f_iv_ir_{\gtau i}]|\leq \delta_nn^{-1/2}$ and $\barEp[(x_i'\xi)^2r_{\gtau i}^2]\leq C\|\xi\|^2 \barEp[r_{\gtau i}^2]$ for all $\xi \in \RR^{p}$. (iii) Suppose that $K_q = \Ep[\max_{1 \leq i \leq n}\|(d_i,v_i,x_i')' \|_\infty^q]^{1/q}$ is finite and satisfies $(K_q^2s^2 + s^3)\log^3(pn) \leq n\delta_n$ and $K_q^4s\log(pn)\log^3n\leq \delta_nn$.} 

Condition M(i) imposes moment conditions on the variables. Condition M(ii) imposes requirements on the approximation error. Condition M(iii) imposes growth conditions on $s$, $p$, and $n$. In particular these conditions imply that the population eigenvalues of the design matrix are bounded away from zero and from above. They ensure that sparse eigenvalues and restricted eigenvalues are well behaved which are used in the analysis of penalized estimators and sparsity properties needed for the post-selection estimator.

\begin{remark}[Handling Approximately Sparse Models]
\label{RemarkExtraMoment}
To handle approximately sparse models to represent $g_\tau$ in (\ref{ApproxSparse}), we assume a near orthogonality between  $r_{\gtau}$ and $f v$, namely $\barEp[ f_i v_i r_{\gtau i} ] = o(n^{-1/2})$.
This condition is automatically satisfied if the orthogonality condition in (\ref{Eq:indirect}) can be strengthen to $\Ep[f_iv_i\mid z_i]=0$. However, it can be satisfied under weaker conditions as discussed in Subsection \ref{DiscussionNearOrtho} of the Supplementary Appendix.
\end{remark}

Our last set of conditions pertains to the estimation of the conditional density function $(f_i)_{i=1}^n$ which has a non-trivial impact on the analysis. We denote by $\mathcal{U}$ the finite set of quantile indices used in the estimation of the conditional density. Under mild regularity conditions the estimators (\ref{Eq:hatf}) and (\ref{Eq:hatfsecond}) achieve
\begin{equation}
\label{Ratef}
\hat f_i - f_i = O\left(h^{\bar k} + \frac{1}{h}\sum_{u\in \mathcal{U}} |\hat Q(\tau+u\mid d_i,z_i)-\hat Q(\tau-u\mid d_i,z_i)|\right),
\end{equation}
where $\bar k = 2$ for (\ref{Eq:hatf}) and $\bar k  =4$ for (\ref{Eq:hatfsecond}).
Condition D summarizes sufficient conditions to account for the impact of density estimation via (post-selection) $\ell_1$-penalized quantile regression estimators.

{\bf Condition D.} \textit{For $u \in \mathcal{U}$, assume that $u\textrm{-quantile}(y_i\mid z_i, d_i) = d_i\alpha_u + x_i'\beta_u + r_{ui}$, $f_{ui}=f_{y_i\mid d_i,z_i}(d_i\alpha_u + x_i'\beta_u + r_{ui} \mid  z_i,d_i)\geq c$  where  $\barEp[r_{u i}^2] \leq \delta_n n^{-1/2}$ and $|r_{ui}|\leq \delta_nh$ for all $i=1,\ldots,n$,  and the vector $\beta_u$ satisfies $\|\beta_u\|_0\leq s$.  (ii) For $\widetilde s_\mtau = s+\frac{ns\log (n\vee p)}{h^2\lambda^2}+ \left(\frac{nh^{\bar k}}{\lambda}\right)^2$, suppose $h^{\bar k}\sqrt{\widetilde s_{\mtau}\log(pn) }\leq \delta_n$,  $h^{-2}K_q^2s\log(pn)\leq \delta_n n$, $\lambda K_q^2\sqrt{s} \leq \delta_n n$, $h^{-2}s\widetilde s_\mtau \log(pn) \leq \delta_n n$, $\lambda\sqrt{s \widetilde s_\mtau \log(pn)} \leq \delta_n n$,  and $K_q^2\widetilde s_{\mtau}\log^2(pn)\log^3(n) \leq \delta_nn$.}

Condition D(i) imposes the approximately sparse assumption for the $u$-conditional quantile function for quantile indices $u$ in a neighborhood of the quantile index $\tau$. Condition D(ii) provides growth conditions relating $s$, $p$, $n$, $h$ and $\lambda$. Subsection \ref{Sec:Bandwidth} in the Supplementary Appendix discusses specific choices of penalty level $\lambda$ and of bandwidth $h$ together with the implied conditions on the triple $(s,p,n)$. In particular they imply that sparse eigenvalues of order $\widetilde s_{\mtau}$ are well behaved.

\subsection{Main results} In this section we state our theoretical results. We establish the first order equivalence of the proposed estimators. We construct the estimators as defined as in Algorithm \ref{Alg:1} and \ref{Alg:2} with parameters $\lambda_\tau$ as in (\ref{Parameter:lambda}), $\widehat \Gamma_\tau$ as in (\ref{Parameter:Gamma}), and $\mathcal{A}_\tau$ as in (\ref{Parameter:Atau}). The choices of $\lambda$ and $h$ satisfy Condition D.

\begin{theorem}\label{theorem:inferenceAlg1prime}
Let $\{\Pr_n\}$ be a sequence of data-generating processes.  Assume that conditions $\mathrm{AS} (\Pr)$, $\mathrm{M} (\Pr)$ and $\mathrm{D} (\Pr)$ are satisfied with $\Pr = \Pr_n$ for each $n$.  Then the orthogonal score estimator $\check \alpha_\tau^{OS}$ based on Algorithm \ref{Alg:1} and the double selection estimator $\check \alpha_\tau^{DS}$ based on Algorithm \ref{Alg:2} are first order equivalent, $\sqrt{n}(\check \alpha_\tau^{OS}-\check \alpha_\tau^{DS})=o_P(1).$ Moreover, either estimator  satisfies
$$
\sigma_n^{-1} \sqrt{n} (\check \alpha_\tau - \alpha_\tau) = \mathbb{U}_n(\tau) + o_P(1) \quad \mbox{and} \quad  \mathbb{U}_n(\tau) \rightsquigarrow N(0,1),
$$
where $\sigma^2_n=   \tau(1-\tau)\barEp [v_i^2]^{-1}$ and
$\mathbb{U}_n(\tau) := \tau(1-\tau) \{ \barEp[v_i^2]\}^{-1/2} n^{-1/2} \sum_{i=1}^n ( \tau - 1\{U_i\leq \tau\})v_{i}$, and $U_1, \ldots, U_n$ are i.i.d. uniform random variables on $(0,1)$ independent  from $v_1,\ldots, v_n$. Furthermore,
$$
nL_n(\alpha_\tau) = \mathbb{U}_n^2(\tau)+ o_P(1) \quad \mbox{and} \quad \mathbb{U}_n^2(\tau)\rightsquigarrow \chi^2(1).
$$
The result continues to apply if $\sigma^2_n$ is replaced by any of the estimators in (\ref{Def:EstSigma}), namely, $\hat\sigma_{kn}/\sigma_n = 1 + o_\Pr(1)$ for $k=1,2,3$.
\end{theorem}

The asymptotically correct coverage of the confidence regions $\mathcal{C}_{\xi,n}$  and $\mathcal{I}_{\xi,n}$ as defined in (\ref{Eq:InferenceDirect}) and (\ref{Def:Inv}) follows immediately. Theorem \ref{theorem:inferenceAlg1prime} relies on post model selection estimators which in turn rely on achieving sparse estimates $\hat \beta_\tau$ and $\hat\theta_\tau$. The sparsity of $\hat\theta_\tau$ is derived in Section \ref{Sec:EstLasso} in the Supplemental Appendix under the recommended penalty choices. The sparsity of $\hat \beta_\tau$ is not guaranteed under the recommended choices of penalty level $\lambda_\tau$ which leads to sharp rates. We bypass that by truncating small components to zero (as in Step 2 of Algorithm \ref{Alg:1}) which (provably) preserves the same rate of convergence and ensures the sparsity.

In addition to the asymptotic normality, Theorem \ref{theorem:inferenceAlg1prime} establishes that the rescaled estimation error $\sigma_n^{-1}\sqrt{n}(\check\alpha_\tau - \alpha_\tau)$ is approximately equal
to the process $\mathbb{U}_n(\tau)$, which is pivotal conditional on $v_1,\ldots, v_n$. Such a property is very useful since it is easy to simulate $\mathbb{U}_n(\tau)$ conditional on $v_1,\ldots, v_n$. Thus this representation provides us with another procedure to construct confidence intervals without relying on asymptotic normality which are useful for the construction of simultaneous confidence bands; see Section \ref{Sec:Multiple}.

Importantly, the results in Theorem \ref{theorem:inferenceAlg1prime}  allow for the data generating process to depend on the sample size $n$ and have no requirements on the separation from zero of the coefficients. In particular these results allow for sequences of data generating processes for which perfect model selection is not possible. In turn this translates into uniformity properties over a large class of data generating processes. Next we formalize these uniform properties. We let $\mathcal{P}_{n}$ denote the collection of distributions $\Pr$ for the data $\{ (y_{i},d_{i}, z_i')' \}_{i=1}^{n}$
 such that Conditions AS, M and D are satisfied for given $n$.  This is the collection of all approximately sparse models where the above sparsity conditions, moment conditions, and growth conditions are satisfied.
Note that the uniformity results for the approximately sparse and heteroscedastic case are new even under fixed $p$ asymptotics.



\begin{corollary}[\textbf{Uniform Validity of Confidence Regions}]
\label{cor:Uniformity22}
Let $\mathcal{P}_{n}$ be the collection of all distributions of $\{ (y_{i},d_{i},z_i')' \}_{i=1}^{n}$ for which Conditions $\mathrm{AS}$, $\mathrm{M}$, and $\mathrm{D}$ are  satisfied for  given $n \geq 1$. Then the confidence regions  $\mathcal{C}_{\xi,n}$ and $\mathcal{I}_{\xi,n}$ defined based on either the orthogonal score estimator or by the double selection estimator are asymptotically uniformly valid 
$$
\lim_{n \to \infty} \sup_{\Pr \in \mathcal{P}_n} |\Pr ( \alpha_\tau \in \mathcal{C}_{\xi,n} ) - (1- \xi)| =0
\quad  \mbox{and} \quad
\lim_{n \to \infty} \sup_{\Pr \in \mathcal{P}_n} |\Pr ( \alpha_\tau \in \mathcal{I}_{\xi,n} ) - (1- \xi)| =0.
$$
\end{corollary}

\subsection{Simultaneous Inference over $\tau$ and Many Coefficients}\label{Sec:Multiple}

In some applications we are interested on building confidence intervals that are simultaneously valid for many coefficients as well as for a range of quantile indices $\tau \in \mathcal{T}\subset (0,1)$ a fixed compact set. The proposed methods directly extend to the case of $d \in \RR^{K}$ and $\tau \in \mathcal{T}$
$$
\tau\textrm{-quantile}(y\mid z, d) = \sum_{j=1}^K d_j\alpha_{\tau j} + \tilde g_\tau(z).
$$
Indeed, for each $\tau \in \mathcal{T}$ and each $k=1,\ldots,K$, estimates can be obtained by applying the methods to the model (\ref{Eq:direct}) as
$$
\tau\textrm{-quantile}(y\mid z, d) = d_k\alpha_{\tau k} + g_\tau(z) \ \ \mbox{where} \ \  g_\tau(z) := \tilde g_\tau(z) + \sum_{j\neq k} d_j\alpha_{\tau j}.
$$
For each $\tau \in \mathcal{T}$, Step 1 and the conditional density function $f_i$, $i=1,\ldots,n$, are the same for all $k=1,\ldots,K$. However, Steps 2 and 3 adapt to each quantile index and each coefficient of interest. The uniform validity of $\ell_1$-penalized methods for a continuum of problems (indexed by $\mathcal{T}$ in our case) has been established for quantile regression in \cite{BC-SparseQR} and for least squares in \cite{BCFH2013program}. The conclusions of Theorem \ref{theorem:inferenceAlg1prime}  are uniformly valid over $k=1,\ldots,K$ and $\tau \in \mathcal{T}$ (in the $\ell_\infty$-norm).


Simultaneous confidence bands are constructed by defining the following critical value
$$
c^*(1-\xi) = \inf\left\{ t: \Pr\left( \sup_{\tau \in \mathcal{T},k=1,\ldots,K} |\mathbb{U}_n(\tau,k)| \leq t \mid \{d_i,z_i\}_{i=1}^n\right) \geq
1-\xi \right\},
$$
where the random variable $\mathbb{U}_n(\tau,k)$ is pivotal conditional on the data, namely,
$$
\mathbb{U}_n(\tau,k) := \frac{\{\tau(1-\tau)\barEp[v_{\tau k i}^2]\}^{-1/2}}{\sqrt{n}} \sum_{i=1}^n ( \tau - 1\{U_i\leq \tau\})v_{\tau k i},
$$
where $U_i$ are i.i.d. uniform random variables on $(0,1)$ independent from $\{d_i,z_i\}_{i=1}^n$, and $v_{\tau k i}$ is the error term in the decomposition (\ref{Eq:indirect}) for the pair $(\tau,k)$. Therefore $c^*(1-\xi)$ can be estimated since estimates of $v_{\tau k i}$ and $\sigma_{n\tau k}$, $\tau \in \mathcal{T}$ and $k=1,\ldots,K$, are available. Uniform confidence bands can be defined as
$$
[ \check \alpha_{\tau k} - \sigma_{n\tau k}c^*(1-\xi)/\sqrt{n},  \check \alpha_{\tau k} + \sigma_{n\tau k}c^*(1-\xi)/\sqrt{n} ] \ \ \mbox{for} \ \ \tau \in \mathcal{T}, \ k = 1,\ldots,K.
$$

\section{Empirical Performance}


\subsection{Monte-Carlo Experiments}

Next we provide a simulation study to assess the finite sample performance of the proposed estimators and confidence regions. We focus our discussion on the double selection estimator as defined in Algorithm \ref{Alg:2} which exhibits a better performance. We consider the median regression case ($\tau =1/2$) under the  following data generating process:
\begin{align}
&y = d\alpha_\tau + x'(c_y\nu_0) + \epsilon, \quad  \epsilon \sim N(0,\{2-\mu+ \mu d^2\}/2),\label{direct} \\
&d = x'(c_d\nu_0) + \tilde v, \quad \tilde v \sim N(0,1),
\end{align}
where $\alpha_\tau = 1/2$, $\theta_{0j} = 1/j^2, j=1,\ldots,p$, $x = (1,z')'$ consists of an intercept and covariates $z \sim N(0,\Sigma)$, and the errors $\epsilon$ and $\tilde v$ are independent. The dimension $p$ of the covariates $x$ is $300$, and the sample size $n$ is $250$.  The regressors are correlated with $\Sigma_{ij} = \rho^{|i-j|}$ and $\rho = 0.5$. In this case, $f_i= 1/\{\sqrt{\pi(2-\mu +\mu d^2)}\}$ so that the coefficient $\mu \in \{0, 1\}$ makes the conditional density function of $\epsilon$ homoscedastic if $\mu = 0$ and heteroscedastic if $\mu = 1$. The coefficients $c_y$ and $c_d$ are used to control the $R^2$ in the
 equations: $y - d\alpha_\tau = x'(c_y\nu_0) + \epsilon$ and $ d = x'(c_d\nu_0) + \tilde v$ ; we denote the values of $R^2$ in each equation by $R^2_y$ and $R_d^2$. We consider values $(R^2_y, R^2_d)$ in the set $\{0, .1, .2, \ldots, .9\}\times \{0, .1, .2, \ldots, .9\}$. Therefore we have 100 different designs and perform $500$ Monte-Carlo repetitions for each design. For each repetition we draw new vectors $x_i$'s and errors $\epsilon_i$'s and $\tilde v_i$'s. 

We perform estimation of $f_i$'s via  (\ref{Eq:hatf}) even in the homoscedastic case ($\mu=0$), since we do not want to rely on whether the assumption of homoscedasticity is valid or not. We use $\hat \sigma_{2n}$ as the standard error estimate  for the post double selection estimator based on Algorithm \ref{Alg:2}. As a benchmark we consider the standard (naive) post-selection procedure that applies $\ell_1$-penalized median regression of $y$ on $d$ and $x$ to select a subset of covariates that have predictive power for $y$, and then runs median regression of $y$ on $d$ and the selected covariates, omitting the covariates that were not selected. We report the rejection frequency of the confidence intervals with the nominal coverage probability of $95\%$.  Ideally we should see the rejection rate of $5\%$, the nominal level, regardless of the underlying generating process $\Pr \in \mathcal{P}_n$.  This is the so called uniformity property or honesty property of the confidence regions (see, e.g., \cite{RomanoWolf2007}, \cite{RomanoShaikh20012}, and \cite{LeebPotscher2006}).

In the homoscedastic case, reported on the left column of Figure \ref{Fig:SimFirst}, we have the empirical rejection probabilities for the naive post-selection procedure on the first row.  These empirical rejection probabilities deviate strongly away from the nominal level of $5\%$, demonstrating the striking lack of robustness of this standard method. This is perhaps expected due to the Monte-Carlo design having regression coefficients not well separated from zero (that is,  the ``beta min" condition does not hold here). In sharp contrast, we see that the proposed procedure performs substantially better, yielding empirical rejection probabilities close to the desired nominal level of $5\%$. In the right column of Figure \ref{Fig:SimFirst}  we report the results for  the heteroscedastic case ($\mu = 1$). Here too we see the striking lack of robustness of the naive post-selection procedure. We also see that the confidence region based on the post-double selection method significantly outperforms the standard method, yielding empirical rejection probabilities close to the nominal level of $5\%$.

\begin{figure}[h!]
\includegraphics[width=0.4\textwidth]{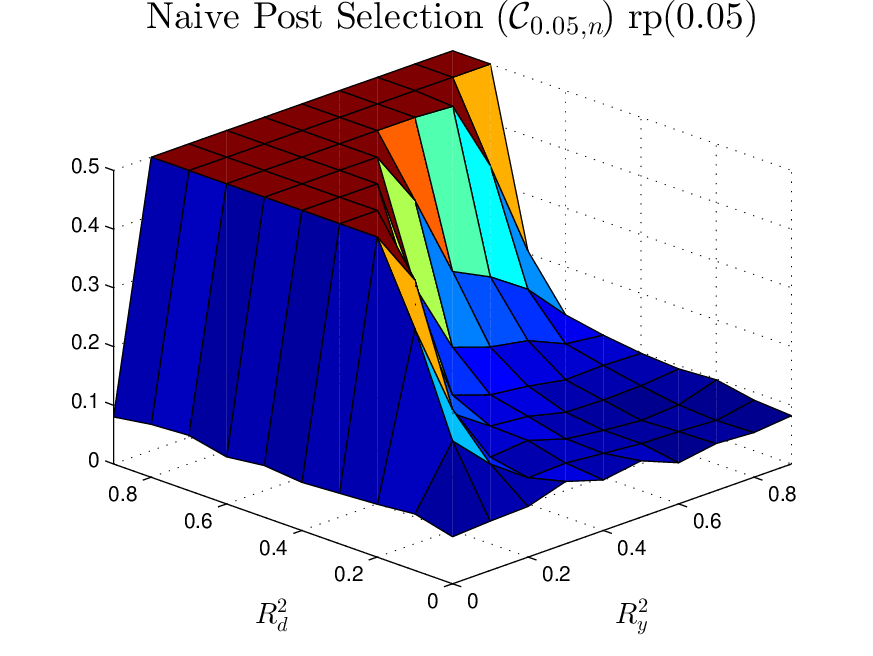}
\includegraphics[width=0.4\textwidth]{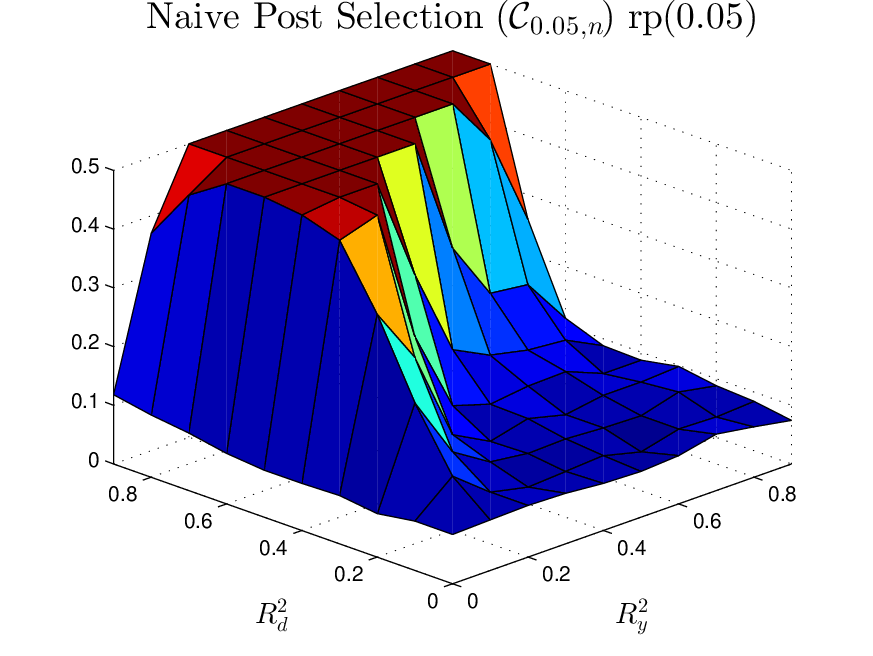}
\includegraphics[width=0.4\textwidth]{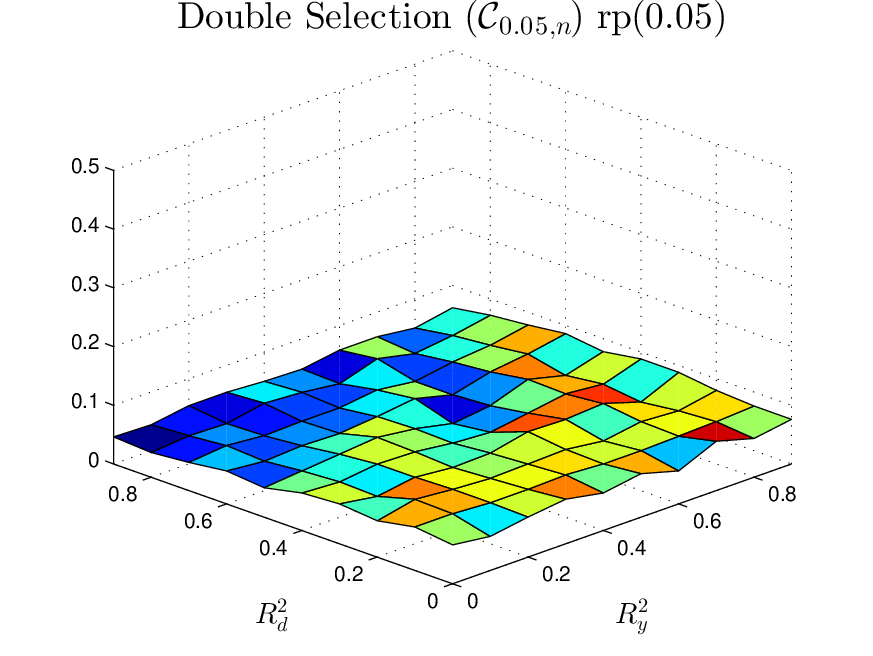} \hspace{3.25cm}
\includegraphics[width=0.4\textwidth]{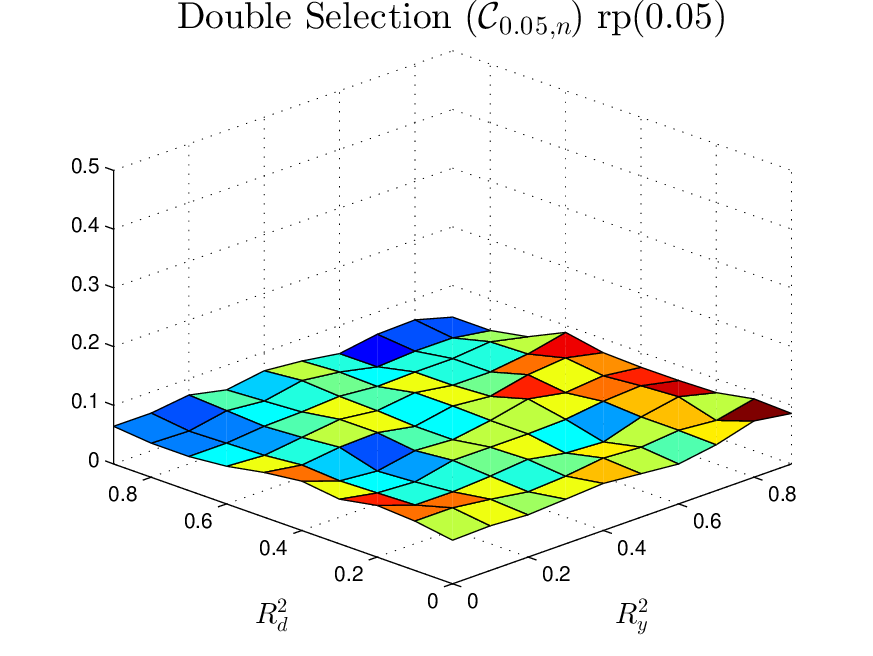}
\caption{\footnotesize  The left column is the homoscedastic design ($\mu = 0$), and the right column is the heteroscedastic design. We the figure displays the rejection probabilities of the following
 confidence regions with nominal coverage of $95\%$: (a) naive selection procedure (1st row),  and (b)  $\mathcal{C}_{0.05,n}$ based on the post double selection estimator (2nd row). The ideal rejection probability should be $5\%$, so ideally we should be seeing a
 flat surface with height $5\%$.}\label{Fig:SimFirst}
\end{figure}

\subsection{Inference on Risk Factors in Childhood Malnutrition}

The purpose of this section is to examine practical usefulness of the new methods and contrast them with the standard post-selection inference (that assumes perfect selection).

We will assess statistical significance of socio-economic and biological factors on children's malnutrition, providing a methodological follow up on the previous studies done by \cite{FKH2011} and \cite{Koenker2011}.  The measure of malnutrition is represented by the child's height, which will be our response variable $y$. The socio-economic and biological factors will be our regressors $x$, which we shall describe in more detail below.  We shall estimate the conditional first decile function of the child's height given the factors (that is, we set $\tau=.1$).   We would like to perform inference on the size of the impact of the various factors on the conditional decile of the child's height.  The problem has material significance, so it is important to conduct statistical inference for this problem responsibly.

The data comes originally from the Demographic and Health Surveys (DHS) conducted regularly in more than 75 countries; we employ the same selected sample of 37,649 as in Koenker (2012). All children in the sample are between the ages of 0 and 5.  The response variable $y$ is the child's height in centimeters.  The regressors $x$ include child's age, breast feeding in months, mothers body-mass index (BMI), mother's age, mother's education, father's education, number of living children in the family, and a large number of categorical variables, with each category coded as binary (zero or one):  child's gender (male or female), twin status (single or twin), the birth order (first, second, third, fourth, or fifth), the mother's employment status (employed or unemployed), mother's religion (Hindu, Muslim, Christian, Sikh, or other), mother's residence (urban or rural), family's wealth (poorest, poorer, middle, richer, richest), electricity (yes or no), radio (yes or no), television (yes or no), bicycle (yes or no), motorcycle (yes or no), and car (yes or no).

Although the number of covariates ($p=30$) is substantial, the sample size ($n=37,649$) is much larger than the number of covariates. Therefore, the dataset is very interesting from a methodological point of view, since it gives us an opportunity to compare various methods for performing inference to an ``ideal" benchmark of standard inference based on the standard quantile regression estimator without any model selection.  This was proven theoretically in  \cite{He:Shao} and in \cite{Belloni:Chern:Fern} under the $p \to \infty, p^3/n \to 0$ regime.  This is also the general option recommended by \cite{koenker:book} and \cite{LeebPotscher2005} in the fixed $p$ regime. Note that this ``ideal" option does not apply in practice when $p$ is relatively large; however it certainly applies in the present example.

We will compare the ``ideal" option with two procedures. First the standard post-selection inference method.  This method performs standard inference on the post-model selection estimator, ``assuming" that the model selection had worked perfectly.
Second the double selection estimator defined as in Algorithm \ref{Alg:2}. (The orthogonal score estimator performs similarly so it is omitted due to space constrains.) The proposed methods do not assume perfect selection, but rather build a protection against (moderate) model selection mistakes.

We now will compare our proposal to the ``ideal" benchmark and to the standard post-selection method.  We report the empirical results in Table \ref{Table:Empitical}. The first column reports results for the ideal option, reporting the estimates and standard errors enclosed in brackets.  The second column reports results for the standard post-selection method, specifically the point estimates resulting from  the post-penalized quantile regression, reporting the standard errors as if there had been no model selection.  The last column report the results for the double selection estimator (point estimate and standard error). Note that the Algorithm \ref{Alg:2} is  applied sequentially to each of the variables. Similarly, in order to provide estimates and confidence intervals for all variables using the naive approach, if the covariate was not selected by the $\ell_1$-penalized quantile regression, it was included in the post-model selection quantile regression for that variable.

\begin{center}{\small \linespread{0.5}  \begin{table}[h!]\hspace{-2cm} \scriptsize   \caption{Empirical Results}\begin{tabular}{l|c|c|c|l|c|c|c}
& quantile    &   Naive post                      &  Double  & & quantile    & Naive post                    &  Double  \\
Variable &  regression  & selection   &  Selection $\check\alpha_\tau$ & Variable &  regression  & selection   &  Selection $\check\alpha_\tau$ \\
   \hline
cage &   0.6456     &   0.6458    &   0.6449 & mreligionhindu &  -0.4351  &   -0.2423    &   -0.5680 \\
   & (0.0030) &  (0.0027)  &  (0.0032)&     & (0.2232) &  (0.1080)  & (0.1771)  \\
mbmi &   0.0603   &   0.0663    &     0.0582  & mreligionmuslim &  -0.3736  &   0.0294    &   -0.5119 \\
   & (0.0159) &   (0.0139)  & (0.0173)&     & (0.2417) &   (0.1438)  & (0.2176)\\
breastfeeding &   0.0691  &     0.0689    &    0.0700  &mreligionother &  -1.1448   &  -0.6977    &   -1.1539  \\
   & (0.0036) &  (0.0038)  &  (0.0044)&     & (0.3296) &   (0.3219)  & (0.3577)\\
mage &   0.0684  &    0.0454    &      0.0685  & mreligionsikh &  -0.5575    &   0.3692    &  -0.3408  \\
   & (0.0090) &   (0.0147)  & (0.0126)&  & (0.2969) &  (0.1897)  &  (0.3889)   \\
medu &   0.1590  &      0.1870    &     0.1566 & mresidencerural &   0.1545 &   0.1085    &   0.1678 \\
   & (0.0136) &  (0.0145)  & (0.0154)&     & (0.0994) &  (0.1363)  & (0.1311) \\
edupartner &   0.0175  &     0.0460    &    0.0348 &wealthpoorer &   0.2732  &   -0.1946    &    0.2648 \\
   & (0.0125) &  (0.0148)  & (0.0143)&     & (0.1761) &  (0.1231)  &  (0.1877)  \\
deadchildren &  -0.0680  &  -0.2121    &  -0.1546  & wealthmiddle &   0.8699  &    0.9197    &    0.9173  \\
   & (0.1124) &  (0.0978)  & (0.1121)&     & (0.1719) &   (0.2236)  & (0.2158) \\
csexfemale &  -1.4625  &    -1.5084    &   -1.5299  &wealthricher &   1.3254    &   0.5754    &  1.4040  \\
   & (0.0948) &  (0.0897)  & (0.1019)&     & (0.2244) &   (0.1408)  & (0.2505)\\
ctwintwin &  -1.7259  &  -1.8683    &   -1.9248  &wealthrichest &   2.0238  &    1.2967    &   2.1133  \\
   & (0.3741) &  (0.2295)  &  (0.7375)&     & (0.2596) &  (0.2263)  &  (0.3318) \\
cbirthorder2 &  -0.7256  &  -0.2230    &   -0.6818  & electricityyes &   0.3866  &    0.7555    &    0.4582 \\
   & (0.1073) &   (0.0983)  &  (0.1337)&     & (0.1581) & (0.1398)  & (0.1577) \\
cbirthorder3 &  -1.2367  &  -0.5751    &    -1.1326  & radioyes &  -0.0385  &    0.1363    &    0.0640  \\
   & (0.1315) &   (0.1423)  & (0.1719)&     & (0.1218) & (0.1214)  & (0.1207)\\
cbirthorder4 &  -1.7455    &  -0.7910    &   -1.5819  & televisionyes &  -0.1633    &  -0.0774    &    -0.0880  \\
   & (0.2244) &   (0.1938)  & (0.2193)&     & (0.1191) &  (0.1234)  &  (0.1386) \\
cbirthorder5 &  -2.4014     &  -1.1747    &   -2.3041  & refrigeratoryes &   0.1544     &   0.2451    &    0.2001  \\
   & (0.1639) & (0.1686)  & (0.2564)&     & (0.1774) &  (0.2081)  &  (0.1891) \\
munemployed &   0.0409  &   0.0077    &     0.0379  & bicycleyes &   0.1438   &   0.1314    &      0.1438  \\
   & (0.1025) &  (0.1077)  &  (0.1124)&     & (0.1048) &  (0.1016)  & (0.1121) \\
motorcycleyes &   0.6104  &    0.5883    &     0.5154  & caryes &   0.2741   &   0.5805    &     0.5470 \\
   & (0.1783) &  (0.1334)  & (0.1625)&    & (0.2058) &   (0.2378)  &  (0.2896) \\
  \end{tabular}\end{table}\label{Table:Empitical}}
\end{center}

What we see is very interesting. First of all, let us compare the ``ideal" option (column 1) and the naive post-selection (column 2).  The Lasso selection method removes 16 out of 30 variables, many of which are highly significant, as judged   by the ``ideal" option. (To judge significance we use normal approximations and critical value of 3, which allows us to maintain $5\%$ significance level after testing up to 50 hypotheses). In particular, we see that the following highly significant variables were dropped by Lasso:  mother's BMI, mother's age, twin status, birth orders one and two, and indicator of the other religion.  The standard post-model selection inference then makes the assumption that these are true zeros, which leads us to misleading conclusions about these effects.  The standard post-selection inference then proceeds to judge the significance of other variables, in some cases deviating sharply and significantly from the ``ideal" benchmark. For example, there is a sharp disagreement on magnitudes of the impact of the birth order variables and the wealth variables (for ``richer" and ``richest" categories). Overall, for the naive post-selection, 8 out of 30 coefficients were more than 3 standard errors away from the coefficients of the ``ideal" option.

We now proceed to comparing our proposed options to the ``ideal" option. We see approximate agreement  in terms of magnitude, signs of coefficients, and in standard errors. In few instances, for example, for the car ownership regressor, the disagreements in magnitude may appear large, but they become insignificant once we account for  the standard errors.

The main conclusion from our study is that the standard/naive post-selection inference can give misleading results, confirming our expectations and confirming predictions of \cite{LeebPotscher2005}. Moreover, the proposed inference procedure is able to deliver inference of high quality, which is very much in agreement with the ``ideal" benchmark.




\begin{center}
{\large\bf SUPPLEMENTARY MATERIAL}
\end{center}

\begin{description}

\item[Supplementary Material.] The supplemental appendix contains the proofs, additional discussions (variants, approximately sparse assumption) and technical results. (pdf)




\end{description}

%

 \bibliographystyle{plain}
\bibliography{mybibSparseQR}

\pagebreak

\setcounter{section}{0}
\setcounter{page}{1}
\begin{center}
{\sc \LARGE Supplementary Appendix for \\ ``Valid Post-Selection Inference in High-dimensional Approximately Sparse Quantile Regression Models"}\end{center}

\bigskip
\begin{quote}
The supplemental appendix contains the proofs of the main results, additional discussions and technical results. Section 1 collects the  notation. Section 2 has additional discussions on variants of the proposed methods, assumptions of approximately sparse functions, minimax efficiency, and the choices of bandwidth and penalty parameters and their implications for the growth of $s$, $p$ and $n$. Section 3 provides new results for $\ell_1$-penalized quantile regression with approximation errors, Lasso with estimated weights under (weaker) aggregated zero mean condition, and for the solution of the zero of the moment condition associated with the orthogonal score function. The proof of the main result of the main text is provided in Section 4. Section 5 collects auxiliary technical inequalities used in the proofs, Section 6 provides the proofs and technical lemmas for $\ell_1$-quantile regression. Section 7 provides proofs and technical lemmas for Lasso with estimated weights. Section 8 provides the proof for the orthogonal moment condition estimation problem. Finally Section 9 provided rates of convergence for the estimates of the conditional density function.
\end{quote}

\spacingset{1.45} 

\section*{Notation}   In what follows, we work with triangular array data $\{ \omega_{i,n} : i=1,\dots,n; n=1,2,3,\dots\}$ where for each $n$, $\{ \omega_{i,n} ; i=1,\dots,n \}$ is defined on
 the probability space $(\Omega, \mathcal{S}, \Pr_n)$. Each  $\omega_{i,n}= (y_{i,n}', z_{i,n}', d_{i,n}')'$
is a vector which are  i.n.i.d., that is, independent across $i$ but not necessarily identically distributed. Hence all parameters that characterize the distribution of  $\{\omega_{i,n} : i=1,\dots,n\}$ are
implicitly indexed by $\Pr_n$ and thus by $n$.  We omit this dependence from the notation for the sake of simplicity.  We use $\En$ to abbreviate the notation $n^{-1}\sum_{i=1}^n$; for example, $\En[f] := \En[f(\omega_{i})] := n^{-1} \sum_{i=1}^n f(\omega_{i})$. We also use the following notation:
$\barEp[f] := \Ep \left [ \En[f] \right] =  \Ep \left [ \En[f(\omega_{i})] \right ]= n^{-1} \sum_{i=1}^n \Ep[f(\omega_{i})]$.
The $\ell_2$-norm is denoted by
$\|\cdot\|$; the $\ell_0$-``norm'' $\|\cdot\|_0$ denotes the number of non-zero components of a vector; and the $\ell_{\infty}$-norm $\| \cdot \|_{\infty}$ denotes the maximal absolute value in the components of a vector.  
Given a vector $\delta \in \RR^p$, and a set of
indices $T \subset \{1,\ldots,p\}$, we denote by $\delta_T \in \RR^p$ the vector in which $\delta_{Tj} = \delta_j$ if $j\in T$, $\delta_{Tj}=0$ if $j \notin T$. We also denote by $\delta^{(k)}$ the vector with $k$ non-zero components corresponding to $k$ of the largest components of $\delta$ in absolute value. We use the notation $(a)_+ = \max\{a,0\}$, $a \vee b = \max\{ a, b\}$, and $a \wedge b = \min\{ a , b \}$. We also use the notation $a \lesssim b$ to denote $a \leqslant c b$ for some constant $c>0$ that does not depend on $n$; and $a\lesssim_P b$ to denote $a=O_P(b)$. For an event $E$, we say that $E$ wp $\to$ 1 when $E$ occurs with probability approaching one as $n$ grows.  Given a $p$-vector $b$, we denote
$\text{support}(b) = \{ j \in \{1,...,p\}: b_j \neq 0\} $. We also use $\rho_\tau(t) = t(\tau - 1\{t\leq 0\})$ and $\varphi_\tau(t_1,t_2) = (\tau - 1\{t_1 \leq t_2\} )$.

Define the minimal and maximal $m$-sparse eigenvalues of a symmetric positive semidefinite matrix $M$ as
\begin{equation}\label{Def:RSE1}
\semin{m}[M] : = \min_{1\leq \|\delta \|_{0} \leqslant m} \frac{  \delta 'M \delta  }{\|\delta\|^2} \ \ \mbox{and} \ \ \hfill
 \semax{m}[M] : = \max_{1\leq \|\delta \|_{0} \leqslant m } \frac{ \delta 'M \delta }{\|\delta\|^2}.
\end{equation}
For notational convenience we write $\tilde x_i = (d_i,x_i')'$, $\semin{m}:= \semin{m}[\En [\tilde x_i\tilde x_i']]$, and $\semax{m}:= \semax{m}[\En [\tilde x_i\tilde x_i']]$.

\section{Additional Discussions}

\subsection{Variants of the Proposed Algorithms}\label{Sec:Variants}

There are several different ways to implement the sequence of steps underlying the  two procedures outlined in Algorithms \ref{Alg:1} and \ref{Alg:2}. The estimation of the control function $g_\tau$ can be done through other regularization methods like $\ell_1$-penalized quantile regression instead of the post-$\ell_1$-penalized quantile regression. The estimation of the error term $v$ in Step 2 can be carried out with  Dantzig selector, square-root Lasso or the associated post-selection method could be used instead of Lasso or post-Lasso. Solving for the zero of the moment condition induced by the orthogonal score function can be substituted by a one-step  correction from the $\ell_1$-penalized quantile regression estimator $\hat\alpha_\tau$, namely, $\check \alpha_\tau = \hat \alpha_\tau + (\En[\hat v_i^2])^{-1} \En[(\tau - 1\{ y_i \leq \hat\alpha_\tau d_i+x_i'\hat\beta_\tau\})\hat v_i]$.

Other variants can be constructing alternative orthogonal score functions. This can be achieved by changing the weights in the equation (\ref{Eq:indirect}) to alternative weights, say $\tilde f_i$, that lead to different errors terms $\tilde v$ that satisfies $\Ep[ \tilde f x \tilde v ] = 0$. Then the orthogonal score function is constructed as $\psi_i(\alpha) = (\tau - 1\{y_i\leq d_i\alpha + g_\tau(z_i)\}) \tilde v_i (\tilde f_i/f_i)$. It turns out that the choice $\tilde f_i = f_i$ minimizes the asymptotic variance of the estimator of $\alpha_\tau$ based upon the empirical analog of (\ref{Eq:Id}),  among all the score functions satisfying (\ref{Eq:OrthoI}).
An example is to set $\tilde f_i = 1$ which would lead to $\tilde v_i=d_i-\Ep[d_i\mid z_i]$. Although such choice leads to a less efficient estimator, the estimation of $\Ep[d_i\mid z_i]$ and $f_i$ can be carried out separably which can lead to weaker regularity conditions.

\subsection{Handling Approximately Sparse Functions}\label{DiscussionNearOrtho}

As discussed in Remark \ref{RemarkExtraMoment}, in order to handle approximately sparse models to represent $g_\tau$ in (\ref{ApproxSparse}) an approximate orthogonality condition is assumed, namely
\begin{equation}
\label{ExtraMomentDisc}
\barEp[ f_i v_i r_{\gtau i} ] = o(n^{-1/2}).
\end{equation}
In the literature such a condition has been (implicitly) used before. For example, (\ref{ExtraMomentDisc}) holds if the function $g_\tau$ is an  exactly sparse linear combination of the covariates so that all the approximation errors are exactly zero, namely, $r_{\gtau i} = 0, i=1,\ldots,n$. An alternative assumption in the literature that implies (\ref{ExtraMomentDisc}) is to have $\Ep[f_id_i\mid z_i]=f_i\{x_i'\theta_{\tau}+r_{\mtau i}\}$, where $\theta_{\tau}$ is sparse and $r_{\mtau i}$ is suitably small, which implies orthogonality to all functions of $z_i$ since we have $\Ep[f_iv_i \mid z_i ] = 0$.

The high-dimensional setting makes the condition (\ref{ExtraMomentDisc}) less restrictive as $p$ grows. Our discussion is based on the assumption that the function $g_\tau$ belongs to a well behaved class of functions. For example, when $g_\tau$ belongs to a Sobolev space $\mathcal{S}(\alpha,L)$ for some $\alpha \geq 1$ and $L>0$ with respect to the basis $\{x_j=P_j(z), j\geq 1\}$. As in \cite{TsybakovBook}, a Sobolev space of functions consists of  functions $g(z)=\sum_{j=1}^\infty\theta_jP_j(z)$ whose Fourier coefficients $\theta$
 satisfy $$\theta \in \Theta(\alpha, L) = \left\{ \theta \in \ell^2(\mathbb{N}): \begin{array}{l}  \sum_{j=1}^\infty|\theta_j| < \infty, \ \ \sum_{j=1}^\infty j^{2\alpha}
\theta_j^2 \leq L^2 \end{array}\right\}.$$
More generally, we can consider functions in a $p$-Rearranged Sobolev space $\mathcal{RS}(\alpha,p,L)$ which allow permutations in the first $p$ components as in \cite{BCW2014pivotal}. Formally, the class of functions $g(z)=\sum_{j=1}^\infty\theta_jP_j(z) $ such that  $$\theta \in \Theta^{R}(\alpha,p,L) = \left\{ \theta \in \ell^2(\mathbb{N}) : \sum_{j=1}^\infty|\theta_j| < \infty, \  \begin{array}{l}\exists \ \mbox{permutation} \ \Upsilon \ \mbox{of} \ \{1,\ldots,p\}: \\ \sum_{j=1}^p j^{2\alpha}\theta_{\Upsilon(j)}^2 + \sum_{j=p+1}^\infty j^{2\alpha}\theta_{j}^2 \leq L^2\end{array}\right\}.$$
It follows that $\mathcal{S}(\alpha,L) \subset \mathcal{RS}(\alpha,p,L)$ and $p$-Rearranged Sobolev space reduces substantially the dependence on the ordering of the basis.

Under mild conditions, it was shown in \cite{BCW2014pivotal} that for functions in $\mathcal{RS}(\alpha,p,L)$ the rate-optimal choice for the size of the  support of the oracle model obeys $s\lesssim n^{1/[2\alpha+1]}$. It follows that
$$\begin{array}{l} \barEp[ r_{\gtau}^2]^{1/2} = \barEp[ \{ \sum_{j> s} \theta_{(j)}P_{(j)}(z_i)\}^2]^{1/2}  \lesssim n^{-\alpha/\{1+2\alpha\}}.\end{array}$$
However, this bound cannot guarantee converge to zero faster than $\sqrt{n}$-rate to potentially imply (\ref{ExtraMomentDisc}). Fortunately, to establish relation (\ref{ExtraMomentDisc}) one can exploit orthogonality with respect all $p$ components of $x_i$. Indeed we have
$$ \begin{array}{rl}
|\barEp[ f_i v_i r_{\gtau i} ]| & = |\barEp[ f_i v_i \{ \sum_{j=s+1}^p \theta_jP_j(z_i) + \sum_{j\geq p+1} \theta_jP_j(z_i)]| \\
& =|  \sum_{j\geq p+1} \barEp[ f_i v_i \theta_jP_j(z_i)]| \leq \sum_{j\geq p+1} |\theta_j| \{\barEp[ f_i^2 v_i^2]\Ep[ P_j^2(z_i)] \}^{1/2}\\
& \leq \{\barEp[ f_i^2 v_i^2]\max_{j\geq p+1}\Ep[ P_j^2(z_i)] \}^{1/2} ( \sum_{j\geq p+1} |\theta_j|^2j^{2\alpha})^{1/2} ( \sum_{j\geq p+1} j^{-2\alpha})^{1/2}\\
& =O( p^{-\alpha+1/2}). \end{array}$$
Therefore, condition (\ref{ExtraMomentDisc}) holds if $n=o(p^{2\alpha-1})$, in particular, for any $\alpha\geq 1$, $n = o(p)$ suffices.

\subsection{Minimax Efficiency}

In this section we make some connections to the (local) minimax efficiency analysis from the semiparametric efficiency analysis.   In this section  for the sake of exposition we assume that $(y_i,x_i,d_i)_{i=1}^n$ are i.i.d., sparse models, $r_{\mtau i}=r_{\gtau i} = 0$, $i=1,\ldots,n$, and the median case ($\tau = .5$). \cite{Lee2003} derives an efficient score function for the partially linear median regression model:
$$
S_i= 2\varphi_\tau(y_i,d_i\alpha_\tau+x_i'\beta_\tau) f_i [d_i -  m_\tau^*(z)],
$$
where  $m_\tau^*(x_i)$ is given by
$$
m^*_\tau(x_i) =  \frac{\Ep[f_i^2 d_i|x_i]}{\Ep[f^2_i |x_i]}.
$$
Using the assumption $m^*_\tau(x_i)=x_i'\theta^*_\tau$ , where $\|\theta^*_\tau\|_0 \leq s \ll n$ is sparse,  we have that
$$
S_i= 2\varphi_\tau(y_i, d_i\alpha_\tau+x_i'\beta_\tau) v_i^*,
$$
where $v_i^* = f_id_i - f_im^*_\tau(x_i)$ would correspond to $v_i$ in (\ref{Eq:indirect}).  It follows
that the estimator based on  $v^*_i$ is actually efficient in the minimax sense (see Theorem
18.4 in \cite{kosorok:book}), and inference about $\alpha_\tau$
based on this estimator provides best minimax power against local alternatives (see Theorem  18.12 in \cite{kosorok:book}).

The claim above is formal as long as,  given a law $Q_n$, the least favorable submodels
are permitted as deviations that lie within the overall model.  Specifically, given a law $Q_n$, we shall need to allow for a certain neighborhood $\mathcal{Q}_n^{\delta}$ of $Q_n$ such that
$Q_n \in \mathcal{Q}_n^{\delta} \subset \mathcal{Q}_n$, where the overall model $\mathcal{Q}_n$
is defined similarly as before, except now permitting heteroscedasticity (or we can keep homoscedasticity
$f_i = f_{\epsilon}$ to maintain formality).    To allow for this we consider
a collection of models indexed by a parameter $t  = (t_1, t_2)$:
\begin{eqnarray}
y_i & = & d_i (\alpha_\tau+ t_1) + x_i'(\beta_\tau + t_2 \theta^*_\tau) + \epsilon_i,    \  \  \  \|t\| \leq \delta, \\
f_i d_i & =&   f_i x_i'\theta_\tau^* + v_i^*,  \ \ \Ep [f_i v_i^*|x_i] =0,
\end{eqnarray}
where $\|\beta_\tau\|_0 \vee \| \theta^*_\tau\|_0 \leq s/2$ and conditions as in Section \ref{Sec:Model} hold.  The case with $t=0$ generates the model $Q_n$; by varying $t$ within $\delta$-ball, we generate models $ \mathcal{Q}_n^{\delta}$, containing the least favorable deviations. By \cite{Lee2003}, the efficient score for the model given above is $S_i$, so we cannot have a better regular estimator than the estimator whose influence function is $J^{-1}S_i$, where $J = \Ep[S_i^2]$.  Since our model $\mathcal{Q}_n$ contains $ \mathcal{Q}_n^{\delta}$, all the formal conclusions about (local minimax) optimality of our estimators hold from theorems cited above (using subsequence arguments to handle models changing with $n$).   Our estimators are regular,
since under $\mathcal{Q}_n^t$ with $t = ( O(1/\sqrt{n}), o(1))$, their first order asymptotics do not change, as a consequence of Theorems in Section \ref{Sec:Model}.  (Though our theorems actually prove more than this.)

\subsection{Choice of Bandwidth $h$ and Penalty Level $\lambda$}\label{Sec:Bandwidth}

The proof of Theorem \ref{theorem:inferenceAlg1prime} provides a detailed analysis for generic choice of bandwidth $h$ and the penalty level $\lambda$ in Step 2 under Condition D. Here we discuss two particular choices for $\lambda$, for $\gamma = 0.05/n$
\begin{center}
\begin{tabular}{ll}
(i)   $\lambda = h^{-1}\sqrt{n}\Phi^{-1}(1-\gamma/2p)$  \ \ \mbox{and} \ \ (ii)  $\lambda = 1.1\sqrt{n}2\Phi^{-1}(1-\gamma/2p)$. \\
\end{tabular}
\end{center}
 The choice (i) for $\lambda$ leads to a sparser estimators by adjusting to the slower rate of convergence of $\hat f_i$, see (\ref{Ratef}). The choice (ii) for $\lambda$ corresponds to the (standard) choice of penalty level in the literature for Lasso. Indeed, we have the following sparsity guarantees for the $\tilde \theta_{\tau}$ under each choice
\begin{center}
\begin{tabular}{ll}
(i)   $\widetilde s_{\theta_\tau} \lesssim s + nh^{2\bar k+2}/\log(pn)$ \ \  \mbox{and} \ \ (ii)  $\widetilde s_{\theta_\tau} \lesssim  h^{-2}s + nh^{2\bar k}/\log(pn)$. \\
\end{tabular}
\end{center}

In addition to the requirements in Condition M, $(K_q^2s^2+s^3)\log^3(pn)\leq \delta_n n$ and $K_q^4s\log(pn)\log^3n\leq \delta_n n$, which are independent of $\lambda$ and $h$, we have that Condition D simplifies to
\begin{center}
\begin{tabular}{ll}
(i) &  $ h^{2\bar k}s\log (pn)\leq \delta_n,  \ \ \   h^{2{\bar k}+1}\sqrt{n}\leq \delta_n, \ \ \  h^{-2}K_q^4s\log (p n) \leq \delta_nn,$ \\
    & $ h^{-2}s^2\log^2 (p n) \leq \delta_nn, \ \ h^{2\bar k+2}K_q^2\log(pn)\log^3 n \leq \delta_n $ \\
(ii)  &  $ h^{2\bar k-2}s\log (pn)\leq \delta_n,  \ \ \   h^{2{\bar k}}\sqrt{n}\leq \delta_n, \ \ \  (h^{-2}\log(pn)\log^3 n+K_q^2)K_q^2s\log (p n) \leq \delta_nn, $ \\
    & $ \{h^{-2}+\log(pn)\}h^{-2}s^2\log (p n) \leq \delta_nn, \ \ h^{2\bar k}K_q^2\log(pn)\log^3 n \leq \delta_n  $ \\
\end{tabular}
\end{center}
For example, using the choice of $\hat f_i$ as in (\ref{Eq:hatfsecond}) so that $\bar k = 4$, we have that the following choice growth conditions suffice for the conditions above:
\begin{center}
\begin{tabular}{ll}
(i) &  $K_q^3 s^3\log^3 (pn) \leq \delta_n n$, $K_q^3 \leq n^{1/3}$ and $h=n^{-1/6}$ \\
(ii) &  $(s + K_x^3 )s^3\log^3 (pn)  \leq \delta_n n$, $K_q^3 \leq n^{1/3}$, and  $h = n^{-1/8}$\\
\end{tabular}
\end{center}

\section{Analysis of the Estimators}\label{Sec:AnalysisHL}

This section contains the main tools used in establishing the main inferential results. The high-level conditions here are intended to be applicable in a variety of settings and they are implied by the regularities conditions provided in the previous sections. The results provided here are of independent interest (e.g. properties of Lasso under estimated weights). We establish the inferential results (\ref{Eq:InferentialAlpha}) and (\ref{Eq:InferenceLn}) in Section  \ref{Sec:EstIVQR} under high level conditions. To verify these high-level conditions we need rates of convergence for the estimated residuals $\hat v$ and the estimated confounding function $\hat g_\tau (z) = x'\hat \beta_\tau$ which are established in sections \ref{Sec:EstLasso} and \ref{Sec:Step1} respectively. The main design condition relies on the restricted eigenvalue proposed in \cite{BickelRitovTsybakov2009}, namely for $\tilde x_i = [ d_i, x_i' ]'$
\begin{equation}\label{Def:RestrEig} \kappa_{\cc} = \inf_{\|\delta_{T^c}\|_1\leq \cc \|\delta_T\|_1} \| \tilde x_i'\delta\|_{2,n}/\|\delta_T\|\end{equation}
where $\cc = (c+1)/(c-1)$ for the slack constant $c>1$, see \cite{BickelRitovTsybakov2009}. When $\cc$ is bounded, it is well known that $\kappa_\cc$ is bounded away from zero provided sparse eigenvalues of order larger than $s$ are well behaved, see \cite{BickelRitovTsybakov2009}.

\subsection{$\ell_1$-Penalized Quantile Regression}\label{Sec:Step1}

In this section for a quantile index $u\in (0,1)$, we consider the equation
\begin{equation}\label{Def:ModelQuantile} \tilde y_i =\tilde x_i' \eta_u + r_{ui} +\epsilon_i, \ \ \mbox{$u$-quantile of} \ \ (\epsilon_i \mid \tilde x_i, r_{ui}) = 0\end{equation}
where  we observe $\{(\tilde y_i,\tilde x_i):i=1,\ldots,n\}$,  which are independent across $i$. To estimate $\eta_u$ we consider the $\ell_1$-penalized $u$-quantile regression estimate
$$ \hat\eta_u \in \arg\min_{\eta} \En[ \rho_u(\tilde y_i - \tilde x_i'\eta)] + \frac{\lambda_u}{n}\|\eta\|_1. $$
and the associated post-model selection estimates. That is, given an estimator $\bar\eta_{uj}$
 \begin{equation}\label{def:l1qr}\widetilde\eta_u \in \arg\min_{\eta} \ \left\{ \ \En[ \rho_u(\tilde y_i  - \tilde x_i'\eta)] \ :\ \eta_{j} = 0 \ \ \mbox{if} \ \ \bar\eta_{uj} = 0\right\}.\end{equation}
We will be typically concerned with $\hat\eta_u$, thresholded versions of the $\ell_1$-penalized quantile regression, defined as $\hat\eta_{uj}^\mu = \hat\eta_{uj}1\{ |\hat\eta_{uj}|\geq \mu/\En[\tilde x_{ij}^2]^{1/2}\}$.

As established in \cite{BC-SparseQR} for sparse models and in \cite{kato} for approximately sparse models, under the event that
\begin{equation}\label{Eq:PenaltyL1QR} \frac{\lambda_u}{n} \geq c \| \En[ (u - 1\{ \tilde y_i \leq \tilde x_i'\eta_u+r_{ui}\})\tilde x_i]\|_\infty\end{equation}
the estimator above achieves good theoretical guarantees under mild design conditions. Although $\eta_u$ is unknown, we can set $\lambda_u$ so that the event in (\ref{Eq:PenaltyL1QR}) holds with high probability. In particular, the pivotal rule proposed in \cite{BC-SparseQR} and generalized in \cite{kato} proposes to set $\lambda_u :=  c n\Lambda_u(1-\gamma\mid \tilde x)$ for $c> 1$ where
\begin{equation}\label{Eq:PenaltyL1QRPivotal} \Lambda_u(1-\gamma\mid \tilde x, r_{u})=(1-\gamma)-\mbox{quantile of} \ \| \En[ (u - 1\{ U_i \leq u\})\tilde x_i]\|_\infty \end{equation}
where $U_i\sim U(0,1)$ are independent random variables conditional on $\tilde x_i$, $i=1,\ldots,n$. This quantity can be easily approximated via simulations. Below we summarize the high level conditions we require.

%

{\bf Condition PQR.} Let $T_u = \supp(\eta_u)$ and normalize $\En[\tilde x_{ij}^2]=1$, $j=1,\ldots,p$. Assume that for some $s\geq 1$, $\|\eta_u\|_0\leq s$, $\|r_{ui}\|_{2,n}\leq C\sqrt{s\log(p)/n}$. Further,
the conditional distribution function of $\epsilon_i$ is absolutely continuous with continuously differentiable density $f_\epsilon(\cdot\mid \tilde x_i, r_{ui})$ such that $0<\underline{f} \leq f_i \leq \sup_t f_{\epsilon_i\mid \tilde x_i, r_{ui}}(t \mid \tilde x_i, r_{ui}) \leq \bar f$, $\sup_{t} |f_{\epsilon_i\mid \tilde x_i, r_{ui}}'( t \mid \tilde x_i, r_{ui})|<\bar f'$ for fixed constants $\underline{f}$, $\bar f$ and $\bar f'$.


Condition PQR is implied by Condition AS. The conditions on the approximation error and near orthogonality conditions follows from choosing a model $\eta_u$ that optimally balance the bias/variance trade-off. The assumption on the conditional density is standard in the quantile regression literature even with fixed $p$ case developed in \cite{koenker:book} or the case of $p$ increasing slower than $n$ studied in \cite{Belloni:Chern:Fern}.


Next we present bounds on the prediction norm of the $\ell_1$-penalized quantile regression estimator.

\begin{lemma}[Estimation Error of $\ell_1$-Penalized Quantile Regression]\label{Theorem:L1QRnp}
Under Condition PQR, setting $\lambda_u \geq cn\Lambda_u(1-\gamma\mid \tilde x)$, we have with probability $1-4\gamma$ for $n$ large enough
$$\|\tilde x_i'(\hat\eta_u - \eta_u)\|_{2,n} \lesssim N:=\frac{\lambda_u\sqrt{s}}{n\kappa_{2\cc}}+ \frac{1}{\kappa_{2\cc}}\sqrt{\frac{s\log(p/\gamma)}{n}}$$
and $\hat\eta_u - \eta_u \in A_u:=\Delta_{2\cc}\cup \{ v  : \|\tilde x_i'v\|_{2,n}=N, \|v\|_1 \leq 8C\cc s\log(p/\gamma)/\lambda_u\}$, provided that
$$ \sup_{\bar\delta\in A_u} \frac{\En[|r_{ui}||\tilde x_i'\bar\delta|^2]}{\En[|\tilde x_i'\bar \delta|^2]} +N\sup_{\bar\delta\in A_u} \frac{\En[|\tilde x_i'\bar\delta|^3]}{\En[|\tilde x_i'\bar \delta|^2]^{3/2}} \to 0.$$
\end{lemma}

Lemma \ref{Theorem:L1QRnp} establishes the rate of convergence in the prediction norm for the $\ell_1$-penalized quantile regression estimator. Exact constants are derived in the proof. The extra growth condition required for identification is mild. For instance we typically have $\lambda_u \sim \sqrt{n\log (np)}$ and for many designs of interest we have $$\inf_{\delta\in\Delta_\cc} \|\tilde x_i'\delta\|_{2,n}^{3}/\En[|\tilde x_i'\delta|^3]$$  bounded away from zero (see \cite{BC-SparseQR}). For more general designs we have {\small $$\inf_{\delta\in A_u} \frac{\|\tilde x_i'\delta\|_{2,n}^{3}}{\En[|\tilde x_i'\delta|^3]}\geq \inf_{\delta\in A_u} \frac{\|\tilde x_i'\delta\|_{2,n}}{ \|\delta\|_1\max_{i\leq n}\|\tilde x_i\|_\infty} \geq \frac{1}{\max_{i\leq n}\|\tilde x_i\|_\infty} \left( \frac{\kappa_{2\cc}}{\sqrt{s}(1+\cc)} \wedge \frac{\lambda_u N}{8C\cc s\log(p/\gamma) }\right) .$$}

\begin{lemma}[Estimation Error of Post-$\ell_1$-Penalized Quantile Regression]\label{Thm:MainTwoStepGeneric}  Assume Condition PQR holds, and that the Post-$\ell_1$-penalized quantile regression is based on an arbitrary vector $\hat\eta_u$. Let $\bar r_u \geq \|r_{ui}\|_{2,n}$, $\widehat s_u \geq |  \supp(\hat\eta_u)|$ and $\hat Q \geq  \En[\rho_u(\tilde y_i - \tilde x_i'\hat\eta_u)]-\En[\rho_u(\tilde y_i-\tilde x_i'\eta_u))]$ hold with probability $1-\gamma$. Then we have for $n$ large enough, with probability $1-\gamma-\varepsilon-o(1)$
$$\|\tilde x_i'(\widetilde \eta_u - \eta_u)\|_{2,n} \lesssim \widetilde N:= \sqrt{ \frac{ (\hat s_u + s)
\log(p/\varepsilon)}{n\semin{\hat s_u+s}}} + \bar f \bar r_u + \hat Q^{1/2}$$
provided that
$$ \sup_{\|\bar\delta\|_0\leq \hat s_u+s} \frac{\En[|r_{ui}||\tilde x_i'\bar\delta|^2]}{\En[|\tilde x_i'\bar \delta|^2]} +\widetilde N\sup_{\|\bar\delta\|_0\leq \hat s_u+s} \frac{\En[|\tilde x_i'\bar\delta|^3]}{\En[|\tilde x_i'\bar \delta|^2]^{3/2}} \to 0.$$
%
\end{lemma}

Lemma \ref{Thm:MainTwoStepGeneric} provides the rate of convergence in the prediction norm for the post model selection estimator despite of possible imperfect model selection. In the current nonparametric setting it is unlikely for the coefficients to exhibit a large separation from zero. The rates rely on the overall quality of the selected model by $\ell_1$-penalized quantile regression and the overall number of components $\hat s_u$. Once again the extra growth condition required for identification is mild.  For more general designs we have $$\inf_{\|\delta\|_0 \leq \hat s_u + s} \frac{\|\tilde x_i'\delta\|_{2,n}^{3}}{\En[|\tilde x_i'\delta|^3]}\geq \inf_{\|\delta\|_0 \leq \hat s_u + s} \frac{\|\tilde x_i'\delta\|_{2,n}}{ \|\delta\|_1\max_{i\leq n}\|\tilde x_i\|_\infty} \geq \frac{\sqrt{\semin{\hat s_u+s}}}{\sqrt{\hat s_u + s}\max_{i\leq n}\|\tilde x_i\|_\infty}.$$

\subsection{Lasso with Estimated Weights}\label{Sec:EstLasso}

In this section we consider the equation
\begin{equation}\label{Def:ModelEstLasso} f_id_i = f_ix_i'\theta_\tau + f_ir_{\mtau i} + v_i, \ \ i=1,\ldots, \  \barEp[ f_iv_ix_i ] = 0\end{equation}
where we observe $\{(d_i,z_i, x_i=X(z_i)):i=1,\ldots,n\}$,  which are independent across $i$.
We do not observe $\{f_i = f_\tau(d_i, z_i)\}_{i=1}^n$ directly and only estimates $\{\hat f_i\}_{i=1}^n$ are available. Importantly, we only  require that $ \barEp[ f_iv_ix_i ] = 0$ and not $\Ep[f_ix_iv_i]=0$ for every $i=1,\ldots,n$. Also, we have that $T_\mtau=\supp(\theta_\tau)$ is unknown but a sparsity condition holds, namely $|T_\mtau| \leq s$. To estimate $\theta_\mtau$ and $v_i$, we compute
\begin{equation}\label{EstLasso}\hat \theta_\tau \in \arg \min_{\theta} \En[ \hat f_i^2 ( d_i - x_i'\theta)^2] + \frac{\lambda}{n}\|\hat \Gamma_\tau\theta\|_1  \ \ \mbox{and set} \ \ \hat v_i = \hat f_i (d_i - x_i'\hat \theta_\tau ), \ \ i=1,\ldots,n, \end{equation}
where $\lambda$ and $\hat \Gamma_\tau$ are the associated penalty level and loadings specified below. A difficulty is to account for the impact of estimated weights $\hat f_i$ while also only using $\barEp[ f_iv_ix_i ] = 0$. 

We will establish bounds on the penalty parameter $\lambda$ so that with high probability
the following regularization event occurs
\begin{equation}\label{Eq:reg} \frac{\lambda}{n} \geq 2c \|\hat \Gamma^{-1}_\tau \En[ f_i x_i v_i ]\|_\infty.
\end{equation} As discussed in \cite{BickelRitovTsybakov2009,BC-PostLASSO,BCW-SqLASSO}, the event above allows to exploit the restricted set condition $\|\hat\theta_{\tau T^c_\mtau}\|_1\leq \tilde \cc \|\hat\theta_{\tau T_\mtau}-\theta_\tau\|_1$ for some $\tilde\cc >1$. Thus rates of convergence for $\hat \theta_\tau$ and $\hat v_i$ defined on (\ref{EstLasso}) can be established based on the restricted eigenvalue $ \kappa_{\tilde\cc}$ defined in (\ref{Def:RestrEig}) with $\tilde x_i = x_i$. 

However, the estimation error in the estimate $\hat f_i$ of $f_i$ could slow the rates of convergence. The following are sufficient high-level conditions. In what follows $\underline c , \bar c, C, \underline{f}, \bar{f}$ are strictly positive constants independent of $n$.

{\bf Condition WL.} {\it  For the model (\ref{Def:ModelEstLasso}) suppose that:\\
 (i) for $s\geq 1$ we have $\|\theta_\tau\|_0\leq s$, $\Phi^{-1}(1-\gamma/2p) \leq \delta_n n^{1/6},$ \\
(ii) $\underline{f} \leq f_i \leq \bar{f}$,  $\underline c \leq {\displaystyle \min_{j\leq p}} \{\barEp[|f_ix_{ij}v_i-\Ep[f_ix_{ij}v_i]|^2]\}^{1/2} \leq {\displaystyle \max_{j\leq p} } \{\barEp[|f_ix_{ij}v_i|^3]\}^{1/3} \leq C,$\\
(iii) with probability $1-\Delta_n$ we have  $\En[ \hat f_i^2 r_{\mtau i}^2] \leq  c_r^2$, 
{\small $$\begin{array}{c}
{\displaystyle \max_{j\leq p}} |(\En-\barEp)[f_i^2x_{ij}^2v_i^2]|+|(\En-\barEp)[\{f_ix_{ij}v_i-\Ep[f_ix_{ij}v_i]\}^2]| \leq \delta_n,  \\   {\displaystyle \max_{j\leq p}} \  \En[(\hat f_i - f_i)^2x_{ij}^2v_i^2] \leq \delta_n,    \ \ \En\left[ \frac{(\hat f_i^2 - f_i^2)^2}{f_i^2}v_i^2\right] + \En\left[ \frac{(\hat f_i^2 - f_i^2)^2}{\hat f_i^2f_i^2}v_i^2\right]\leq c_f^2.
\end{array}$$}
(iv) $\ell \widehat \Gamma_{\tau 0} \leq \widehat \Gamma_{\tau} \leq u \widehat \Gamma_{\tau 0}$, for $\widehat \Gamma_{\tau 0jj}=\{\En[\hat f_i^2 x_{ij}^2v_i^2]\}^{1/2}$,  $1 - \delta_n\leq \ell \leq u  \leq C$ with prob $1-\Delta_n$.}

\begin{remark}Condition WL(i) is a standard condition on the approximation error that yields the optimal bias variance trade-off (see \cite{BC-PostLASSO}) and imposes a growth restriction on $p$ relative to $n$, in particular $\log p = o(n^{1/3})$. Condition WL(ii) imposes conditions on the conditional density function and mild moment conditions which are standard in quantile regression models even with fixed dimensions, see \cite{koenker:book}. Condition WL(iii) requires high-level rates of convergence for the estimate $\hat f_i$. Several primitive moment conditions imply first requirement in Condition WL(iii). These conditions allow the use of self-normalized moderate deviation theory to control heteroscedastic non-Gaussian errors similarly to \cite{BellChenChernHans:nonGauss} where there are no estimated weights. Condition WL(iv) corresponds to the asymptotically valid penalty loading in \cite{BellChenChernHans:nonGauss} which is satisfied by the proposed choice $\widehat \Gamma_\tau$ in (\ref{Parameter:Gamma}).
\end{remark}

Next we present results on the performance of the estimators generated by Lasso with estimated weights. In what follows, $\hat \kappa_\cc$ is defined with $\hat f_i x_i$ instead of $\tilde x_i$ in (\ref{Def:RestrEig}) so that $\hat\kappa_\cc \geq \kappa_\cc \min_{i\leq n} \hat f_i$.

\begin{lemma}[Rates of Convergence for Lasso]\label{Thm:RateEstimatedLasso}
Under Condition WL and setting $\lambda \geq 2c'\sqrt{n}\Phi^{-1}(1-\gamma/2p)$ for $c'>c>1$, we have for $n$ large enough with probability $1-\gamma-o(1)$
$$\begin{array}{l}
 \displaystyle \|\hat f_i x_i'(\hat\theta_\tau - \theta_\tau)\|_{2,n} \leq  2\{ c_f+ c_r\} + \frac{\lambda\sqrt{s}}{n\hat\kappa_{\tilde \cc}}\left(u+\frac{1}{c}\right)\\
 \displaystyle \|\hat\theta_\tau-\theta_\tau\|_1  \leq 2 \frac{\sqrt{s}\{ c_f+ c_r\}}{\hat \kappa_{2\tilde\cc}} + \frac{\lambda s }{n\hat\kappa_{\tilde \cc}\hat \kappa_{2\tilde\cc}}\left(u+\frac{1}{c}\right) + \left(1+\frac{1}{2\tilde\cc}\right)\frac{2c\|\hat\Gamma_{\tau 0}^{-1}\|_\infty}{\ell c-1}\frac{n}{\lambda}\{ c_f+ c_r\}^2\end{array}$$
where $\tilde \cc = \|\hat \Gamma_{\tau 0}^{-1}\|_\infty\|\hat \Gamma_{\tau 0}\|_\infty(uc+1)/(\ell c - 1)$
\end{lemma}

Lemma \ref{Thm:RateEstimatedLasso} above establishes the rate of convergence for Lasso with estimated weights. This automatically leads to bounds on the estimated residuals $\hat v_i$ obtained with Lasso through the identity
\begin{equation}\label{Eq:Identityv}\hat v_i - v_i = (\hat f_i -f_i) \frac{v_i}{f_i} + \hat f_i x_i'(\theta_\tau - \hat\theta_\tau) + \hat f_i r_{\mtau i}.\end{equation}
The Post-Lasso estimator applies the least squares estimator to the model selected by the Lasso estimator (\ref{EstLasso}),
$$ \widetilde\theta_\tau \in \arg\min_{\theta\in \RR^p} \ \left\{ \ \En[\hat f_i^2(d_i - x_i'\theta)^2] \ : \ \theta_j = 0, \ \mbox{if} \ \hat\theta_{\tau j} = 0 \ \right\}, \ \ \mbox{set} \ \tilde v_i = \hat f_i (d_i - x_i'\widetilde\theta_\tau).$$
It aims to remove the bias towards zero induced by the $\ell_1$-penalty function which is used to select components. Sparsity properties of the Lasso estimator $\hat\theta_\tau$ under estimated weights follows similarly to the standard Lasso analysis derived in \cite{BellChenChernHans:nonGauss}. By combining such sparsity properties and the rates in the prediction norm we can establish rates for the post-model selection estimator under estimated weights. The following result summarizes the properties of the Post-Lasso estimator.

\begin{lemma}[Model Selection Properties of Lasso
and Properties of Post-Lasso]\label{Thm:EstLassoPostRate} Suppose that Condition WL holds, and  $\kappa' \leq \semin{\{s + \frac{n^2}{\lambda^2}\{c_f^2 + c_r^2\}\}/\delta_n} \leq \semax{\{s + \frac{n^2}{\lambda^2}\{c_f^2 + c_r^2\}\}/\delta_n} \leq \kappa''$ for some positive and bounded constants $\kappa', \kappa''$.
Then the data-dependent model $\widehat T_\mtau$ selected by the Lasso estimator with $\lambda \geq 2c'\sqrt{n}\Phi^{-1}(1-\gamma/2p)$ for $c'>c>1$, satisfies with probability $1-\gamma-o(1)$:
\begin{equation}\label{eq: sparsity bound}
\|\widetilde \theta_\tau \|_0 = | \widehat T_\mtau | \lesssim s + \frac{n^2}{\lambda^2}\{ c_f^2 +  c_r^2\}
 \end{equation}
Moreover, the corresponding Post-Lasso estimator obeys with probability $1-\gamma-o(1)$
$$ \| x_i'(\widetilde \theta_\tau -\theta_\tau)\|_{2,n} \lesssim_P  c_f +  c_r + \sqrt{\frac{| \widehat T_\mtau |\log (p \vee n)}{n}} + \frac{\lambda\sqrt{s}}{n\kappa_\cc}.$$\end{lemma}

\subsection{Moment Condition based on Orthogonal Score Function}\label{Sec:EstIVQR}

Next we turn to analyze the estimator $\check \alpha_\tau$ obtained based on the orthogonal moment condition. In this section we assume that
$$ L_n(\check \alpha_\tau) \leq \min_{\alpha \in \mathcal{A}_\tau} L_n(\alpha) + \delta_n n^{-1}.$$
This setting is related to the instrumental quantile regression method proposed in \cite{ch:iqrWeakId}. However, in this application we need to account for the estimation of the noise $v$ that acts as the instrument which is known in the setting in \cite{ch:iqrWeakId}. Condition IQR below suffices to make the impact of the estimation of instruments negligible to the first order asymptotics of the estimator $\check \alpha_\tau$. Primitive conditions that imply Condition IQR are provided and discussed in the main text.

Let $\{(y_i,d_i,z_i):i=1,\ldots,n\}$ be independent observations satisfying
\begin{equation}\label{Eq:directRepeat}
\begin{array}{c}y_i = d_i\alpha_\tau + g_\tau(z_i) + \epsilon_i,   \ \ \   \tau\textrm{-quantile}(\epsilon_i\mid d_i, z_i) = 0, \\
f_id_i=f_ix_i'\theta_{0\tau}+v_i, \ \ \ \barEp[ f_ix_iv_i]= 0.
\end{array}
\end{equation}
Letting $\mathcal{D}\times\mathcal{Z}$ denote the domain of the random variables $(d,z)$, for $\tilde h = (\tilde g, \tilde \z)$, where $\tilde g$ is a function of variable $z$, and the instrument $\tilde \z$ is a function that maps $(d,x)\mapsto \tilde \z(d,z)$ we write
$$\begin{array}{rl}
\psi_{\tilde \alpha,\tilde h}(y_i,d_i,z_i) & = \psi_{\tilde \alpha,\tilde g,\tilde \z}(y_i,d_i,z_i) = (\tau - 1\{y_i \leq \tilde g(z_i)+d_i\alpha\})\tilde \z(d_i,z_i)\\
& =  (\tau - 1\{y_i \leq  \tilde g_i+d_i\alpha\})\tilde \z_i.\end{array} $$
We denote $h_0=(g_\tau,\z_0)$ where $\z_{0i}:=v_i=f_i(d_i-x_i'\theta_{0\tau})$. For some sequences $\delta_n\to 0$ and $\Delta_n\to 0$, we let $\overline{\mathcal{F}}$ denote a set of functions such that each element $\tilde h = (\tilde g,\tilde \z)\in \overline{\mathcal{F}}$ satisfies
\begin{equation}\label{Eq:HLfirstestimated}
\begin{array}{c} {\displaystyle \barEp[(1+|\z_{0i}|+|\tilde \z_i-\z_{0i}|)(g_{\tau i}- \tilde g_i)^2] \leq \delta_n n^{-1/2}, \ \ {\displaystyle\barEp[ (\tilde \z_i - \z_{0i})^2] \leq \delta_n,} }\\   \barEp[ |g_{\tau i}- \tilde g_i| |\tilde \z_i - \z_{0i}|] \leq  \delta_n n^{-1/2}, \ \  |\barEp[f_i\z_{0i}\{\tilde g_i-g_{\tau i}\}]|\leq \delta_n n^{-1/2},
\end{array}\end{equation}
and with probability $1-\Delta_n$ we have
\begin{equation}\label{Eq:HLfirstestimatedPartRandom}\displaystyle   \sup_{|\alpha-\alpha_\tau|\leq \delta_n^2, \tilde h \in \overline{\mathcal{F}}}\left|(\En-\barEp)\left[\psi_{\alpha,\tilde h}(y_i,d_i,z_i)-\psi_{\alpha,h_0}(y_i,d_i,z_i)\right]\right| \leq \delta_n \ n^{-1/2}\end{equation}

We assume that the estimated functions $\hat g$ and $\hat \z$ satisfy the following condition.

{\bf Condition IQR.} {\textit Let $\{(y_i,d_i,z_i):i=1,\ldots,n\}$ be random variables independent across $i$ satisfying (\ref{Eq:directRepeat}). Suppose that there are positive constants $0<c\leq C<\infty$ such that:\\
(i) $f_{y_i\mid d_i,z_i}(y\mid d_i,z_i) \leq \bar f$, $f_{y_i\mid d_i,z_i}'(y\mid d_i,z_i) \leq \bar f'$; $c \leq |\barEp[f_id_i\z_{0i}]|$, and $\barEp[ \z_{0i}^4 ] + \barEp[ d_i^4 ] \leq C$;\\
(ii) $\{ \alpha : |\alpha-\alpha_\tau| \leq n^{-1/2}/\delta_n\} \subset \mathcal{A}_\tau$, where $\mathcal{A}_\tau$ is a (possibly random) compact interval; \\
(iii) with probability at least $1-\Delta_n$ the estimated functions $\hat h = (\hat g,\hat \z) \in \overline{\mathcal{F}}$ and
\begin{equation}\label{Eq:HLhatalpha}
|\check \alpha_\tau-\alpha_\tau|\leq \delta_n \ \ \ \ \mbox{and} \ \ \ \ \En[\psi_{\check\alpha_\tau,\hat h}(y_i,d_i,z_i)] |\leq \delta_n\ n^{-1/2}
\end{equation}
(iv)  with probability at least $1-\Delta_n$, the estimated functions $\hat h = (\hat g,\hat \z)$ satisfy $$\|\hat \z_i - \z_{0i}\|_{2,n}  \leq \delta_n \ \ \mbox{and} \ \  \|1\{|\epsilon_i|\leq |d_i(\alpha_\tau-\check\alpha_\tau)+g_{\tau i}-\hat g_i|\}\|_{2,n}\leq \delta_n^2.$$}

\begin{lemma}\label{Thm:EstIVQR}
Under Condition IQR(i,ii,iii) we have
$$ \bar\sigma_n^{-1}\sqrt{n}(\check \alpha_\tau-\alpha_\tau) = \mathbb{U}_n(\tau)+o_P(1), \ \  \mathbb{U}_n(\tau)\rightsquigarrow N(0,1)$$
where $\bar \sigma^2_n = \barEp[f_id_i\z_{0i}]^{-1}\barEp[\tau(1-\tau)\z_{0i}^2]\barEp[f_id_i\z_{0i}]^{-1}$ and
$$\mathbb{U}_n(\tau)=\{\barEp[\psi_{\alpha_\tau,h_0}^2(y_i,d_i,z_i)]\}^{-1/2}\sqrt{n}\En[\psi_{\alpha_\tau,h_0}(y_i,d_i,z_i)].$$
Moreover, IQR(iv) also holds we have $$nL_n(\alpha_\tau) =  \mathbb{U}_n(\tau)^2+o_P(1), \ \ \mathbb{U}_n(\tau)^2\rightsquigarrow \chi^2(1)$$ and the variance estimator is consistent, namely {\small $$\En[\hat f_id_i\hat \z_i]^{-1}\En[(\tau-1\{y_i\leq \hat g_i+d_i\check\alpha_\tau\})^2\hat \z_i^2]\En[\hat f_i d_i \hat \z_i]^{-1}\to_P\barEp[f_id_i\z_{0i}]^{-1}\barEp[\tau(1-\tau)\z_{0i}^2]\barEp[f_id_i\z_{0i}]^{-1}.$$}
\end{lemma}

\section{Proofs for Section \ref{Sec:MainResults} of Main Text (Main Result)}

{\bf Proof.}{\bf \ (Proof of Theorem \ref{theorem:inferenceAlg1prime})}
The first-order equivalence between the two estimators follows from establishing the same linear representation for each estimator. In Part 1 of the proof we consider the orthogonal score estimator. In Part 2 we consider the double selection estimator.

{\rm Part 1. Orthogonal score estimator.} We will verify Condition IQR and the result follows by Lemma \ref{Thm:EstIVQR} applied with $\z_{0i}=v_i=f_i(d_i-x_i'\theta_{0\tau})$ and noting that $1\{y_i \leq d_i\alpha_\tau + g_\tau(z_i)\} = 1\{U_i \leq \tau\}$ for some uniform $(0,1)$ random variable (independent of $\{d_i,z_i\}$) by the definition of the conditional quantile function.

Condition IQR(i) requires conditions on the probability density function that are assumed in Condition AS(iii). The fourth moment conditions are implied by Condition M(i) using $\xi=(1,0')'$ and $\xi=(1,-\theta_{0\tau}')'$, since $\barEp[v_i^4]\leq \bar f^4\barEp[ \{ (d_i,x_i')\xi\}^4] \leq C'(1+\|\theta_{0\tau}\|)^4$, and $\|\theta_{0\tau}\|\leq C$ assumed in Condition AS(i).
Finally, by relation (\ref{Eq:indirect}), namely $\Ep[f_ix_iv_i]=0$, we have $\barEp[f_id_iv_i] = \barEp[v_i^2] \geq \underline f \barEp[ (d_i-x_i'\theta_{0\tau})^2] \geq c\underline f\|(1,\theta_{0\tau}')'\|^2$ by Condition AS(iii) and Condition M(i).

Next we will construct the estimate for the orthogonal score function which are based on post-$\ell_1$-penalized quantile regression and post-Lasso with estimated conditional density function. We will show that with probability $1-o(1)$ the estimated nuisance parameters belong to $\overline{\mathcal{F}}$.

To establish the rates of convergence for $\widetilde\beta_\tau$, the post-$\ell_1$-penalized quantile regression based on the thresholded estimator $\hat\beta_\tau^{\lambda_\tau}$, we proceed to provide rates of convergence for the $\ell_1$-penalized quantile regression estimator $\hat\beta_\tau$. We will apply Lemma \ref{Theorem:L1QRnp} with $\gamma=1/n$. Condition PQR holds by Condition AS with probability $1-o(1)$ using Markov inequality and $\barEp[r_{\gtau i}^2]\leq s/n$.

Since population eigenvalues are bounded above and bounded away from zero by Condition M(i), by Lemma \ref{thm:RV34} (where $\bar \delta_n\to 0$ under $K_q^2 Cs \log^2(1+Cs)  \log (pn) \log n =o(n)$), we have that sparse eigenvalues of order $\ell_n s$ are bounded away from zero and from above with probability $1-o(1)$ for some slowly increasing function $\ell_n$. In turn, the restricted eigenvalue $\kappa_{2\cc}$ is also bounded away from zero for bounded $\cc$ and $n$ sufficiently large. Since $\lambda_\tau \lesssim \sqrt{n\log(p\vee n)}$, we will take $N = C\sqrt{s\log(n\vee p)/n}$ in Lemma \ref{Theorem:L1QRnp}. To establish the side conditions note that
\begin{equation*}
\begin{array}{rl}
{\displaystyle \sup_{\delta \in A_\tau}}\frac{\En[|\tilde x_i'\delta|^3]}{\|\tilde x_i'\delta\|_{2,n}^{3}} & \leq {\displaystyle \sup_{\delta \in A_\tau}}\frac{\max_{i\leq n}\|\tilde x_i\|_\infty\|\delta\|_1}{\|\tilde x_i'\delta\|_{2,n}}\leq {\displaystyle \max_{i\leq n}}\|\tilde x_i\|_\infty \left( \frac{\sqrt{s}(1+2\cc)}{\kappa_{2\cc}} \vee \frac{8C\cc s \log(pn)}{\lambda_u N}\right)  \lesssim_P K_q \sqrt{s\log(pn)}.
\end{array}
\end{equation*}
which implies that $N{\displaystyle \sup_{\delta \in A_\tau}}\frac{\En[|\tilde x_i'\delta|^3]}{\|\tilde x_i'\delta\|_{2,n}^{3}} \to 0$ with probability $1-o(1)$ under $K_q^2s^2\log^2(p\vee n) \leq \delta_n n$. Moreover, the second part of the side condition
\begin{equation*}
\begin{array}{rl} {\displaystyle \sup_{\delta \in A_\tau}}\frac{\En[|r_{\gtau i}| \ |\tilde x_i'\delta|^2]}{\|\tilde x_i'\delta\|^2_{2,n}} & \leq  {\displaystyle \sup_{\delta \in A_\tau}}\frac{\En[|r_{\gtau i}| \ |\tilde x_i'\delta|]\max_{i\leq n}\|\tilde x_i\|_\infty  \|\delta\|_1}{\|\tilde x_i'\delta\|^2_{2,n}} \leq {\displaystyle \sup_{\delta \in A_\tau}}\frac{\|r_{\gtau i}\|_{2,n}\max_{i\leq n}\|\tilde x_i\|_\infty  \|\delta\|_1}{\|\tilde x_i'\delta\|_{2,n}} \\
& \leq \max_{i\leq n}\|\tilde x_i\|_\infty \|r_{\gtau i}\|_{2,n} \left( \frac{\sqrt{s}(1+2\cc)}{\kappa_{2\cc}} \vee \frac{8C\cc s \log(pn)}{\lambda_u N}\right) \\ & \lesssim_P \sqrt{\frac{s\log(pn)}{n}}K_q \sqrt{s\log(pn)}.\\
\end{array}\end{equation*}
Under $K_q^2s^2\log^2(p\vee n) \leq \delta_n n$, the side condition holds with probability $1-o(1)$. Therefore, by Lemma \ref{Theorem:L1QRnp} we have $\|\tilde x_i'\hat \beta_\tau - \beta_\tau \|_{2,n} \lesssim \sqrt{ s \log(pn)/n }$ and  $\|\hat \beta_\tau - \beta_\tau \|_1 \lesssim s\sqrt{\log(pn)/n }$ with probability $1-o(1)$.

With the same probability, by Lemma \ref{Lemma:Bound2nNormSecond} with $\mu = \lambda_\tau/n$, since $\semax{Cs}$ is uniformly bounded with probability $1-o(1)$, we have that the thresholded estimator $\hat\beta^\mu_\tau$ satisfies the following bounds with probability $1-o(1)$:  $\|\hat \beta_\tau^\mu - \beta_\tau \|_1 \lesssim s\sqrt{\log(pn)/n }$, $\|\hat\beta_\tau^\mu\|_0 \lesssim s$ and $\|\tilde x_i'(\hat\beta^\mu_\tau - \beta_\tau)\|_{2,n} \lesssim \sqrt{ s \log(pn)/n }$. We use the support of $\hat\beta^\mu_\tau$ as to construct the refitted estimator $\widetilde\beta_\tau$.

We will apply Lemma \ref{Thm:MainTwoStepGeneric}. By Lemma \ref{Lemma:UpperQ} and the rate of $\hat\beta^\mu_\tau$, we can take $\widehat Q = C s\log(pn)/n$. Since sparse eigenvalues of order $Cs$ are bounded away from zero, we will use $\tilde N = C\sqrt{s\log(pn)/n}$ and $\varepsilon = 1/n$. Therefore,
$\|x_i'(\widetilde\beta_\tau-\beta_\tau)\|_{2,n} \lesssim \sqrt{s\log(np)/n}$ with probability $1-o(1)$ provided the side conditions of the lemma hold. To verify the side conditions note that
\begin{equation*}
\begin{array}{rl}
\widetilde N{\displaystyle \sup_{\|\delta\|_0\leq Cs}}\frac{\En[|\tilde x_i'\delta|^3]}{\|\tilde x_i'\delta\|_{2,n}^{3}} & \leq \widetilde N{\displaystyle \sup_{\|\delta\|_0\leq Cs}}\frac{\max_{i\leq n} \|\tilde x_i\|_\infty \|\delta\|_1}{\|\tilde x_i'\delta\|_{2,n}} \\
&\leq \widetilde N {\displaystyle \sup_{\|\delta\|_0\leq Cs}}\frac{\max_{i\leq n} \|\tilde x_i\|_\infty \sqrt{Cs}\|\delta\|}{\sqrt{\semin{Cs}}\|\delta\|}\lesssim_P \sqrt{\frac{s\log(pn)}{n}}K_q\sqrt{s}\\
\end{array}\end{equation*}
\begin{equation*}
\begin{array}{rl}{\displaystyle \sup_{\|\delta\|_0 \leq Cs}}\frac{\En[|r_{\gtau i}| \ |\tilde x_i'\delta|^2]}{\|\tilde x_i'\delta\|^2_{2,n}} & \leq \|r_{\gtau i}\|_{2,n}\max_{i\leq n}\|\tilde x_i\|_\infty {\displaystyle \sup_{\|\delta\|_0 \leq Cs}}\frac{\|\delta\|_1}{\|\tilde x_i'\delta\|_{2,n}} \\
&\leq  \|r_{\gtau i}\|_{2,n} \max_{i\leq n}\|\tilde x_i\|_\infty \sqrt{\frac{Cs}{\semin{Cs}}}  \lesssim_P \sqrt{\frac{s\log(pn)}{n}}K_q \sqrt{s\log(pn)}\\
\end{array}\end{equation*}
and the side condition holds with probability $1-o(1)$ under $K_q^2 s \log^2(p\vee n) \leq \delta_n n$.

Next we proceed to construct the estimator for $v_i$. Note that by the same arguments above, under Condition D, we have the same rates of convergence for the post-selection (after truncation) estimators $(\widetilde \alpha_u,\widetilde \beta_u)$, $u\in \mathcal{U}$, $\|\widetilde\beta_u\|_0\lesssim C$ and $\|(\widetilde \alpha_u,\widetilde \beta_u)-(\widetilde \alpha_u,\widetilde \beta_u)\|\lesssim \sqrt{s\log (pn)/n}$, to estimate the conditional density function via (\ref{Eq:hatf}) or (\ref{Eq:hatfsecond}). Thus with probability $1-o(1)$
\begin{equation}\label{Ratehatf}
\|f_i - \hat f_i\|_{2,n} \lesssim  \frac{1}{h}\sqrt{\frac{s\log (n\vee p)}{n}}+h^{\bar k} \ \ \mbox{and} \ \ \max_{i\leq n}|\hat f_i-f_i| \lesssim  \delta_n\end{equation}
where $\bar k$ depends on the estimator. (See relation (\ref{EstDensityCoverMain}) below.) Note that the last relation implies that $\max_{i\leq n}\hat f_i \leq 2\bar f$ is automatically satisfied with probability $1-o(1)$ for $n$ large enough. Let $\mathcal{U}$ denote the (finite) set of quantile indices used in the calculation of $\hat f_i$.

Next we verify Condition WL to invoke Lemmas \ref{Thm:RateEstimatedLasso} and \ref{Thm:EstLassoPostRate} with $c_r = C\sqrt{s\log(pn)/n}$ and $c_f= C\left( \{1/h\}\sqrt{s\log (n\vee p)/n}+h^{\bar k}\right)$.
The sparsity condition in Condition WL(i) is implied by Condition AS(ii) and $\Phi^{-1}(1-\gamma/2p) \leq \delta_n n^{1/6}$ is implied by $\log (1/\gamma) \lesssim \log(p\vee n)$ and $\log^{3} p \leq \delta_n n$ from Condition M(iii).
Condition WL(ii) follows from the assumption on the density function in Condition AS(iii) and the moment conditions in Condition M(i).

The first requirement in Condition WL(iii) holds with $c_r^2 \lesssim s\log(pn)/n$ since
$$  \En[\hat f_i^2r_{\mtau i}^2] \leq  \max_{i\leq n}\hat f_i \En[r_{\mtau i}^2] \lesssim \bar f s\log(pn)/n $$
from $\max_{i\leq n} \hat f_i \leq C$ holding with probability $1-o(1)$,  and Markov's inequality under $\barEp[r_{\mtau i}^2]\leq Cs/n$ by Condition AS(ii).

The second requirement, $\max_{j\leq p}|(\En-\barEp)[f_i^2x_{ij}^2v_i^2]|+|(\En-\barEp)[\{f_ix_{ij}v_i-\Ep[f_ix_{ij}v_i]\}^2]| \leq \delta_n $ with probability $1-o(1)$,
  is implied by Lemma \ref{thm:RV34} with $k=1$ and $K=\{\Ep[\max_{i\leq n} \|f_i x_iv_i\|_\infty^2]\}^{1/2} \leq \bar f \{\Ep[\max_{i\leq n} \|(x_i,v_i)\|_\infty^4]\}^{1/2}\leq \bar fK_q^{2}$, under $K_q^4 \log(pn) \log n \leq \delta_n n$.

The third requirement of Condition WL(iii) follows from uniform consistency in (\ref{Ratehatf}) and the second part of Condition WL(ii) since
$$ \max_{j\leq p}\En[(\hat f_i - f_i)^2x_{ij}^2v_i^2] \leq  \max_{i\leq n}\frac{|\hat f_i-f_i|^2}{f_i^2} \left\{ \max_{j\leq p}(\En-\barEp)[f_i^2x_{ij}^2v_i^2] + \max_{j\leq p}\barEp[f_i^2x_{ij}^2v_i^2] \right\} \lesssim \delta_n $$
with probability $1-o(1)$ by (\ref{Ratehatf}), the second requirement, and the bounded fourth moment assumption in Condition M(i).

To show the fourth part of Condition WL(iii), because both $\hat f_i$ and $f_i$ are bounded away from zero and from above with probability $1-o(1)$, uniformly over $u\in \mathcal{U}$ (the finite set of quantile indices used to estimate the density), and $\hat f_i^2-f_i^2 = (\hat f_i-f_i)(\hat f_i+f_i)$, it follows that with probability $1-o(1)$
$$ \En[(\hat f_i^2-f_i^2)^2v_i^2/f_i^2] + \En[(\hat f_i^2-f_i^2)^2v_i^2/\{\hat f_i^2f_i^2\}] \lesssim \En[(\hat f_i-f_i)^2v_i^2/f_i^2].$$ Next let $\delta_{u}=(\widetilde\alpha_u-\alpha_u,\widetilde \beta_{u}' - \beta_{u}')'$ where the estimators satisfy Condition D.
We have that with probability $1-o(1)$\begin{equation}\label{BoundEnvf}
\begin{array}{rl}
 &  \En[(\hat f_i-f_i)^2v_i^2/f_i^2] \lesssim h^{2{\bar k}}\En[v_i^2] + h^{-2} {\displaystyle \sum_{u\in \mathcal{U}}} \En[v_i^2 (\tilde x_i'\delta_{u})^2
+v_i^2r_{u i}^2].\\
\end{array}
\end{equation} The following relations hold for all $u\in \mathcal{U}$
 $$\begin{array}{rl}
 \En[v_i^2r_{u i}^2] & \lesssim_P \barEp[v_i^2r_{u i}^2] \leq \bar f \barEp[(d_i-x_i'\theta_{0\tau})^2r_{u i}^2] \lesssim s/n\\
\En[v_i^2 (\tilde x_i'\delta_{u})^2]& = \barEp[ (x_i'\delta_{u})^2v_i^2] + (\En-\barEp)[v_i^2 (x_i'\delta_{u})^2]\\
& \leq C \|\delta_{u}\|^2 + \|\delta_{u}\|^2 \sup_{\|\delta\|_0\leq \|\delta_u\|_0,\|\delta\|=1}| (\En-\barEp)[ \{v_i\tilde x_i'\delta\}^2]|\\
\end{array}
$$
where we have $\|\delta_{u}\|\lesssim \sqrt{s\log (n\vee p) / n}$  and $\|\delta_{u}\|_0 \lesssim s$ with with probability $1-o(1)$. Then we apply Lemma \ref{thm:RV34} with $X_i=v_i\tilde x_i$. Thus, we can take $K=\{\Ep[\max_{i\leq n}\|X_i\|_\infty^2]\}^{1/2} \leq \{\Ep[\max_{i\leq n}\|(v_i,\tilde x_i')\|_\infty^4]\}^{1/2}\leq K_q^2$, and $\barEp[(\delta'X_i)^2] \leq \barEp[v_i^2(\tilde x_i'\delta)^2]\leq C\|\delta\|^2$ by the fourth moment assumption in Condition M(i) and $\|\theta_{0\tau}\|\leq C$. Therefore,
{\small $$\begin{array}{rl}
{\displaystyle \sup_{\|\delta\|_0\leq Cs, \|\delta\| =1}} \left| (\En-\barEp)\[ \{v_i\tilde x_i'\delta\}^2 \]\right|& \lesssim_P   \left\{ \frac{K_q^2 s \log^{3}n \log (p\vee n)}{n} + \sqrt{\frac{K_q^2 s \log^{3}n \log (p\vee n)}{n}}\right\} \\
\end{array}$$}
Under $K_q^4 s \log^{3}n \log (p\vee n) \leq \delta_n n$, with probability $1-o(1)$ we have
$$ c_f^2 \lesssim \frac{s\log (n\vee p)}{h^2n} + h^{2{\bar k}}.$$

Condition WL(iv) pertains to the penalty loadings which are estimated iteratively. In the first iteration we have that the loadings are constant across components as $\varpi := \max_{i\leq n} \hat f_i \max_{i\leq n}\|x_i\|_\infty \{\En[\hat f_i^2 d_i^2]\}^2$. Thus the solution of the optimization problem is the same if we use penalty parameters $\tilde \lambda$ and $\widetilde \Gamma$ defined as $\tilde\lambda := \lambda \varpi/\max_{j\leq p} \{\En[ \hat f_i^2x_{ij}^2v_i^2]\}^{1/2}$ and $\widetilde \Gamma_{jj} =\max_{j\leq p} \{\En[ \hat f_i^2x_{ij}^2v_i^2]\}^{1/2}$. By construction $\widehat\Gamma_{0\tau jj} \leq \widetilde \Gamma_{jj} \leq  C \widehat\Gamma_{0\tau jj}$ for some bounded $C$ with probability $1-o(1)$ as $\widehat\Gamma_{0\tau jj}$ are bounded away from zero and from above with probability $1-o(1)$. Since Condition WL holds for $(\tilde \lambda,\widetilde \Gamma)$, and $\tilde \lambda \lesssim \lambda \max_{i\leq n}\|x_i\|_\infty$,  by Lemma \ref{Thm:RateEstimatedLasso} we have with probability $1-o(1)$ $$\begin{array}{rl}
\displaystyle  \|x_i'(\hat\theta_\tau-\theta_\tau)\|_{2,n} &\displaystyle  \lesssim \frac{1}{h}\sqrt{\frac{s\log (n\vee p)}{n}} + h^{\bar k} + \frac{\lambda \sqrt{s}\max_{i\leq n}\|x_i\|_\infty}{n} 
\end{array}$$
The iterative choice of of penalty loadings satisfies
$$
\begin{array}{rl}
|\En[\hat f_i^2x_{ij}^2v_i^2]^{1/2}- \En[\hat f_i^2x_{ij}^2\hat v_i^2]^{1/2}| & \leq \max_{i\leq n} \hat f_i\|x_i\|_\infty |\En[\{f_i^2(d_i-x_i'\theta_{0\tau})-\hat f_i(d_i-x_i'\hat\theta_\tau)\}^2]^{1/2}|\\
& \leq \max_{i\leq n} \hat f_i\|x_i\|_\infty  \En[ (\hat f_i-f_i)^2v_i^2/f_i^2 ]^{1/2} \\
& + \max_{i\leq n} \hat f_i\|x_i\|_\infty  \En[ \hat f_i^2 \{x_i'(\hat \theta_\tau - \theta_{0\tau})\}^2]^{1/2} \\
& \lesssim_P \frac{K_q }{h}\sqrt{\frac{s\log (n\vee p)}{n}} + K_qh^{\bar k} + \frac{\lambda K_q^2\sqrt{s}}{n} \\
\end{array}
$$ uniformly in $j\leq p$.  Thus the iterated penalty loadings are uniformly consistent with probability $1-o(1)$ and also satisfy Condition WL(iv) under $h^{-2}K_q^2s\log(pn)\leq \delta_n n$ and $K_q^4s\log(pn)\leq \delta_nn$. Therefore, in the subsequent iterations, by Lemma \ref{Thm:RateEstimatedLasso} and Lemma \ref{Thm:EstLassoPostRate} we have that the post-Lasso estimator satisfies with probability $1-o(1)$ $$\begin{array}{rl}
\displaystyle \|\widetilde \theta_\tau\|_0 &\displaystyle  \lesssim \frac{n^2\{ c_f^2+ c_r^2\}}{\lambda^2} + s \lesssim \widetilde s_{\mtau} := s+\frac{ns\log (n\vee p)}{h^2\lambda^2}+ \left(\frac{nh^{\bar k}}{\lambda}\right)^2 \ \mbox{and} \\
\displaystyle  \|x_i'(\widetilde\theta_\tau-\theta_\tau)\|_{2,n} &\displaystyle  \lesssim \frac{1}{h}\sqrt{\frac{s\log (n\vee p)}{n}} + h^{\bar k} + \frac{\lambda\sqrt{s}}{n}\end{array}$$
where we used that $\semax{\widetilde s_\mtau/\delta_n}\leq C$ implied by Condition D and Lemma \ref{thm:RV34}, and that $\lambda \geq \sqrt{n}\Phi^{-1}(1-\gamma/2p)\sim \sqrt{n\log (pn)}$ so that
$ \sqrt{\widetilde s_\mtau \log (pn)/n} \lesssim (1/h)\sqrt{s\log (pn)/n} + h^{\bar k}.$

Next we construct a class of functions that satisfies the conditions required for $\overline{\mathcal{F}}$ that is used in Lemma \ref{Thm:EstIVQR}. Define the following class of functions.
\begin{equation}\label{DefClassesFunctions} \begin{array}{rl}
\mathcal{M}& =\{ x_i'\theta : \|\theta-\theta_{\tau0}\| \leq C\{\frac{1}{h}\sqrt{s\log (pn)/n} + h^{\bar k}+\frac{\lambda\sqrt{s}}{n}\}, \|\theta\|_0\leq C\widetilde s_\mtau  \}\\
\mathcal{G} & = \{ x_i'\beta : \|\beta\|_0\leq Cs, \ \|\beta - \beta_\tau\| \leq C\sqrt{s\log (pn)/n} \} \\
\mathcal{J} & =  \{ \tilde f_i \leq 2\bar f \ : \|\tilde \eta_u\|_0\leq Cs, \|\tilde \eta_u-\eta_u\| \leq C\sqrt{s\log (pn)/n}, u \in \mathcal{U} \} \\
\end{array}\end{equation}
where $\tilde f_i = f(d_i,z_i, \{\tilde \eta_u : u\in\mathcal{U}\}), \tilde{\tilde{f}}_i = f(d_i,z_i, \{\tilde{\tilde{\eta}}_u : u\in\mathcal{U}\})  \in \mathcal{J}$ are functions that satisfy \begin{equation}\label{EstDensityCoverMain}|\tilde f_i - f_i| \leq \frac{4\bar f}{h}  \sum_{u\in \mathcal{U}} |\tilde x_i'(\eta_u - \tilde \eta_u) + r_{ui}|+ 4\bar f h^{\bar k} \ \mbox{and} \ |\tilde f_i - \tilde{\tilde{f}}_i| \leq \frac{4\bar f}{h}  \sum_{u\in \mathcal{U}} |\tilde x_i'(\tilde \eta_u - \tilde{\tilde{\eta_u}})| \end{equation} In particualr, taking $\tilde f_i := \hat f_i \wedge 2\bar f$ where $\hat f_i$ is defined in (\ref{Eq:hatf}) or (\ref{Eq:hatfsecond}). (Due to uniform consistency (\ref{Ratehatf}) the minimum with $2\bar f$ is not binding for $n$ large.)
Therefore we have
\begin{equation}\label{Eq:ExpDensity} \barEp[ (\tilde f_i - f_i)^2 ] \lesssim  h^{2\bar k} + (1/h^2)\sum_{u\in \mathcal{U}} \barEp[ |\tilde x_i'(\eta_u-\tilde\eta_u)|^2] + \barEp[r_{ui}^2] \lesssim h^{2\bar k} +  h^{-2}s\log(pn)/n.\end{equation}
We define the function class $\overline{\mathcal{F}}$ as
$$ \overline{\mathcal{F}} = \{ ( \tilde g, \tilde v  := \tilde f ( d- \tilde m ) ) : \tilde g \in \mathcal{G}, \tilde m \in \mathcal{M}, \tilde f \in \mathcal{J}\}$$
The rates of convergence and sparsity guarantees for $\widetilde\beta_\tau$ and $\widetilde\theta_\tau$, and Condition D implies that the estimates of the nuisance parameters $\hat g_i := x_i'\widetilde\beta_\tau$ and $\hat v_i=\hat f_i(d_i-x_i'\widetilde\theta_{\tau})$ belongs to the proposed $ \overline{\mathcal{F}}$ with probability $1-o(1)$.

We will proceed to verify relations (\ref{Eq:HLfirstestimated}) and (\ref{Eq:HLfirstestimatedPartRandom}) hold under Condition AS, M and D. We begin with (\ref{Eq:HLfirstestimated}). We have
$$ \begin{array}{rl}
\barEp[|\tilde g_i-g_{\tau i}|^2] & \leq 2\barEp[\{x_i'(\tilde\beta-\beta_\tau)\}^2]+2\barEp[r_{\tau i}^2]\lesssim s\log(pn)/n\lesssim \delta_nn^{-1/2}\\
\barEp[|\tilde v_i||\tilde g_i-g_{\tau i}|^2] & \leq  2\barEp[|\tilde v_i||x_i'(\tilde\beta-\beta_\tau)|^2]+2\barEp[|\tilde v_i|r_{\tau i}^2]\\
& \leq  2\{\barEp[\tilde v_i^2\{x_i'(\tilde\beta-\beta_\tau)\}^2]\barEp[\{x_i'(\tilde\beta-\beta_\tau)\}^2]\}^{1/2}+2\{\barEp[\tilde v_i^2r_{\tau i}^2]\barEp[r_{\tau i}^2]\}^{1/2}\\
& \lesssim \{s\log(pn)/n\}^{1/2}\bar f \{ \barEp[ (d-x_i'\tilde\theta)^2\{x_i'(\tilde\beta-\beta_\tau)\}^2]+\barEp[(d-x_i'\tilde\theta)^2r_{\tau i}^2]\}^{1/2} \\
& \lesssim s\log(pn)/n \lesssim \delta_n n^{-1/2}. \end{array} $$
since $\|(1,-\theta)\|\leq C$ and $f_i\vee \tilde f_i \leq 2\bar f$, $\barEp[(\tilde x_i'\xi)^4]\leq C\|\xi\|^4$, $\barEp[(\tilde x_i'\xi)^2r_{\tau i}^2]\leq C\|\xi\|^2\barEp[r_{\tau i}^2]$ and $\barEp[r_{\tau i}^2]\lesssim s/n$ which hold by Conditions AS and M.
Similarly we have
$$ \begin{array}{rl}
\barEp[|\tilde v_i-v_i|\{\tilde g_i-g_{\tau i}\}^2] & \leq  \barEp[|\tilde f_i - f_i||d_i-x_i'\theta_{\tau0}|\{\tilde g_i-g_{\tau i}\}^2]+\barEp[\tilde f_i|x_i'(\tilde\theta-\theta_{\tau0})|\{\tilde g_i-g_{\tau i}\}^2]\\
& \lesssim s\log(pn)/n \lesssim \delta_n n^{-1/2}. \end{array} $$
as $|\tilde f_i - f_i| \leq 2\bar f$. Moreover we have that
$$ \begin{array}{rl} \barEp[ (\tilde v_i - v_i)^2 ] & \leq 2\barEp[ (\tilde f_i - f_i)^2 (d_i-x_i\theta_{\tau0})^2 + \tilde f_i^2 \{ x_i'(\tilde \theta - \theta_{\tau0})\}^2 ] \\
& \leq 4\bar f\barEp[ (\tilde f_i - f_i)^2 ]^{1/2} \barEp[(d_i-x_i'\theta_{\tau0})^4]^{1/2}+2\bar f^2 \barEp[\{ x_i'(\tilde \theta - \theta_{\tau0})\}^2 ]\\
& \lesssim \frac{1}{h}\sqrt{s\log(pn)/n} + h^{\bar k} + \lambda\sqrt{s}/n \lesssim \delta_n \end{array} $$
under (\ref{Eq:ExpDensity}), $h^{-2}s\log(pn) \leq \delta_n^2 n$, $h \leq \delta_n$, and $\lambda\sqrt{s} \leq \delta_n n$.
The next relation follows from
 $$ \begin{array}{rl}
\barEp[|\tilde v_i-v_i||\tilde g_i-g_{\tau i}|] & \leq  \barEp[|\tilde f_i - f_i||d_i-x_i'\theta_{\tau0}||\tilde g_i-g_{\tau i}|]+\barEp[\tilde f_i|x_i'(\tilde\theta-\theta_{\tau0})||\tilde g_i-g_{\tau i}|]\\
& \leq \barEp[|\tilde f_i - f_i|^2]^{1/2}\barEp[|d_i-x_i'\theta_{\tau0}|^2|\tilde g_i-g_{\tau i}|^2]^{1/2}\\
& + \bar f\barEp[|x_i'(\tilde\theta-\theta_{\tau0})|^2]^{1/2}\barEp[|\tilde g_i-g_{\tau i}|^2]^{1/2}\\
& \lesssim \{s\log(pn)/n\}^{1/2}\{ \frac{1}{h}\sqrt{s\log(pn)/n} + h^{\bar k}\} \lesssim \delta_n n^{-1/2}. \end{array} $$
under $h^{-2}s^2\log^2(pn) \leq \delta_n^2 n$ and $h^{\bar k}\sqrt{s\log(pn)} \leq \delta_n$.

Finally, since $|\barEp[f_iv_ir_{\gtau i}]|\leq \delta_n n^{-1/2}$ from Condition M and $\barEp[f_iv_ix_i]=0$ from (\ref{Eq:indirect}), we have
$$ |\barEp[f_iv_i\{\tilde g_i - g_{\tau i}\}]| \leq |\barEp[f_iv_ix_i'(\tilde\beta - \beta_\tau)] |+|\barEp[f_iv_ir_{\tau i}]| \leq \delta_n n^{-1/2}. $$

Next we verify relation (\ref{Eq:HLfirstestimatedPartRandom}). By triangle inequality we have
\begin{equation}\label{FinalStepHope} \begin{array}{rl}
\displaystyle\sup_{|\alpha-\alpha_\tau|\leq \delta_n^2, \tilde h \in \overline{\mathcal{F}}} \left| (\En-\barEp)[ (\tau - 1\{y_i\leq d_i\alpha+\tilde g_i\})\tilde v_i -  (\tau - 1\{y_i\leq d_i\alpha+g_{\tau i}\})v_i ]\right| \\
\displaystyle \leq \sup_{|\alpha-\alpha_\tau|\leq \delta_n^2, \tilde h \in \overline{\mathcal{F}}} \left| (\En-\barEp)[ (1\{y_i\leq d_i\alpha+g_{\tau i}\}- 1\{y_i\leq d_i\alpha+\tilde g_i\})v_i] \right| \\+\sup_{|\alpha-\alpha_\tau|\leq \delta_n^2, \tilde h \in \overline{\mathcal{F}}} \left| (\En-\barEp)[ (\tau - 1\{y_i\leq d_i\alpha+\tilde g_i\})(\tilde v_i - v_i)] \right| \\
\end{array}\end{equation}
Consider the first term of the right hand side in (\ref{FinalStepHope}). Note that $\mathcal{F}_1 := \{ (1\{y_i\leq d_i\alpha+g_{\tau i}\} - 1\{y_i\leq d_i\alpha+\tilde g_i\})v_i :  \tilde g \in \mathcal{G}, |\alpha-\alpha_\tau| \leq \delta_n^2 \} \subset \mathcal{F}_{1a}-\mathcal{F}_{1b}$ where $\mathcal{F}_{1a}:=1\{y_i\leq d_i\alpha+g_{\tau i}\}v_i : |\alpha-\alpha_\tau| \leq \delta_n^2 \}$ is the product of a VC class of dimension 1 with the random variable $v$, and $\mathcal{F}_{1b}:=\{1\{y_i\leq d_i\alpha+\tilde g_{i}\}v_i : \tilde g \in \mathcal{G}, |\alpha-\alpha_\tau| \leq \delta_n^2 \}$ is the product of $v$ with the union of $\binom{p}{Cs}$ VC classes of dimension $Cs$. Therefore, its entropy number  satisfies
$N(\epsilon\|F_1\|_{Q,2},\mathcal{F}_1,\|\cdot\|_{Q,2}) \leq (A/\epsilon)^{C's}$ where we can take the envelope $F_1(y,d,x)=2|v|$. Since for any function $m \in \mathcal{F}_1$ we have $$\begin{array}{rl}
\barEp[ m^2] & = \barEp[ (1\{y_i\leq d_i\alpha+g_{\tau i}\} - 1\{y_i\leq d_i\alpha+\tilde g_i\})^2v_i^2] \\
& \leq \bar f \barEp[ |g_{\tau i} - \tilde g_i| v_i^2]\leq \barEp[ |g_{\tau i} - \tilde g_i|]^{1/2} \barEp[v_i^4]\lesssim \{s\log(pn)/n\}^{1/2},\end{array}$$ from the conditional density function being uniformly bounded and the bounded fourth moment assumption in Condition M(i), by Lemma \ref{lemma:CCK} with $\sigma:= C\{s\log(pn)/n\}^{1/4}$, $a=pn$, $\|F_1\|_{P,2} = \barEp[v_i^2]^{1/2} \leq C$, and $\|M\|_{P,2}\leq K_q$ we have with probability $1-o(1)$
$$\sup_{m\in \mathcal{F}_1}|(\En-\barEp)[m]| \lesssim \frac{\{s\log(pn)/n\}^{1/4}}{n^{1/2}}\sqrt{s\log(pn)}+ n^{-1} s K_q \log(pn) \lesssim \delta_n n^{-1/2}$$
under $s^3\log^3(pn) \leq \delta_n^4 n$ and $K_q^2s^2\log^2(pn)\leq \delta_n^2 n$.

Next consider the second term of the right hand side in (\ref{FinalStepHope}). Note that $\mathcal{F}_2 := \{ (\tau - 1\{y_i\leq d_i\alpha+\tilde g_{i}\})(\tilde v_i-v_i) :  (\tilde g,\tilde v) \in \overline{\mathcal{F}}, |\alpha-\alpha_\tau| \leq \delta_n^2 \} \subseteq \mathcal{F}_{2a} \cdot \mathcal{F}_{2b}$ where $\mathcal{F}_{2a} := \{ (\tau -  1\{y_i\leq d_i\alpha+\tilde g_{i}\}): \tilde g \in \mathcal{G}\}$ is a constant minus the union of $\binom{p}{Cs}$ VC classes of dimension $Cs$, and $\mathcal{F}_{2b} := \{ \tilde v_i - v_i : (\tilde g, \tilde v) \in \overline{\mathcal{F}} \}$. Note that standard entropy calculations yield $$ N(\epsilon \|F_{2a}F_{2b}\|_{Q,2}, \mathcal{F}_2 , \|\cdot\|_{Q,2}) \leq N(\mbox{$\frac{\epsilon}{2}$} \|F_{2a}\|_{Q,2}, \mathcal{F}_{2a} , \|\cdot\|_{Q,2}) \ N(\mbox{$\frac{\epsilon}{2}$}\|F_{2b}\|_{Q,2}, \mathcal{F}_{2b} , \|\cdot\|_{Q,2})$$
Furthermore, since $\tilde v_i - v_i = (\tilde f_i-f_i)v_i/f_i + \tilde f_ix_i'(\theta_{\tau 0}-\tilde \theta)$, we have
$$\mathcal{F}_{2b} \subset \mathcal{F}_{2b}' + \mathcal{F}_{2b}'' := ( \mathcal{J} - \{f_i\} ) \cdot \{ v_i/f_i \} + \mathcal{J} \cdot (\{x_i'\theta_{\tau0}\} - \mathcal{M})$$
and $\mathcal{F}_2 \subset \mathcal{F}_{2a} \cdot  \mathcal{F}_{2b}' + \mathcal{F}_{2a} \cdot \mathcal{F}_{2b}''$.
By (\ref{EstDensityCoverMain}), a covering for $\mathcal{J}$ can be constructed based on a covering for $\mathcal{B}_u:=\{ \tilde \eta_u : \|\tilde\eta_u\|_0 \leq Cs, \|\tilde\eta_u - \eta_u\|\leq C\sqrt{s\log(pn)/n}\}$ which is the union of $\binom{p}{Cs}$ sparse balls of dimension $Cs$. It follows that for the envelope $F_J:= 2\bar f \vee K_q^{-1}\|\tilde x_i\|_{\infty}$, we have
$$  N(\epsilon \|F_J\|_{Q,2}, \mathcal{J},\|\cdot\|_{Q,2}) \leq \prod_{u\in \mathcal{U}} N\left( \mbox{$\frac{\epsilon h/4\bar f}{|\mathcal{U}|K_q \sqrt{2Cs}}$}, \mathcal{B}_u,\|\cdot\|\right)$$
For any $m \in \mathcal{F}_{2a}\cdot \mathcal{F}_{2b}'$ we have
 $$\begin{array}{rl}
 \barEp[m^2] & \leq \barEp[(\tilde f_i-f_i)^2v_i^2/f_i^2] \lesssim  h^{-2}\sum_{u\in \mathcal{U}}\barEp[\{ |\tilde x_i'(\tilde \eta_u - \eta_u)|^2 +r_{ui}^2\}v_i^2/f_i^2] + h^{2\bar k} \barEp[v_i^2/f_i^2]\\
 & \lesssim  h^{-2}s\log(pn)/n+h^{2\bar k}.\end{array} $$
 Therefore, by Lemma \ref{lemma:CCK} with $\sigma:= \frac{1}{h}\sqrt{s\log(pn)/n}+h^{\bar k}$, $a=pn$, $\|F_2'\|_{P,2} = \|(2\bar f\vee K_q^{-1}\|\tilde x_i\|_\infty)v_i/f_i\|_{P,2} \leq (1+2\bar f) K_q $, and $\|M\|_{P,2}\leq (1+2\bar f)K_q$ we have with probability $1-o(1)$
$$\sup_{m\in \mathcal{F}_{2a}\cdot \mathcal{F}_{2b}'}|(\En-\barEp)(m)| \lesssim \frac{\frac{1}{h}\sqrt{s\log(pn)/n}+h^{\bar k}}{n^{1/2}}\sqrt{s\log(pn)}+ n^{-1} s K_q \log(pn) \lesssim \delta_n n^{-1/2}$$
under $h^{-2}s^2\log^2(pn) \leq \delta_n^2 n$, $h^{\bar{k}}\sqrt{s\log(pn)}\leq \delta_n$ and $K_q^2s^2\log^2(pn)\leq \delta_n^2 n$.

Moreover, $\mathcal{M}$ is the union of $\binom{p}{C\widetilde s_{\mtau}}$ VC subgraph of dimension $C\widetilde s_{\mtau}$, and for any $m \in \mathcal{F}_{2a}\cdot \mathcal{F}_{2b}''$ we have
$$\begin{array}{rl}
 \barEp[m^2] & \leq \barEp[\tilde f_i^2 |x_i'(\theta_{\tau 0} - \tilde \theta)|^2] \leq 4\bar f \barEp[|x_i'(\theta_{\tau 0} - \tilde \theta)|^2] \lesssim h^{-2}s\log(pn)/n + h^{2\bar k} + (\lambda/n)^2 s.\end{array} $$
Therefore,  by Lemma \ref{lemma:CCK} with $\sigma:= C\{\frac{1}{h}\sqrt{s\log(pn)/n}+h^{\bar k}+\lambda\sqrt{s}/n\}$, $a=pnK_qs$, $\|F_2''\|_{P,2} = \|(2\bar f\vee K_q^{-1}\|\tilde x_i\|_\infty)\|x_i\|_\infty\|_{P,2} \leq (1+2\bar f) K_q$, and $\|M\|_{P,2}\leq (1+2\bar f)K_q$ we have with probability $1-o(1)$
$$\sup_{m\in \mathcal{F}_{2a}\cdot \mathcal{F}_{2b}''}|(\En-\barEp)(m)|  \lesssim \frac{\frac{1}{h}\sqrt{\frac{s\log(pn)}{n}}+h^{\bar k}+ \frac{\lambda\sqrt{s}}{n}}{n^{1/2}}\sqrt{\widetilde s_{\mtau}\log(pn)}+ n^{-1} \widetilde s_{\mtau} K_q \log(pn) \lesssim \delta_n n^{-1/2}$$
under the conditions $h^{-2}s\widetilde s_{\mtau}\log^2(pn) \leq \delta_n^2 n$, $h^{\bar{k}}\sqrt{\widetilde s_{\mtau}\log(pn)}\leq \delta_n$, $\lambda\sqrt{s\widetilde s_{\mtau} \log(pn)}\leq \delta_n n$ and $K_q^2\widetilde s_{\mtau}^2\log^2(pn)\leq \delta_n^2 n$ assumed in Condition D.

Next we verify the second part of Condition IQR(iii), namely (\ref{Eq:HLhatalpha}). Note that since $\mathcal{A}_\tau \subset \{ \alpha : |\alpha - \alpha_\tau | \leq C / \log n + |\widetilde \alpha_\tau - \alpha_\tau| \leq C\sqrt{s\log(pn)/n}\}$, $\check\alpha_\tau \in \mathcal{A}_\tau$ and $s\log(pn) \leq \delta_n^2 n$ implies the required consistency $|\check\alpha_\tau-\alpha_\tau|\leq \delta_n$.   To show the other relation in (\ref{Eq:HLhatalpha}), equivalent to with probability $1-o(1)$
$$| \En[\psi_{\check\alpha_\tau,\hat h}(y_i,d_i,z_i)] |\leq \delta_n^{1/2}\ n^{-1/2},$$
note that for any $\alpha \in \mathcal{A}_\tau$ (since it implies $|\alpha-\alpha_\tau|\leq \delta_n$) and $\hat h \in \overline{\mathcal{F}}$, we have with probability $1-o(1)$
$$\begin{array}{rl}
 \displaystyle | \En[\psi_{\alpha,\hat h}(y_i,d_i,z_i)] | & \lesssim | \En[\psi_{\alpha_\tau, h_0}(y_i,d_i,z_i)] +  \barEp[f_id_iv_i](\alpha-\alpha_\tau)| \\
 & + O( \delta_n |\alpha-\alpha_\tau|\barEp[d_i^2|v_i|]+\delta_n n^{-1/2})\\
 \end{array}$$
from relations (\ref{Def:IdentityHelp}) with $\alpha$ instead of $\check\alpha_\tau$, (\ref{EqGamma20}), (\ref{BoundGamma2first}), and (\ref{BoundGamma2third}). Letting $\alpha^* = \alpha_\tau - \{\barEp[f_id_iv_i]\}^{-1}\En[\psi_{\alpha_\tau, h_0}(y_i,d_i,z_i)]$, we have $| \En[\psi_{\alpha^*,\hat h}(y_i,d_i,z_i)] | = O(\delta_n^{1/2} n^{-1/2})$ with probability $1-o(1)$ since  $|\alpha^* - \alpha_\tau| \lesssim_P n^{-1/2}$.
Thus, with probability $1-o(1)$
$$\begin{array}{rl}
 \displaystyle \frac{\{ \ \En[\psi_{\check\alpha_\tau,\hat h}(y_i,d_i,z_i)] \ \}^2}{\En[\hat v_i^2]}&\leq L_n(\check \alpha_\tau)\displaystyle \leq  \frac{{\displaystyle \min_{\alpha\in \mathcal{A}_\tau}} \{ \ \En[\psi_{\alpha,\hat h}(y_i,d_i,z_i)] \ \}^2}{\tau^2(1-\tau)^2\En[\hat v_i^2]} \lesssim \delta_n n^{-1}\\
 \end{array} $$ as $\En[\hat v_i^2]$ is bounded away from zero with probability $1-o(1)$.


Next we verify Condition IQR(iv). The first condition follows from the uniform consistency of $\hat f_i$ and $\max_{i\leq n}\|x_i\|_\infty \|\tilde \theta_\tau - \theta_{0\tau}\|_1 \lesssim_P K_q s \sqrt{\log(pn)/n}\to 0$ under $K_q^2s^2\log^2(pn)\leq \delta_n n$. The second condition in IQR(iv) also follows since
{\small $$\begin{array}{rl}
\|1\{|\epsilon_i|\leq |d_i(\alpha_\tau-\check\alpha_\tau)+g_{\tau i}-\hat g_i|\}\|_{2,n}^2 & \leq \En[ 1\{|\epsilon_i|\leq |d_i(\alpha_\tau-\check\alpha_\tau)|+ |x_i'(\widetilde\beta_\tau-\beta_\tau)|+|r_{\gtau i}|\}] \\
 & \leq \En[ 1\{|\epsilon_i|\leq 3|d_i(\alpha_\tau-\check\alpha_\tau)|\}] \\
 & + \En[ 1\{|\epsilon_i|\leq 3 |x_i'(\widetilde\beta_\tau-\beta_\tau)|\}]+\En[ 1\{|\epsilon_i|\leq 3|r_{\gtau i}|\}] \\
 & \lesssim_P \bar f K_q s\sqrt{\log (pn)/n}\end{array}$$}
which implies the result under $K_q^2s^2\log(pn)\leq \delta_n n$.


The consistency of $\hat\sigma_{1n}$ follows from $\|\hat v_i-v_i\|_{2,n}\to_P 0$ and the moment conditions. The consistency of $\hat\sigma_{3,n}$ follow from Lemma \ref{Thm:EstIVQR}. Next we show the consistency of $\hat \sigma_{2n}^2=\{ \En[\check f_i^2 (d_i,x_{i\check T}')'(d_i,x_{i\check  T }')]\}^{-1}_{11}$. Because $f_i\geq \underf$, sparse eigenvalues of size $\ell_ns$ are bounded away from zero and from above with probability $1-\Delta_n$, and $\max_{i\leq n} |\hat f_i - f_i| = o_P(1)$ by Condition D, we have $$ \{ \En[\hat f_i^2 (d_i,x_{i\check  T }')'(d_i,x_{i\check  T }')]\}^{-1}_{11} = \{ \En[ f_i^2 (d_i,x_{i\check T }')'(d_i,x_{i\check  T }')]\}^{-1}_{11} + o_P(1).$$
So that $\hat\sigma_{2n}-\tilde \sigma_{2n}\to_P 0$ for
 $$\tilde \sigma_{2n}^2=\{ \En[ f_i^2 (d_i,x_{i\check T }')'(d_i,x_{i\check  T }')]\}^{-1}_{11} = \{ \En[ f_i^2 d_i^2] - \En[ f_i^2 d_ix_{i\check T }']\{ \En[f_i^2 x_{i\check T }x_{i\check  T }']\}^{-1}\En[ f_i^2 x_{i\check  T } d_i] \}^{-1}.$$
Next define $\check\theta_\tau[\check T] = \{ \En[ f_i^2 x_{i\check  T }x_{i\check T}']\}^{-1}\En[ f_i^2 x_{i\check  T } d_i]$ which is the least squares estimator of regressing $f_i d_i$ on $f_ix_{i\check  T }$. Let $\check \theta_\tau$ denote the associated $p$-dimensional vector. By definition $f_ix_i'\theta_\tau = f_id_i-f_ir_{\mtau}-v_i$, so that
$$\begin{array}{rl}
\tilde  \sigma_{2n}^{-2} & =  \En[ f_i^2 d_i^2] - \En[ f_i^2 d_ix_i'\check\theta_\tau] \\
& =  \En[ f_i^2 d_i^2] - \En[f_i d_i f_ix_i'\theta_\tau] - \En[f_i d_i f_ix_i'(\check\theta_\tau-\theta_\tau)] \\
& =\En[ f_id_i v_i] - \En[f_i d_i f_i r_{\mtau i}]- \En[f_i d_i f_ix_i'(\check\theta_\tau-\theta_\tau )]\\
& =\En[v_i^2] + \En[  v_i \{f_id_i-v_i\} ]- \En[f_i d_i f_i r_{\mtau i}]- \En[f_i d_i f_ix_i'(\check\theta-\theta_0)]\\
\end{array}$$
We have that $|\En[  v_i \{f_id_i-v_i\} ]|=|\En[ f_i v_ix_i'\theta_{0\tau} ]|=o_P(\delta_n)$ since $$\barEp[ (v_i f_i x_i'\theta_{0\tau} )^2] \leq  \bar f^2\{\barEp[v_i^4]\barEp[(x_i'\theta_{0\tau})^4]\}^{1/2}\leq C$$ and  $\barEp[ f_i v_ix_i'\theta_{0\tau}]=0$. Moreover, $\En[f_i d_i f_i r_{\mtau i}]\leq \bar f_i^2\|d_i\|_{2,n}\|r_{\mtau i}\|_{2,n}=o_P(\delta_n)$, $|\En[f_i d_i f_ix_i'(\check\theta-\theta_\tau)]| \leq \|d_i\|_{2,n}\|f_ix_i'(\check\theta_\tau-\theta_\tau )\|_{2,n} = o_P(\delta_n)$ since $|\check T|\lesssim \widetilde s_{\mtau} + s$ with probability $1-o(1)$ and $\supp(\hat\theta_\tau)\subset \check T$.

~\\
{\it Part 2. Proof of the Double Selection.} The analysis of $\hat\theta_\tau$ and $\hat \beta_\tau$  are identical to the corresponding analysis for Part 1. Let $\hat T^*_\tau$ denote the set of variables used in the last step, namely $\hat T^*_\tau = \supp(\hat\beta_\tau^{\lambda_\tau})\cup \supp(\hat\theta_\tau)$ where $\hat\beta_\tau^{\lambda_\tau}$ denotes the thresholded estimator. Using the same arguments as in Part 1, we have with probability $1-o(1)$ $$|\hat T^*_\tau| \lesssim \widehat s^*_\tau = s + \frac{ns\log p}{h^2\lambda^2}+\left( \frac{nh^{\bar k}}{\lambda} \right)^2. $$

Next we establish preliminary rates for $\check\eta_\tau := (\check\alpha_\tau,\check\beta_\tau')'$ that solves \begin{equation}\label{Eq:DSopt} \check\eta_\tau \in \arg\min_{\eta} \En[\hat f_i \rho_\tau(y_i-(d_i,x_{i\hat T^*_\tau}')\eta)]\end{equation} where $\hat f_i = \hat f_i(d_i,x_i)\geq 0$  is a positive function of $(d_i,z_i)$. We will apply Lemma \ref{Thm:MainTwoStepGeneric} as the problem above is a (post-selection) refitted quantile regression for $(\tilde y_i,\tilde x_i)=(\hat f_i y_i, \hat f_i(d_i,x_i')')$. Indeed, conditional on $\{d_i,z_i\}_{i=1}^n$, quantile moment condition holds as $$\begin{array}{rl}
\Ep[ (\tau - 1\{ \hat f_i y_i \leq \hat f_i d_i\alpha_\tau + \hat f_i x_i'\beta_\tau + \hat f_i r_{\tau i}\})\hat f_ix_i] & = \Ep[ (\tau - 1\{ y_i \leq d_i\alpha_\tau +  x_i'\beta_\tau + r_{\tau i}\})\hat f_i x_i]\\
& = \barEp[ (\tau - F_{y_i\mid d_i,z_i}(d_i\alpha_\tau +  x_i'\beta_\tau + r_{\tau i}))\hat f_i x_i]\\
& = 0.\end{array}$$
Since $\max_{i\leq n}\hat f_i \wedge \hat f_i^{-1} \lesssim 1$ with probability $1-o(1)$, the required side conditions of Lemma \ref{Thm:MainTwoStepGeneric} hold as in Part 1. We can take $\bar r_\tau = C\sqrt{s\log(pn)/n}$ with probability $1-o(1)$. Finally, to provide the bound $\hat Q$,  let $\eta_\tau = (\alpha_\tau,\beta_\tau')'$ and $\hat\eta_\tau=(\hat\alpha_\tau,\hat\beta^{\lambda_\tau}_\tau \ ')'$ which are $Cs$-sparse vectors. By definition $\supp(\hat\beta_\tau) \subset \hat T^*_\tau$ so that
$$\En[\hat f_i\{\rho_\tau(y_i-(d_i,x_i')\check\eta_\tau) - \rho_\tau(y_i-(d_i,x_i')\eta_\tau)\}] \leq \En[\hat f_i\{\rho_\tau(y_i-(d_i,x_i')\hat\eta_\tau) - \rho_\tau(y_i-(d_i,x_i')\eta_\tau)\}] $$
By Lemma \ref{Lemma:UpperQ} we have with probability $1-o(1)$ $$  \En[\hat f_i\{\rho_\tau(y_i-(d_i,x_i')\hat\eta_\tau) - \rho_\tau(y_i-(d_i,x_i')\eta_\tau)\}] \leq \hat Q:= C\frac{s\log (p n)}{n}.$$
Thus, with probability $1-o(1)$,  Lemma \ref{Thm:MainTwoStepGeneric} implies
$$ \| \hat f_i (d_i,x_i')(\check\eta_\tau - \eta_\tau)\|_{2,n} \lesssim \sqrt{ \frac{(s + \hat s^*_\tau)\log(pn)}{n\semin{Cs+C\hat s^*_\tau}}} $$

Since $s \leq \hat s^*_\tau$ and $1/\semin{Cs+C\hat s^*_\tau}\leq C'$ by Condition D we have $ \|\check\eta_\tau-\eta_\tau\|\lesssim \sqrt{\frac{\widehat s^*_\tau\log p}{n}}$.


Next we construct an orthogonal score function based on the solution of the weighted quantile regression (\ref{Eq:DSopt}).
By the first order condition for $(\check\alpha_\tau,\check\beta_\tau)$ in (\ref{Eq:DSopt}) we have for $s_i \in \partial \rho_\tau( y_i-d_i\check\alpha_\tau -x_i'\check\beta_\tau)$ that
$$
  \En\left[ s_i \hat f_i\binom{ d_{i}}{x_{i \hat T^*_\tau}} \right]  = 0.
$$
By taking linear combination of the selected covariates via $(1,-\widetilde\theta_\tau)$ and defining $\hat v_i = \hat f_i(d_{i}-x_{i \hat T^*_\tau}'\widetilde\theta_\tau)$ we have $\psi_{\alpha,\hat h}(y_i,d_i,z_i) = (\tau - 1\{y_{i} \leq d_i  \alpha +  x_{i}'\check \beta_\tau\})\hat v_i$. Since $s_i = \tau - 1\{y_{i} \leq d_i \check \alpha_\tau +  x_{i}'\check \beta_\tau\}$ if $y_i\neq d_i \check \alpha_\tau +  x_{i}'\check \beta_\tau$,
$$
\begin{array}{rl}| \En[ \psi_{\check\alpha_\tau,\hat h}(y_i,d_i,z_i) ] | & \leq | \En[ s_i \hat v_i] | +  \En[1\{ y_{i} = d_i \check \alpha_\tau +  x_{i}'\check \beta_\tau \} |\hat v_i|]\\
& \leq  \En[1\{ y_{i} = d_i \check \alpha_\tau +  x_{i}'\check \beta_\tau \} |\hat v_i- v_{i}|] +  \En[1\{ y_{i} = d_i \check \alpha_\tau +  x_{i}'\check \beta_\tau \} |v_{i}|].\\
\end{array}$$
Since $\widehat s^*_\tau:= |\widehat T^*_\tau|\lesssim \widetilde s_{\mtau}$ with probability $1-o(1)$, and $y_i$ has no point mass, $ y_{i} = d_i \check \alpha_\tau +  x_{i}'\check \beta_\tau$ for at most $C\widetilde s_{\mtau}$ indices in $\{1,\ldots,n\}$. Therefore,  we have with probability $1-o(1)$
  \begin{equation}\label{eq:DSb1}\En[1\{ y_{i} = d_i \check \alpha_\tau +  x_{i}'\check \beta_\tau \} |v_{i}|] \leq n^{-1} C\widetilde s_{\mtau} \max_{i\leq n} |v_i| \lesssim_P  n^{-1} \widetilde s_{\mtau} K_q\delta_n^{-1/6}\lesssim \delta_n^{1/3} n^{-1/2}.\end{equation}
under $K_q^2 \widetilde s^2_{\mtau} \leq \delta_n n$. Moreover,
\begin{equation}\label{eq:DSb2}  \En[1\{ y_{i} = d_i \check \alpha_\tau +  x_{i}'\check \beta_\tau \} |\hat v_i- v_{i}|] \leq  \sqrt{(1+|\widehat T^*_\tau|)/n}\|\hat v_i- v_{i}\|_{2,n} \lesssim \delta_n n^{-1/2}\end{equation}with probability $1-o(1)$  under
$\sqrt{\widehat s^*_\tau}\|\hat v_i- v_{i}\|_{2,n} \leq \delta_n$ holding with probability $1-o(1)$.
Therefore, the orthogonal score function implicitly created by the double selection estimator $\check\alpha_\tau$ approximately minimizes
$$
\widetilde L_n(\alpha) =  \frac{| \En[ (\tau - 1\{y_i \leq  d_i \alpha + x_i'\check\beta_\tau\})\hat v_i ] |^2}{\En[\{ (\tau - 1\{y_i \leq  d_i \alpha + x_i'\check\beta_\tau\})^2\hat v_i^2 ]},
 $$
and we have $\widetilde L_n(\check\alpha_\tau) \lesssim \delta_n^{1/3} n^{-1}$ with probability $1-o(1)$ by (\ref{eq:DSb1}) and (\ref{eq:DSb2}). Thus the conditions of Lemma \ref{Thm:EstIVQR} hold and the double selection estimator has the stated linear representation.
%

{\bf $\square$}

\section{Auxiliary Inequalities}

In this section we collect auxiliary inequalities that we use in our analysis.

\begin{lemma}\label{Lemma:Bound2nNormSecond}
Consider $\hat\eta_u$ and $\eta_u$ where $\|\eta_u\|_0\leq s$.
Denote by $\hat \eta^{\mu}_u$ the vector obtained by thresholding $\hat\eta_u$ as follows $\hat\eta^\mu_{uj}=\hat \eta_{uj}1\{|\hat \eta_{uj}|\geq \mu / \En[\tilde x_{ij}^2]^{1/2}\}$. We have that
$$\begin{array}{rl}
\|\hat \eta^\mu_u - \eta_u\|_{1} & \leq \|\hat \eta_u - \eta_u \|_{1}+s\mu /\min_{j\leq p}\En[\tilde x_{ij}^2]^{1/2}  \\
|\supp(\hat\eta^\mu)| & \leq s + \|\hat \eta_u - \eta_u \|_{1}\max_{j\leq p}\En[\tilde x_{ij}^2]^{1/2}/\mu\\
\|\tilde x_i'(\hat \eta_u^\mu-\eta_u)\|_{2,n} & \leq  \|\tilde x_i'(\hat \eta_u-\eta_u)\|_{2,n}  + \sqrt{\semax{s}} \left\{ \frac{2\sqrt{s}\mu}{\min_{j\leq p}\En[\tilde x_{ij}^2]^{1/2}} + \frac{\|\hat\eta_u-\eta_u\|_{1}}{\sqrt{s}}\right\}\end{array}$$
where $\semax{m}= \sup_{1\leq \|\theta\|_0\leq m}\|\tilde x_i'\theta\|_{2,n}^2/\|\theta\|^2$. \end{lemma}
{\bf Proof.}{\bf \ (Proof of Lemma \ref{Lemma:Bound2nNormSecond})}
Let $T_u = \supp(\eta_u)$. The first relation follows from the triangle inequality
$$\begin{array}{rl}
 \|\hat \eta^\lambda_u - \eta_u\|_{1} &= \|(\hat \eta^\lambda_u - \eta_u)_{T_u}\|_{1} + \|(\hat \eta^\lambda_u)_{T_u^c}\|_{1} \\
&\leq  \|(\hat \eta^\lambda_u - \hat\eta_u)_{T_u}\|_{1} + \|(\hat \eta_u - \eta_u)_{T_u}\|_{1}+ \|(\hat \eta^\lambda_u)_{T_u^c}\|_{1} \\
& \leq \mu s /\min_{j\leq p}\En[\tilde x_{ij}^2]^{1/2} + \|(\hat \eta_u - \eta_u)_{T_u}\|_{1}+ \|(\hat \eta_u)_{T_u^c}\|_{1}\\
& = \mu s /\min_{j\leq p}\En[\tilde x_{ij}^2]^{1/2} + \|\hat \eta_u - \eta_u\|_{1} \end{array}$$

To show the second result note that  $$\|\hat\eta_u-\eta_u\|_{1} \geq \{|\supp(\hat\eta_u^\mu)|-s\}\mu/\max_{j\leq p}\En[\tilde x_{ij}^2]^{1/2}.$$ Therefore, we have
$|\supp(\hat\eta_u^\lambda)|-s \leq \|\hat\eta_u-\eta_u\|_{1}\max_{j\leq p}\En[\tilde x_{ij}^2]^{1/2}/\mu$ which yields the result.

To show the third bound, we start using the triangle inequality
$$ \|\tilde x_i'(\hat\eta^\mu_u-\eta_u)\|_{2,n} \leq \|\tilde x_i'(\hat\eta^\mu_u-\hat\eta_u)\|_{2,n} + \|\tilde x_i'(\hat\eta_u-\eta_u)\|_{2,n}.$$
Without loss of generality assume that order the components is so that $|(\hat\eta^\mu_u-\hat\eta_u)_j|$ is decreasing. Let $T_1$ be the set of $s$ indices corresponding to the largest values of $|(\hat\eta^\mu_u-\hat\eta_u)_j|$. Similarly define $T_k$ as the set of $s$ indices corresponding to the largest values of $|(\hat\eta^\mu_u-\hat\eta_u)_j|$ outside $\cup_{m=1}^{k-1}T_m$. Therefore, $\hat\eta^\mu_u-\hat\eta_u=\sum_{k=1}^{\ceil{p/s}} (\hat\eta^\mu_u-\hat\eta_u)_{T_k}$. Moreover,
given the monotonicity of the components, $\|(\hat\eta^\mu_u-\hat\eta_u)_{T_k}\|_{2,\varpi} \leq \|(\hat\eta^\mu_u-\hat\eta_u)_{T_{k-1}}\|_1/\sqrt{s}$. Then, we have
$$ \begin{array}{rl}
 \|\tilde x_i'(\hat\eta^\mu_u-\hat\eta_u)\|_{2,n} & = \|\tilde x_i'\sum_{k=1}^{\ceil{p/s}} (\hat\eta^\mu_u-\hat\eta_u)_{T_k}\|_{2,n}\\
& \leq \|\tilde x_i'(\hat\eta^\mu_u-\hat\eta_u)_{T_1}\|_{2,n} + \sum_{k\geq 2}\|\tilde x_i'(\hat\eta^\mu_u-\hat\eta_u)_{T_k}\|_{2,n}\\
 & \leq \sqrt{\semax{s}}\|(\hat\eta^\mu_u-\hat\eta_u)_{T_1}\|+\sqrt{\semax{s}}\sum_{k\geq 2}\|(\hat\eta^\mu_u-\hat\eta_u)_{T_k}\|\\
&  \leq  \sqrt{\semax{s}}\mu\sqrt{s}/\min_{j\leq p}\En[\tilde x_{ij}^2]^{1/2} + \sqrt{\semax{s}}\sum_{k\geq 1}\|(\hat\eta^\mu_u-\hat\eta_u)_{T_k}\|_{1}/\sqrt{s}\\
& =  \sqrt{\semax{s}}\mu\sqrt{s}/\min_{j\leq p}\En[\tilde x_{ij}^2]^{1/2} + \sqrt{\semax{s}} \|\hat\eta^\mu_u-\hat\eta_u\|_{1}/\sqrt{s} \\
&  \leq  \sqrt{\semax{s}} \{ 2\sqrt{s}\mu/\min_{j\leq p}\En[\tilde x_{ij}^2]^{1/2} + 2\|\hat\eta_u-\eta_u\|_{1}/\sqrt{s}\} \\ \end{array}$$

where the last inequality follows from the first result and the triangle inequality.

{\bf $\square$}

The following result follows from  Theorem 7.4 of \cite{delapena} and the union bound.
\begin{lemma}[Moderate Deviation Inequality for Maximum of a Vector]\label{Lemma:SNMD} Suppose that
$$ \mathcal{S}_{j} =  \frac{\sum_{i=1}^n U_{ij}}{\sqrt{ \sum_{i=1}^n U^2_{ij}}},$$
where $U_{ij}$ are independent variables across $i$ with mean zero.  We have that
$$
\Pr \left( \max_{1 \leq j\leq p }|\mathcal{S}_{j}|  >  \Phi^{-1}(1- \gamma/2p)  \right) \leq \gamma \(1 + \frac{A}{\ell^3_n}\),
$$
where $A$ is an absolute constant, provided that for $\ell_n>0$
$$
0 \leq \Phi^{-1}(1- \gamma/(2p))  \leq \frac{n^{1/6}}{\ell_n} \min_{1\leq j \leq p} M[U_j]-1, \ \  \ M[U_j] := \frac{\left( \frac{1}{n} \sum_{i=1}^n E U_{ij}^2\right)^{1/2}}{\left(\frac{1}{n} \sum_{i=1}^n E|U_{ij}^3| \right)^{1/3}}.
$$
\end{lemma}

\begin{lemma}\label{thm:RV34}
Let $X_i$, $i=1,\ldots,n$, be independent random vectors in $\RR^p$. Let $$\bar\delta_n:= 2\left( \bar C K \sqrt{k} \log(1+k) \sqrt{\log (p\vee n)} \sqrt{\log n}  \right)/\sqrt{n},$$ where $K\geq \{\Ep[ \max_{1\leq i\leq n}\|X_i\|_\infty^2]\}^{1/2}$ and $\bar C$ is a universal constant. Then we have
$$ \Ep\left[ \sup_{\|\alpha\|_0\leq k, \|\alpha\| =1} \left| \En\[ (\alpha'X_i)^2 - \Ep[(\alpha'X_i)^2] \]\right|\right] \leq \bar\delta_n^2 + \bar\delta_n \sup_{\|\alpha\|_0\leq k, \|\alpha\| =1} \sqrt{\barEp[(\alpha'X_i)^2]}. $$
\end{lemma}
{\bf Proof.}
It follows from Theorem 3.6 of \cite{RudelsonVershynin2008}, see \cite{BelloniChernozhukovHansen2011} for details.
{\bf $\square$}

We will also use the following result of \cite{chernozhukov2012gaussian}.
\begin{lemma}[Maximal Inequality]
\label{lemma:CCK}  Work with the setup above.  Suppose that $F\geq \sup_{f \in \mathcal{F}}|f|$ is a measurable envelope for $\mathcal{F}$
with $\| F\|_{P,q} < \infty$ for some $q \geq 2$.  Let $M = \max_{i\leq n} F(W_i)$ and $\sigma^{2} > 0$ be any positive constant such that $\sup_{f \in \mathcal{F}}  \| f \|_{P,2}^{2} \leq \sigma^{2} \leq \| F \|_{P,2}^{2}$. Suppose that there exist constants $a \geq e$ and $v \geq 1$ such that
\begin{equation*}
\log \sup_{Q} N(\epsilon \| F \|_{Q,2}, \mathcal{F},  \| \cdot \|_{Q,2}) \leq  v \log (a/\epsilon), \ 0 <  \epsilon \leq 1.
\end{equation*}
Then
\begin{equation*}
\Ep_P [ \sup_{f\in \mathcal{F}} | \Gn(f)| ] \leq K  \left( \sqrt{v\sigma^{2} \log \left ( \frac{a \| F \|_{P,2}}{\sigma} \right ) } + \frac{v\| M \|_{P, 2}}{\sqrt{n}} \log \left ( \frac{a \| F \|_{P,2}}{\sigma} \right ) \right),
\end{equation*}
where $K$ is an absolute constant.  Moreover, for every $t \geq 1$, with probability $> 1-t^{-q/2}$,
\begin{multline*}
\sup_{f\in \mathcal{F}} | \Gn(f)| \leq (1+\alpha) \Ep_P [ \sup_{f\in \mathcal{F}} | \Gn(f)| ] + K(q) \Big [ (\sigma + n^{-1/2} \| M \|_{P,q}) \sqrt{t}
+  \alpha^{-1}  n^{-1/2} \| M \|_{P,2}t \Big ], \
\end{multline*}
$\forall \alpha > 0$ where $K(q) > 0$ is a constant depending only on $q$.  In particular, setting $a \geq n$ and $t = \log n$,
with probability $> 1- c(\log n)^{-1}$,
\begin{equation} \label{simple bound}
\sup_{f\in \mathcal{F}} | \Gn(f)| \leq K(q,c) \left ( \sigma \sqrt{v \log \left ( \frac{a \| F \|_{P,2}}{\sigma} \right ) } + \frac{v
 \| M \|_{P,q} } {\sqrt{n}}\log \left ( \frac{a \| F \|_{P,2}}{\sigma} \right ) \right),
\end{equation}
where $  \| M \|_{P,q}  \leq n^{1/q} \| F\|_{P,q}$ and  $K(q,c) > 0$ is a constant depending only on $q$ and $c$.

\end{lemma}

\section{Proofs for Section \ref{Sec:Step1} of Supplementary Material}

{\bf Proof.}{\bf \ (Proof of Lemma \ref{Theorem:L1QRnp})}
Let $\delta = \hat\eta_u - \eta_u$  and define $$\hat R(\eta) = \En[\rho_u(\tilde y_i-\tilde x_i'\eta)]-\En[\rho_u(\tilde y_i-\tilde x_i'\eta_u-r_{ui})]-\En[(u-1\{\tilde y_i\leq \tilde x_i'\eta_u+r_{ui}\})(\tilde x_i'\eta-\tilde x_i'\eta_u-r_{ui})].$$ By Lemma \ref{Lemma:ControlR}, $\hat R(\eta) \geq 0$, $\barEp[\hat R(\eta_u)] \leq \bar f \|r_{ui}\|_{2,n}^2/2$ and with probability at least $1-\gamma$, $\hat R(\eta_u) \leq \bar R_\gamma:= 4\max\{\bar f \|r_{ui}\|_{2,n}^2, \ \|r_{ui}\|_{2,n}\sqrt{\log(8/\gamma)/n}\} \leq 4C s \log(p/\gamma) / n$ from Condition PQR. By definition of $\hat \eta_u$ we have
 \begin{equation}\label{ImpDefbeta}\begin{array}{rl}
 \hat R(\hat\eta_u)-\hat R(\eta_u)+\En[(u-1\{\tilde y_i\leq \tilde x_i'\eta_u+r_{ui}\})\tilde x_i']\delta & = \En[\rho_u(\tilde y_i-\tilde x_i'\hat\eta_u)] - \En[\rho_u(\tilde y_i-\tilde x_i'\eta_u)] \\
  & \leq \frac{\lambda_u}{n}\|\eta_u\|_1-\frac{\lambda_u}{n}\|\hat\eta_u\|_1.\end{array}\end{equation}

Let $N=\sqrt{8\cc\bar R_\gamma/\underf} + \frac{10}{\underf}\left\{\bar f\|r_{ui}\|_{2,n}+\frac{3\cc\lambda_u\sqrt{s}}{n\kappa_{2\cc}}+ \frac{8(1+2\cc)\sqrt{s\log(16p/\gamma)}}{\sqrt{n}\kappa_{2\cc}}+\frac{8\cc \sqrt{n} \bar R_\gamma \sqrt{\log(16p/\gamma)}}{\lambda_u\{s\log(p/\gamma)/n\}^{1/2}}\right\}$ denote the upper bound in the rate of convergence. Note that $N \geq \{s\log(p/\gamma)/n\}^{1/2}$. Suppose that the result is violated, so that $\|\tilde x_i'\delta\|_{2,n} > N$. Then by convexity of the objective function in (\ref{def:l1qr}), there is also a vector $\tilde \delta$ such that $\|\tilde x_i'\tilde \delta\|_{2,n} = N$, and
\begin{equation}\label{ImpDefbetaTilde0}
\begin{array}{rl}
 \En[\rho_u(\tilde y_i-\tilde x_i'(\tilde \delta +\eta_u))] - \En[\rho_u(\tilde y_i-\tilde x_i'\eta_u)]  & \leq \frac{\lambda_u}{n}\|\eta_u\|_1-\frac{\lambda_u}{n}\|\tilde \delta +\eta_u\|_1.\end{array}\end{equation}
Next we will show that with high probability such $\tilde \delta$ cannot exist implying that $\|\tilde x_i'\delta\|_{2,n}\leq  N$.

By the choice of $\lambda_u \geq c\Lambda_u(1-\gamma\mid \tilde x)$ the event $\Omega_1:=\{ \frac{\lambda_u}{n}\geq c\|\En[(u-1\{\tilde y_i\leq \tilde x_i'\eta_u+r_{ui}\}) \tilde x_i]\|_\infty\}$ occurs with probability at least $1-\gamma$. The event $\Omega_2 :=\{\hat R_1(\eta_u) \leq \bar R_\gamma\}$ also holds with probability at least $1-\gamma$. Under $\Omega_1 \cap \Omega_2$, and since $\hat R(\eta)\geq 0$, we have
\begin{equation}\label{ImpDefbetaTilde}
\begin{array}{rl}
-\hat R(\eta_u)-\frac{\lambda_u}{cn}\|\tilde\delta\|_1 & \leq \hat R(\eta_u+\tilde \delta)-\hat R(\eta_u)+\En[(u-1\{\tilde y_i\leq \tilde x_i'\eta_u+r_{ui}\})\tilde x_i']\tilde \delta \\
  & = \En[\rho_u(\tilde y_i-\tilde x_i'(\tilde \delta +\eta_u))] - \En[\rho_u(\tilde y_i-\tilde x_i'\eta_u)] \\
 & \leq \frac{\lambda_u}{n}\|\eta_u\|_1-\frac{\lambda_u}{n}\|\tilde \delta +\eta_u\|_1\end{array}\end{equation}
so that for $\cc = (c+1)/(c-1)$
$$ \|\tilde \delta_{T^c_u}\|_1 \leq \cc \|\tilde \delta_{T_u}\|_1  + \frac{nc}{\lambda_u(c-1)}\hat R(\eta_u).$$
To establish that $\tilde \delta \in A_u:=\Delta_{2\cc}\cup \{ v  : \|\tilde x_i'v\|_{2,n}=N, \|v\|_1 \leq 2\cc n \bar R_\gamma/\lambda_u\}$  we consider two cases. If $\|\tilde \delta_{T^c_u}\|_1 \geq 2\cc\|\tilde \delta_{T_u}\|_1$ we have
$$ \frac{1}{2}\|\tilde \delta_{T^c_u}\|_1 \leq  \frac{nc}{\lambda_u(c-1)}\hat R(\eta_u)$$
and consequentially $$ \|\tilde\delta\|_1\leq \{1+1/(2c)\}\|\tilde \delta_{T^c_u}\|_1 \leq  \frac{2n\cc}{\lambda_u
}\hat R(\eta_u).$$
Otherwise $\|\tilde \delta_{T^c_u}\|_1 \leq 2\cc\|\tilde \delta_{T_u}\|_1$, and we have
$$\|\tilde\delta\|_1 \leq (1+2\cc)\|\tilde\delta_{T_u}\|_1 \leq (1+2\cc)\sqrt{s}\|\tilde x_i'\tilde\delta\|_{2,n}/\kappa_{2\cc}.$$
Thus with probability $1-2\gamma$,  $\tilde \delta \in A_u$.

Therefore, under $\Omega_1 \cap \Omega_2$, from (\ref{ImpDefbetaTilde0}), applying  Lemma \ref{Lemma:EmpProc:Normalized} (part (1) and (3) to cover $\tilde \delta \in A_u$), for $ \|\tilde x_i'\tilde \delta\|_{2,n} = N$ with probability at least $1-4\gamma$ we have
{\small $$\begin{array}{rl}
& \barEp[\rho_u(\tilde y_i-\tilde x_i'(\tilde \delta + \eta_u))] - \barEp[\rho_u(\tilde y_i-\tilde x_i'\eta_u)]\\
 &  \leq  \frac{\lambda_u}{n}\|\tilde \delta\|_1 + \frac{\|\tilde x_i'\tilde \delta\|_{2,n}}{\sqrt{n}}\left\{ \frac{8(1+2\cc)\sqrt{s}}{\kappa_{2\cc}}+\frac{8\cc n \bar R_\gamma}{\lambda_u\underline{N}}\right\}\sqrt{\log(16p/\gamma)}\\
 & \leq  2\cc\bar R_\gamma +\|\tilde x_i'\tilde \delta\|_{2,n}\left[\frac{3\cc\lambda_u\sqrt{s}}{n\kappa_{2\cc}}+ \left\{ \frac{8(1+2\cc)\sqrt{s}}{\kappa_{2\cc}}+\frac{8\cc n \bar R_\gamma}{\lambda_u N}\right\} \frac{\sqrt{\log(16p/\gamma)}}{\sqrt{n}}\right]\\
 \end{array}$$}
where we used the bound for $\|\tilde\delta\|_1 \leq (1+2\cc)\sqrt{s}\|\tilde x_i'\tilde\delta\|_{2,n}/\kappa_{2\cc} + \frac{2n\cc}{\lambda_u}\bar R_\gamma$.

Using Lemma \ref{Lemma:IdentificationSparseQRNP}, since by assumption $ \sup_{\bar\delta\in A_u} \frac{\En[|r_{ui}||\tilde x_i'\bar\delta|^2]}{\En[|\tilde x_i'\bar \delta|^2]} \to 0$,  we have
$$ \barEp[\rho_u(\tilde y_i-\tilde x_i'(\eta_u+\tilde \delta))-\rho_u(\tilde y_i-\tilde x_i'\eta_u)]\geq - \bar f \|r_{ui}\|_{2,n} \|\tilde x_i'\tilde \delta\|_{2,n}  + \frac{\underf\|\tilde x_i'\tilde\delta\|_{2,n}^2}{4}\wedge \bar q_{A_u}\underf \|\tilde x_i'\tilde\delta\|_{2,n}$$

Note that $N < 4 \bar q_{A_u}$ for $n$ sufficiently large by the assumed side condition,
so that the minimum on the right hand side is achieved for the quadratic part. Therefore we have
{\small $$ \frac{\underf\|\tilde x_i'\tilde \delta\|_{2,n}^2}{4} \leq  2\cc\bar R_\gamma + \|\tilde x_i'\tilde \delta\|_{2,n}\left\{\bar f\|r_{ui}\|_{2,n}+\frac{3\cc\lambda_u\sqrt{s}}{n\kappa_{2\cc}}+ \frac{8(1+2\cc)\sqrt{s\log(16p/\gamma)}}{\sqrt{n}\kappa_{2\cc}}+\frac{8\cc \sqrt{n} \bar R_\gamma \sqrt{\log(16p/\gamma)}}{\lambda_u N}\right\}$$}
which implies that
{\small $$\begin{array}{rl}
\|\tilde x_i'\tilde \delta\|_{2,n} & \leq \sqrt{8\cc\bar R_\gamma/\underf} \\
& + \frac{8}{\underf}\left\{\bar f\|r_{ui}\|_{2,n}+\frac{3\cc\lambda_u\sqrt{s}}{n\kappa_{2\cc}}+ \frac{8(1+2\cc)\sqrt{s\log(16p/\gamma)}}{\sqrt{n}\kappa_{2\cc}}+\frac{8\cc \sqrt{n} \bar R_\gamma \sqrt{\log(16p/\gamma)}}{\lambda_u N}\right\}\end{array}$$}
which violates the assumed condition that $\|\tilde x_i'\tilde \delta\|_{2,n} = N$ since $N>\{s\log(p/\gamma)/n\}^{1/2}$.
{\bf $\square$}

{\bf Proof.}{\bf \ (Proof of Lemma \ref{Thm:MainTwoStepGeneric})}
Let $ \ \hat \delta_u = \hat \eta_u - \eta_u$. By optimality of $\widetilde
\eta_u$ in (\ref{def:l1qr}) we have with probability $1-\gamma$
\begin{equation}\label{2step:Rel1a} \begin{array}{rl}
  \En[\rho_u( \tilde y_i-\tilde x_i'\widetilde \eta_u )] - \En[\rho_u(\tilde y_i-\tilde x_i'\eta_u )] & \leq  \En[\rho_u(\tilde y_i-\tilde x_i'\hat \eta_u )] - \En[\rho_u(\tilde y_i-\tilde x_i'\eta_u )] \leq  \hat Q.\end{array}\end{equation}

Let $N= 2\bar{f}\bar r_u +  A_{\varepsilon,n} +
2\hat Q^{1/2}$ denote the upper bound in the rate of convergence where $A_{\varepsilon,n}$ is defined below. Suppose that the result is violated, so that $\|\tilde x_i'(\widetilde\eta_u - \eta_u)\|_{2,n} > N$. Then by convexity of the objective function in (\ref{def:l1qr}), there is also a vector $\widetilde \delta_u$ such that $\|\tilde x_i'\widetilde \delta_u\|_{2,n} =  N$, $\|\widetilde \delta_u\|_0 = \|\widetilde\eta_u - \eta_u\|_0\leq \hat s_u+s$ and
 \begin{equation}\label{Additoinal}
 \En[\rho_u( \tilde y_i-\tilde x_i'(\eta_u+\widetilde \delta_u) )] - \En[\rho_u(\tilde y_i-\tilde x_i'\eta_u )] \leq  \hat Q.\end{equation} Next we will show that with high probability such $\widetilde \delta_u$ cannot exist implying that $\|\tilde x_i'(\widetilde\eta_u - \eta_u)\|_{2,n}\leq  N$ with high probability.

By Lemma \ref{Lemma:EmpProc:Normalized}, with probability at least $1-\varepsilon$,
we have
\begin{equation}\label{2stepRel3a}
\frac{|(\En-\barEp)[\rho_u(\tilde y_i-\tilde x_i'(\eta_u+\widetilde \delta_u) ) - \rho_u(\tilde y_i-\tilde x_i'\eta_u )]|}{ \|\tilde x_i'\widetilde
\delta_u\|_{2,n}} \leq 8 \sqrt{\frac{{(\widehat s_u + s) \log
(16p/\varepsilon)}}{n\semin{\hat s_u+s}}}=:A_{\varepsilon,n}.\end{equation}
Thus combining relations
(\ref{2step:Rel1a}) and (\ref{2stepRel3a}),  we have  $$ \begin{array}{rcl} \displaystyle
\barEp[\rho_u(\tilde y_i-\tilde x_i'(\eta_u+\widetilde \delta_u) )] - \barEp[\rho_u(\tilde y_i-\tilde x_i'\eta_u )] &\leq &
\displaystyle  \|\tilde x_i'\widetilde \delta_u\|_{2,n} A_{\varepsilon,n} +  \hat Q\\
 \end{array}
 $$ with probability at least $1-\varepsilon$.
Invoking the sparse identifiability relation  of Lemma \ref{Lemma:IdentificationSparseQRNP}, with the same probability,  since $\sup_{\|\delta\|_0\leq \hat s_u + s}  \frac{\En[|r_{ui}|\ |\tilde x_i'\theta|^2]}{\En[|\tilde x_i'\theta|^2]}\to 0$ by assumption,  $$ \begin{array}{rcl} \displaystyle
(\underline{f}\|\tilde x_i'\widetilde \delta_u\|_{2,n}^2/4) \wedge \left( \widetilde q_{\hat s_u}\underline{f}\|\tilde x_i'\widetilde \delta_u\|_{2,n} \right) & \leq & \displaystyle \|\tilde x_i'\widetilde \delta_u\|_{2,n}\left\{\bar{f}\|r_{ui}\|_{2,n} +  A_{\varepsilon,n}\right\} +  \hat Q.\\
 \end{array}$$ where $\widetilde q_{\hat s_u} :=  \mbox{$\frac{\underf^{3/2}}{2\bar f'}$} \inf_{\|\delta\|_0\leq \hat s_u + s}  \frac{\|\tilde x_i'\theta\|_{2,n}^3}{\En[|\tilde x_i'\theta|^3]}$.

Under the assumed growth condition, we have $N<4\widetilde q_{\hat s_u}$ for $n$ sufficiently large and the minimum is achieved in the quadratic part. Therefore, for $n$ sufficiently large, we have
$$\|\tilde x_i'\widetilde \delta_u\|_{2,n} \leq \bar{f}\|r_{ui}\|_{2,n} +  A_{\varepsilon,n} + 2\hat Q^{1/2} < N$$

Thus with probability at least $1-\varepsilon-\gamma-o(1)$ we have  $\|\tilde x_i'\widetilde \delta_u\|_{2,n}<  N$ which contradicts its definition. Therefore, $\|\tilde x_i'(\widetilde \eta_u -\eta_u)\|_{2,n}\leq N$ with probability at least $1-\gamma-\varepsilon-o(1)$. 
{\bf $\square$}

\subsection{Technical Lemmas for Quantile Regression}\label{Sec:ProofsTechnicalLemmaL1QR}

\begin{lemma}\label{Lemma:IdentificationSparseQRNP}
For a subset $A\subset \RR^p$ let
$$ \bar q_A = (1/2) \cdot (\underf^{3/2}/\bar{f'}) \cdot \inf_{ \delta \in A} \En\[|\tilde x_i'\delta|^2\]^{3/2}/\En\[|\tilde x_i'\delta|^3\]$$
and assume that for all $\delta \in A$
$$\barEp\[ |r_{ui}|\cdot |\tilde x_i'\delta|^2\]  \leq (\underf/[4\bar f'])\barEp[|\tilde x_i'\delta|^2].$$
Then, we have
$$ \barEp[\rho_u(\tilde y_i-\tilde x_i'(\eta_u + \delta))] - \barEp[\rho_u(\tilde y_i-\tilde x_i'\eta_u)] \geq  \frac{\underf\|\tilde x_i'\delta\|_{2,n}^2}{4} \wedge \left\{ \bar q_{A}\underf\|\tilde x_i'\delta\|_{2,n}\right\} -  \bar{f} \|r_{ui}\|_{2,n}\|\tilde x_i'\delta\|_{2,n}.$$
\end{lemma}
{\bf Proof.}{\bf \ (Proof of Lemma \ref{Lemma:IdentificationSparseQRNP})} Let $T = \supp(\eta_u)$, $Q_u(\eta):=\barEp[\rho_u(\tilde y_i-\tilde x_i'\eta)]$, $J_u=(1/2)\En\[f_i\tilde x_i\tilde x_i'\]$ and define $\|\delta\|_u = \|J_u^{1/2}\delta\|$. The proof proceeds in steps.

 Step 1. (Minoration).   Define the maximal radius over which the criterion function can be minorated by a quadratic function
$$ r_A = \sup_{r} \left\{ r \ : Q_u(\eta_u+ \delta) - Q_u(\eta_u) + \bar{f}\|r_{ui}\|_{2,n}\|\tilde x_i'\delta\|_{2,n}  \geq \frac{1}{2} \|  \delta\|^{2}_u, \  \forall  \delta\in A, \ \|\delta\|_u \leq r\right\}.$$
Step 2 below shows that  $r_{A} \geq \bar q_A$. By construction of $r_A$ and the convexity of $Q_u(\cdot)$ and $\|\cdot \|_u$,
 $$ \begin{array}{lll}
 && Q_u(\eta_u + \delta) - Q_u(\eta_u) + \bar{f}\|r_{ui}\|_{2,n}\|\tilde x_i'\delta\|_{2,n} \geq  \\
&&     \geq \frac{\|\delta\|^2_u}{2} \wedge \left\{ \frac{\|\delta\|_u}{r_A} \cdot \inf_{\tilde \delta\in A, \| \tilde \delta\|_u \geq r_A} \!\! Q_u(\eta_u+\tilde \delta) - Q_u(\eta_u) + \bar{f}\|r_{ui}\|_{2,n}\|\tilde x_i'\tilde\delta\|_{2,n} \right\}\\
&&  \geq   \frac{\|\delta\|^2_u}{2} \wedge \left\{ \frac{\|\delta\|_u}{r_A} \frac{r_A^2}{4}\right\}
 \geq    \frac{\| \delta\|_u^2}{2} \wedge \left\{ \bar q_A\|\delta\|_u\right\}.
\end{array}$$

Step 2. ($r_A \geq \bar q_A$) Let $F_{\tilde y \mid \tilde x}$ denote the conditional distribution of
$\tilde y$ given $\tilde x$. From \cite{Knight1998},
for any two scalars $w$ and $v$  we have that
\begin{equation}\label{Eq:TrickRho}
\rho_u(w-v) - \rho_u(w) = -v (u - 1\{w\leq 0\}) + \int_0^v(
1\{w\leq z\} - 1\{w\leq 0\})dz.
\end{equation}
We will use (\ref{Eq:TrickRho}) with $w=\tilde y_i - \tilde x_i'\eta_u$ and $v =
\tilde x_i'\delta$. Using the law of iterated expectations and mean value
expansion, we obtain for $\tilde t_{\tilde x_i,t} \in [0,t]$
\begin{equation}\label{Eq:Piece1}
\begin{array}{rcl}
&& Q_u(\eta_u + \delta) - Q_u(\eta_u) + \bar f \|r_{ui}\|_{2,n}\|\tilde x_i'\delta\|_{2,n} \geq \\
&& Q_u(\eta_u + \delta) - Q_u(\eta_u) +\barEp\[(u-1\{\tilde y_i\leq \tilde x_i'\eta_u\})\tilde x_i'\delta\] = \\
&& = \barEp\[ \int_0^{\tilde x_i'\delta} F_{\tilde y_i|\tilde x_i}(\tilde x_i'\eta_u + t) - F_{\tilde y_i|\tilde x_i}(\tilde x_i'\eta_u) dt \] \\
&& =  \barEp\[ \int_0^{\tilde x_i'\delta} tf_{\tilde y_i|\tilde x_i}(\tilde x_i'\eta_u) + \frac{t^2}{2}f'_{\tilde y_i|\tilde x_i}(\tilde x_i'\eta_u+\tilde t_{\tilde x,t}) dt \] \\
&& \geq    \|\delta\|_u^2  - \frac{1}{6}\bar f ' \barEp[|\tilde x_i'\delta|^3] -  \barEp\[ \int_0^{\tilde x_i'\delta} t[f_{\tilde y_i\mid \tilde x_i}(\tilde x_i'\eta_u)-f_{\tilde y_i\mid \tilde x_i}(g_{ui})]dt\]\\
&&\geq \frac{1}{2} \|\delta\|^2_u  + \frac{1}{4} \underf \barEp[|\tilde x_i'\delta|^2] - \frac{1}{6} \bar f' \barEp[|\tilde x_i'\delta|^3]-(\bar f'/2) \barEp\[ |\tilde x_i'\eta_u - g_{ui}|\cdot |\tilde x_i'\delta|^2\].\\
\end{array}
\end{equation} where the first inequality follows noting that $F_{\tilde y_i\mid \tilde x_i}(\tilde x_i'\eta_u+r_{ui})=u$ and $|F_{\tilde y_i\mid \tilde x_i}(\tilde x_i'\eta_u+r_{ui})-F_{\tilde y_i\mid \tilde x_i}(\tilde x_i'\eta_u)|\leq \bar f |r_{ui}|$.

Moreover, by assumption we have
\begin{equation}\label{Eq:AUXlemma}\begin{array}{rl}
 \barEp\[ |\tilde x_i'\eta_u  - g_{ui}|\cdot |\tilde x_i'\delta|^2\] & =  \barEp\[ |r_{ui}|\cdot |\tilde x_i'\delta|^2\]\\
 & \leq (\underf/8)(2/\bar f')\barEp[|\tilde x_i'\delta|^2] \\
\end{array} \end{equation}

Note that for any $\delta$ such that $ \|\delta\|_u \leq  \bar q_A$ we have $ \|\delta\|_u \leq  \bar q_A \leq
 (1/2) \cdot (\underf^{3/2}/\bar{f'}) \cdot \barEp\[|\tilde x_i'\delta|^2\]^{3/2}/\barEp\[|\tilde x_i'\delta|^3\]$,
it follows that  $(1/6)\bar f'\barEp[|\tilde x_i'\delta|^3] \leq
(1/8) \underf \barEp[|\tilde x_i'\delta|^2]$. Combining this with (\ref{Eq:AUXlemma}) we have
\begin{equation}\label{Eq:Piece2}\frac{1}{4} \underf \barEp[|\tilde x_i'\delta|^2] - \frac{1}{6} \bar f' \barEp[|\tilde x_i'\delta|^3]-(\bar f'/2) \barEp\[ |\tilde x_i'\eta_u - g_{ui}|\cdot |\tilde x_i'\delta|^2\] \geq 0.\end{equation}

Combining (\ref{Eq:Piece1}) and (\ref{Eq:Piece2}) we have $r_{A} \geq \bar q_{A}$.
{\bf $\square$}

\begin{lemma}\label{Lemma:ControlR}
Under Condition PQR we have $\barEp[\hat R(\eta_u)] \leq \bar f \|r_{ui}\|_{2,n}^2/2$, $\hat R(\eta_u)\geq 0$ and
$$P( \hat R(\eta_u) \geq 4\max\{  \bar f \|r_{ui}\|_{2,n}^2, \|r_{ui}\|_{2,n}\sqrt{\log(8/\gamma)/n}\} ) \leq \gamma.$$
\end{lemma}
{\bf Proof.}{\bf \ (Proof of Lemma \ref{Lemma:ControlR})}
We have that $\hat R(\eta_u)\geq 0$ by convexity of $\rho_u$. Let $\epsilon_{ui}=\tilde y_i - \tilde x_i'\eta_u - r_{ui}$. By Knight's identity, $\hat R(\eta_u) =  -\En[ r_{ui} \int_0^1  1\{\epsilon_{ui} \leq -t r_{ui}\} - 1\{\epsilon_{ui}\leq 0\} \ dt \geq 0$.
 $$ \begin{array}{rl}
 \barEp[\hat R(\eta_u)] & = \En[ r_{ui} \int_0^1  F_{y_i\mid \tilde x_i}(\tilde x_i'\eta_u+(1-t)r_{ui}) - F_{y_i\mid \tilde x_i}(\tilde x_i'\eta_u+r_{ui}) \ dt]\\
  & \leq \En[ r_{ui} \int_0^1  \bar f t r_{ui} dt] \leq \bar f \|r_{ui}\|_{2,n}^2/2.\end{array}$$
Therefore
$P( \hat R(\eta_u) \leq 2 \bar f \|r_{ui}\|_{2,n}^2 ) \geq 1/2$ by Markov's inequality.

Define $z_{ui} := -\int_0^1  1\{\epsilon_{ui} \leq -t r_{ui}\} - 1\{\epsilon_{ui}\leq 0\} \ dt$, so that $\hat R(\eta_u) = \En[r_{ui}z_{ui}]$. We have $P( \En[r_{ui}z_{ui}] \leq 2 \bar f \|r_{ui}\|_{2,n}^2 ) \geq 1/2$ so that for $t\geq 4 \bar f \|r_{ui}\|_{2,n}^2$ we have by Lemma 2.3.7 in \cite{vdVaartWellner2007}
$$ \frac{1}{2}P( |\En[r_{ui}z_{ui}]| \geq t ) \leq 2P(|\En[r_{ui}z_{ui}\epsilon_i]|>t/4) $$

Since the $r_{ui}z_{ui}\epsilon_i$ is a symmetric random variable and $|z_{ui}|\leq 1$, by Theorem 2.15 in \cite{delapena} we have
$$P( \sqrt{n}|\En[r_{ui}z_{ui}\epsilon_i]|> \bar t \sqrt{\En[r_{ui}^2]} ) \leq P( \sqrt{n}|\En[r_{ui}z_{ui}\epsilon_i]|> \bar t \sqrt{\En[r_{ui}^2z_{ui}^2]} ) \leq 2\exp(-\bar t^2/2)\leq \gamma/8$$
for $\bar t \geq \sqrt{2\log(8/\gamma)}$. Setting $t = 4\max\{  \bar f \|r_{ui}\|_{2,n}^2, \|r_{ui}\|_{2,n}\sqrt{\log(8/\gamma)/n}\}$ we have
$$ P( \En[r_{ui}z_{ui}] \geq t ) \leq 4P(\En[r_{ui}z_{ui}\epsilon_i]>t/4)\leq \gamma.$$

{\bf $\square$}

\begin{lemma}\label{Lemma:UpperQ}
Under Condition PQR, conditionally on $\{\tilde x_i, i=1,\ldots,n\}$, for $\|\hat \eta_u\|_0\leq k$, $\underline{N}\leq \|\tilde x_i'(\hat\eta_u-\eta_u)\|_{2,n}\leq \bar N$, we have with probability $1-\gamma$
  {\small $$
  \En[ \rho_u(\tilde y_i-\tilde x_i'\hat\eta_u)]-\En[ \rho_u(\tilde y_i-\tilde x_i'\eta_u)]  \leq \frac{\|\tilde x_i'(\hat\eta_u-\eta_u)\|_{2,n}}{\sqrt{n}}\left\{  4+4\sqrt{\frac{(k+s) \log(16p\{1+3\sqrt{n}\log(\frac{\bar N}{\underline{N}} )\}/\gamma)}{\semin{k+s}}}\right\}$$ $$+ \bar f \|\tilde x_i'(\hat\eta_u-\eta_u)\|_{2,n}^2+ \bar f \|r_{ui}\|_{2,n}\|\tilde x_i'(\hat\eta_u-\eta_u)\|_{2,n}.$$}
\end{lemma}
{\bf Proof.}{\bf \ (Proof of Lemma \ref{Lemma:UpperQ})}
By triangle inequality we have
$$ \begin{array}{rl}
\En[ \rho_u(\tilde y_i-\tilde x_i'\hat\eta_u)- \rho_u(\tilde y_i-\tilde x_i'\eta_u)] & \leq |(\En-\barEp)[ \rho_u(\tilde y_i-\tilde x_i'\hat\eta_u) \rho_u(\tilde y_i-\tilde x_i'\eta_u)]|\\
& +|\barEp[ \rho_u(\tilde y_i-\tilde x_i'\hat\eta_u) - \rho_u(\tilde y_i-\tilde x_i'\eta_u)]|.\end{array}$$
The first term is bounded by Lemma \ref{Lemma:EmpProc:Normalized}. The second term is bounded using the identity (\ref{Eq:TrickRho}) with $w=\tilde y_i - \tilde x_i'\eta_u$ and $v =
\tilde x_i'\delta$ similarly to the argument in (\ref{Eq:Piece1}). Using the law of iterated expectations and mean value
expansion, we obtain for $\tilde t_{\tilde x_i,t} \in [0,t]$
\begin{equation}\label{Eq:Piece11}
\begin{array}{rcl}
&& Q_u(\eta_u + \delta) - Q_u(\eta_u) - \bar f \|r_{ui}\|_{2,n}\|\tilde x_i'\delta\|_{2,n} \leq \\
&& Q_u(\eta_u + \delta) - Q_u(\eta_u) +\barEp\[(u-1\{\tilde y_i\leq \tilde x_i'\eta_u\})\tilde x_i'\delta\] = \\
&& = \barEp\[ \int_0^{\tilde x_i'\delta} F_{\tilde y_i|\tilde x_i}(\tilde x_i'\eta_u + t) - F_{\tilde y_i|\tilde x_i}(\tilde x_i'\eta_u) dt \] \\
\end{array}
\end{equation} and noting that
$$ \En\[ \int_0^{\tilde x_i'\delta} F_{\tilde y_i|\tilde x_i}(\tilde x_i'\eta_u + t) - F_{\tilde y_i|\tilde x_i}(\tilde x_i'\eta_u) dt \] \leq \bar{f}\En\[ \int_0^{\tilde x_i'\delta} t dt \]\leq \bar{f} \|\tilde x_i'\delta\|_{2,n}^2.$$

{\bf $\square$}

\begin{lemma}\label{Lemma:EmpProc:Normalized} Let $w_i(b) = \rho_u(\tilde y_i-\tilde x_i'\eta_u -b)-\rho_u(\tilde y_i-\tilde x_i'\eta_u)$. Then, conditional on $\{\tilde x_1,\ldots,\tilde x_n\}$, w
e have with probability $1-\gamma$ that for vectors in the restricted set
{ \begin{eqnarray*}
  \sup_{\footnotesize \begin{array}{c}\delta \in \Delta_{\cc},\\ \underline{N}\leq\|\tilde x_i'\delta\|_{2,n}\leq \bar N\end{array}} \left| \Gn\left( \frac{w_i(\tilde x_i'\delta)}{\|\tilde x_i'\delta\|_{2,n}}\right) \right| \leq 4+\frac{4(1+\cc)\sqrt{s \log(16p\{1+3\sqrt{n}\log(\frac{\bar N}{\underline{N}})\}/\gamma)}}{\kappa_\cc}
\end{eqnarray*}}
Similarly, for sparse vectors
{  \begin{eqnarray*}
 \sup_{\footnotesize \begin{array}{c}1\leq \|\delta\|_0 \leq k,\\ \underline{N}\leq\|\tilde x_i'\delta\|_{2,n}\leq \bar N\end{array}} \left| \Gn\left( \frac{w_i(\tilde x_i'\delta)}{\|\tilde x_i'\delta\|_{2,n}}\right) \right| \leq 4+4\sqrt{\frac{k \log(16p\{1+3\sqrt{n}\log(\bar N/\underline{N})\}/\gamma)}{\semin{k}}}
\end{eqnarray*}}
Similarly, for $\ell_1$-bounded vectors
{ \begin{eqnarray*}
  \sup_{\footnotesize \begin{array}{c}\|\delta\|_1 \leq R_1,\\ \underline{N}\leq\|\tilde x_i'\delta\|_{2,n}\leq \bar N\end{array}} \left| \Gn\left( \frac{w_i(\tilde x_i'\delta)}{\|\tilde x_i'\delta\|_{2,n}}\right) \right| \leq 4+4\frac{R_1}{\underline{N}}\sqrt{\log(16p\{1+3\sqrt{n}\log(\bar N/\underline{N})\}/\gamma)}
\end{eqnarray*}}
\end{lemma}
{\bf Proof.}{\bf \ (Proof of Lemma \ref{Lemma:EmpProc:Normalized})}
Let $w_i(b) = \rho_u(\tilde y_i-\tilde x_i'\eta_u -b)-\rho_u(\tilde y_i-\tilde x_i'\eta_u)\leq |b|$. Note that   $w_i(b)- w_i(a)\leq |b-a|$.

For any $\delta\in\RR^p$, since $\rho_u$ is $1$-Lipschitz, we have
$$ \begin{array}{rl}
{\rm var}\left( \Gn\left( \frac{w_i(\tilde x_i'\delta)}{\|\tilde x_i'\delta\|_{2,n}}\right) \right) & \leq \frac{\En[\{w_i(\tilde x_i'\delta)\}^2] }{\|\tilde x_i'\delta\|_{2,n}^2}  \leq  \frac{\En[|\tilde x_i'\delta|^2] }{\|\tilde x_i'\delta\|_{2,n}^2} \leq 1. \end{array}$$

Then, by Lemma 2.3.7 in \cite{vdV-W} (Symmetrization for Probabilities) we have for any $M>1$
$$ P\left( \sup_{\delta \in \Delta_{\cc}} \left| \Gn\left( \frac{w_i(\tilde x_i'\delta)}{\|\tilde x_i'\delta\|_{2,n}}\right) \right| \geq M \right) \leq \frac{2}{1-M^{-2}}P\left( \sup_{\delta \in \Delta_{\cc}} \left| \Gn^o\left( \frac{w_i(\tilde x_i'\delta)}{\|\tilde x_i'\delta\|_{2,n}}\right) \right| \geq M/4 \right)$$ where $\Gn^o$ is the symmetrized process.

Consider $ \mathcal{F}_t = \{ \delta \in \Delta_\cc : \|\tilde x_i'\delta\|_{2,n} = t\}$. We will consider the families of $\mathcal{F}_t$ for $t \in [\underline{N},\bar N]$. For any $\delta \in \mathcal{F}_t$, $t \leq \tilde t$ we have
$$\begin{array}{rl}
\left| \Gn^o\left( \frac{w_i(\tilde x_i'\delta)}{t} - \frac{w_i(\tilde x_i'\delta(\tilde t/t))}{\tilde t} \right) \right| & \leq \left| \Gn^o\left( \frac{w_i(\tilde x_i'\delta)}{t} - \frac{w_i(\tilde x_i'\delta(\tilde t/t))}{t} \right) \right| + \left| \Gn^o\left( \frac{w_i(\tilde x_i'\delta(\tilde t/t))}{t} - \frac{w_i(\tilde x_i'\delta(\tilde t/t))}{\tilde t} \right) \right|\\
& = \frac{1}{t}\left| \Gn^o\left( w_i(\tilde x_i'\delta) - w_i(\tilde x_i'\delta[\tilde t/t]) \right) \right| + \left| \Gn^o\left( w_i(\tilde x_i'\delta(\tilde t/t)) \right) \right|\cdot \left| \frac{1}{t} - \frac{1}{\tilde t}\right| \\
& \leq \sqrt{n}\En\left(\frac{|\tilde x_i'\delta|}{t}\right) \frac{|t-\tilde t|}{t} +    \sqrt{n}\En\left( |\tilde x_i'\delta|\right)  \frac{\tilde t}{t}\left| \frac{1}{t} - \frac{1}{\tilde t}\right|\\
& = 2 \sqrt{n}\En\left(\frac{|\tilde x_i'\delta|}{t}\right) \left| \frac{t-\tilde t}{ t} \right| \leq 2\sqrt{n} \left| \frac{t-\tilde t}{t} \right|.
\end{array}
$$
Let $\mathcal{T}$ be a $\varepsilon$-net $\{\underline{N}=:t_1,t_2,\ldots,t_K:= \bar N\}$ of $[\underline{N},\bar N]$ such that $|t_k-t_{k+1}|/t_k \leq 1/[2\sqrt{n}]$. Note that we can achieve that with $|\mathcal{T}|\leq 3\sqrt{n}\log(\bar N/ \underline{N})$.

Therefore we have $$ \sup_{\delta \in \Delta_{\cc}} \left| \Gn^o\left( \frac{w_i(\tilde x_i'\delta)}{\|\tilde x_i'\delta\|_{2,n}}\right) \right| \leq 1 + \sup_{t\in \mathcal{T}}\sup_{\delta \in \Delta_{\cc}, \|\tilde x_i'\delta\|_{2,n}=t} \left| \Gn^o\left( \frac{w_i(\tilde x_i'\delta)}{t}\right) \right|=: 1 + \mathcal{A}^o.$$

$$ \begin{array}{rl}
P( \mathcal{A}^o \geq K ) & \leq \min_{\psi\geq 0} \exp(-\psi K)\Ep[ \exp(\psi\mathcal{A}^o)]\\
& \leq 8p|\mathcal{T}|\min_{\psi\geq 0} \exp(-\psi K)\exp\left( 8\psi^2 \frac{s(1+\cc)^2}{\kappa_{\cc}^2}   \right)\\
&\leq 8p|\mathcal{T}| \exp (-K^2/[16 \frac{s(1+\cc)^2}{\kappa_{\cc}^2}] )
\end{array}
$$ where we set $\psi = K/ [ 16\frac{s(1+\cc)^2}{\kappa_{\cc}^2}] $ and bounded
{\small $$\begin{array}{rl}
\displaystyle \Ep\left[ \exp\left( \psi \mathcal{A}^o \right)\right]&
\displaystyle  \leq_{(1)} 2|\mathcal{T}|\sup_{t\in \mathcal{T}} \Ep\left[ \exp\left( \psi \sup_{\delta \in \Delta_{\cc},  \|\tilde x_i'\delta\|_{2,n}=t} \Gn^o\left( \frac{w_i(\tilde x_i'\delta)}{t}\right)  \right)\right] \\
& \displaystyle \leq_{(2)} 2|\mathcal{T}|\sup_{t\in \mathcal{T}} \Ep\left[ \exp\left( 2\psi \sup_{\delta \in \Delta_{\cc},  \|\tilde x_i'\delta\|_{2,n}=t} \Gn^o\left( \frac{\tilde x_i'\delta}{t}\right)  \right)\right] \\
& \displaystyle \leq_{(3)} 2|\mathcal{T}|\sup_{t\in \mathcal{T}} \Ep\left[ \exp\left( 2\psi \left[\sup_{\delta \in \Delta_{\cc},  \|\tilde x_i'\delta\|_{2,n}=t} 2\frac{\|\delta\|_1}{t}\right]\max_{j\leq p}|\Gn^o(\tilde x_{ij})|  \right)\right] \\
& \displaystyle \leq_{(4)} 2|\mathcal{T}|  \Ep\left[ \exp\left( 4\psi
 \frac{\sqrt{s}(1+\cc)}{\kappa_{\cc}} \max_{j\leq p}|\Gn^o( \tilde x_{ij})|  \right)\right] \\
& \displaystyle \leq_{(5)} 4p|\mathcal{T}|\max_{j\leq p}\Ep\left[ \exp\left( 4\psi \frac{\sqrt{s}(1+\cc)}{\kappa_{\cc}} \Gn^o(\tilde  x_{ij})  \right)\right] \\
& \displaystyle \leq_{(6)} 8p|\mathcal{T}|\exp\left( 8\psi^2 \ \frac{s(1+\cc)^2}{\kappa_{\cc}^2}   \right) \\
\end{array} $$}
\noindent where (1) follows by $\exp(\max_{i\in I} |z_i|) \leq 2|I|\max_{i\in I} \exp(z_i)$, (2) by contraction principle (Theorem 4.12 \cite{LedouxTalagrandBook}), (3) $|\Gn^o(\tilde x_i'\delta)| \leq \|\delta\|_1\|\Gn^o(\tilde x_i)\|_\infty$, (4) $\sqrt{s}(1+\cc)\|\tilde x_i'\delta\|_{2,n}/\|\delta\|_1 \geq \kappa_{\cc}$, (6) $\En[x_{ij}^2]=1$ and $\exp(z)+\exp(-z)\leq 2\exp(z^2/2)$.

The second result follows similarly by noting that
$$ \sup_{1\leq \|\delta\|_0 \leq k,  \|\tilde x_i'\delta\|_{2,n}=t} \frac{\|\delta\|_1}{t} \leq \sup_{1\leq \|\delta\|_0 \leq k,  \|\tilde x_i'\delta\|_{2,n}=t}\frac{\sqrt{k}\|\tilde x_i'\delta\|_{2,n}}{t\sqrt{\semin{k}}} = \frac{\sqrt{k}}{\sqrt{\semin{k}}}. $$
The third result follows similarly by noting that for ant $t\in [\underline{N},\bar N]$
$$ \sup_{\|\delta\|_1 \leq R_1,  \|\tilde x_i'\delta\|_{2,n}=t} \frac{\|\delta\|_1}{t} \leq \frac{R_1}{\underline{N}}. $$
{\bf $\square$}

\section{Proofs for Section \ref{Sec:EstLasso} of Supplementary Material}

\begin{lemma}[Choice of $\lambda$]\label{Thm:ChoiceLambda}
Suppose Condition WL holds, let $c'>c>1$, $\gamma \leq 1/n^{1/3}$,  and $\lambda \geq 2c'\sqrt{n}\Phi^{-1}(1-\gamma/2p).$ Then for $n \geq n_0(\delta_n,c',c)$ large enough
$$P( \lambda/n \geq 2c \|\hat  \Gamma^{-1}_{\tau 0}\En[ f_i x_i v_i ]\|_\infty ) \geq 1-\gamma\{1+o(1)\}+4\Delta_n.$$
\end{lemma}
{\bf Proof.}{\bf \ (Proof of Lemma \ref{Thm:ChoiceLambda})}
Since $\hat \Gamma_{\tau 0jj} = \sqrt{\En[\hat f_i^2 x_{ij}^2 v_i^2]}$ and $\Gamma_{\tau 0jj} = \sqrt{\En[f_i^2 x_{ij}^2 v_i^2]}$, with probability at least $1-\Delta_n$ we have $$\max_{j\leq p}|\hat \Gamma_{\tau 0jj}-\Gamma_{\tau 0jj}|\leq \max_{j\leq p} \sqrt{\En[(\hat f_i-f_i)^2 x_{ij}^2 v_i^2]}\leq \delta_n^{1/2}$$ by Condition WL(iii). Further, Condition WL implies that $\Gamma_{\tau 0jj}$ is bounded away from zero and from above uniformly in $j=1,\ldots,p$ and $n$. Thus we have $\|\hat \Gamma^{-1}_{\tau 0}\Gamma_{\tau 0}\|_\infty \to_P 1$, so that $\|\hat \Gamma^{-1}_{\tau 0}\Gamma_{\tau 0}\|_\infty \leq \sqrt[4]{c'/c}$ with probability $1-\Delta_n$ for $n \geq n_0(\delta_n,c',c,\Gamma_{\tau 0})$.
By the triangle inequality
\begin{equation}\label{Eq:LEWfirst} \| \hat \Gamma^{-1}_{\tau 0} \En[ f_i x_i v_i ]\|_\infty \leq \| \hat \Gamma^{-1}_{\tau 0} \Gamma_{\tau 0}\|_\infty \|\Gamma_{\tau 0}^{-1}\En[ f_ix_i  v_i ]\|_\infty \end{equation}

Next we will apply Lemma \ref{Lemma:SNMD} which is based on self-normalized moderate deviation theory. Define $U_{ij} = f_ix_{ij}v_i - \Ep[f_ix_{ij}v_i]$ which is zero mean by construction and $\En[ f_i x_{ij} v_i ]= \En[U_{ij}]$ since $\barEp[f_ix_{ij}v_i]=0$. Moreover, we have $\min_{j\leq  p} \barEp[U_{ij}^2]\geq c$ by Condition WL(ii) and $\max_{j\leq p}\barEp[|U_{ij}|^3] \lesssim \max_{j\leq p}\barEp[|f_ix_{ij}v_i|^3] \leq C$  since $U_{ij}$ is demeaned and the last bound from by Condition WL(ii).  Using that  by Condition WL(iii), with probability $1-\Delta_n$ we have $\max_{j\leq p}|(\En-\barEp)[U_{ij}^2]|\leq \delta_n$ and by Condition WL(ii) $\min_{j\leq p} \barEp[U_{ij}^2] \geq c$, we have that $\sqrt{\En[U_{ij}^2]} \leq \sqrt[4]{c'/c}\sqrt{\En[f_i^2 x_{ij}^2 v_i^2]}$ with probability $1-\Delta_n$ for $n$ sufficiently large. Therefore,
$$\begin{array}{rl}
P( \lambda/n \geq 2c \|\hat  \Gamma^{-1}_{\tau 0}\En[ f_i x_i v_i ]\|_\infty ) & \geq P( \Phi^{-1}(1-\gamma/2p) \geq \frac{c}{c'}\sqrt[4]{\frac{c'}{c}} \sqrt{n}\| \Gamma^{-1}_{\tau 0}\En[ f_i x_i v_i ]\|_\infty ) -\Delta_n\\
& = P\left( \Phi^{-1}(1-\gamma/2p)>\frac{c}{c'}\sqrt[4]{\frac{c'}{c}}\max_{j\leq p} \left| \frac{\sqrt{n}\En[ f_i x_{ij} v_i ]}{\sqrt{\En[f_i^2 x_{ij}^2 v_i^2]}}\right| \right) \\
& \geq P\left( \Phi^{-1}(1-\gamma/2p)>\frac{c}{c'}\sqrt[2]{\frac{c'}{c}}\max_{j\leq p} \left| \frac{\sqrt{n}\En[ U_{ij}]}{\sqrt{\En[U_{ij}^2]}}\right| \right) -2\Delta_n\\
& \geq P\left( \Phi^{-1}(1-\gamma/2p)>\max_{j\leq p} \left| \frac{\sqrt{n}\En[ U_{ij}]}{\sqrt{\En[U_{ij}^2]}}\right| \right) -2\Delta_n\\
& \geq 1-2p\Phi(\Phi^{-1}(1-\gamma/2p)) ( 1 + o(1))-2\Delta_n \\
& \geq 1-\gamma\{ 1 + o(1)\} -2\Delta_n\end{array} $$
where the last relation by Condition WL.
{\bf $\square$}

{\bf Proof.}{\bf \ (Proof of Lemma \ref{Thm:RateEstimatedLasso})}
Let $\hat\delta = \hat\theta_\tau-\theta_\tau$. By definition of $\hat\theta_\tau$ we have
\begin{equation}\label{Eq:EstLassoV2}\begin{array}{rl}
 \En[\hat f_i^2(x_i'\hat\delta)^2] - 2\En[\hat f_i^2(d_i-x_i'\theta_\tau)x_i]'\hat\delta & = \En[\hat f_i^2(d_i-x_i'\hat\theta_\tau)^2] -\En[\hat f_i^2(d_i-x_i'\theta_\tau)^2] \\
& \leq \frac{\lambda}{n}\|\hat  \Gamma_\tau \theta_\tau\|_1 -\frac{\lambda}{n}\|\hat  \Gamma_\tau \hat \theta_\tau\|_1 \\
 & \leq \frac{\lambda}{n}\|\hat  \Gamma_\tau\hat\delta_{T_\mtau}\|_1-\frac{\lambda}{n}\|\hat  \Gamma_\tau \hat\delta_{T_\mtau^c}\|_1\\
& \leq \frac{\lambda}{n}u\|\hat  \Gamma_{\tau 0}\hat\delta_{T_\mtau}\|_1-\frac{\lambda}{n}\ell\|\hat  \Gamma_{\tau 0}\hat\delta_{T_\mtau^c}\|_1\end{array}\end{equation}
Therefore, using that $c_f^2 \geq \En[(\hat f_i^2-f_i^2)^2v_i^2/\{\hat f_i^2f_i^2\}]$ and $ c_r^2 \geq \En[\hat f^2 r_{\mtau i}^2]$, we have
{\small \begin{equation}\label{LassoNewEq}\begin{array}{rl}\En[\hat f_i^2(x_i'\hat\delta)^2] & \leq 2\En[(\hat f_i^2-f_i^2)v_ix_i/f_i]'\hat\delta + 2\En[\hat f_i^2r_{\mtau i}x_i]'\hat\delta+2(\hat  \Gamma^{-1}_0\En[f_iv_ix_i])'(\hat  \Gamma_{\tau 0}\hat\delta)\\
&+\frac{\lambda}{n}u\|\hat  \Gamma_{\tau 0}\hat\delta_{T_\mtau}\|_1-\frac{\lambda}{n}\ell\|\hat  \Gamma_{\tau 0}\hat\delta_{T_\mtau^c}\|_1\\
& \leq 2\{ c_f+ c_r\} \{\En[\hat f_i^2(x_i'\hat\delta)^2]\}^{1/2}+2\|\hat  \Gamma^{-1}_0\En[f_i^2(d_i-x_i'\theta_\tau)x_i]\|_\infty\|\hat  \Gamma_{\tau 0}\hat\delta\|_1\\
&+\frac{\lambda}{n}u\|\hat  \Gamma_{\tau 0}\hat\delta_{T_\mtau}\|_1-\frac{\lambda}{n}\ell\|\hat  \Gamma_{\tau 0}\hat\delta_{T_\mtau^c}\|_1\\
& \leq 2\{c_f+ c_r\} \{\En[\hat f_i^2(x_i'\hat\delta)^2]\}^{1/2}+\frac{\lambda}{cn}\|\hat  \Gamma_{\tau 0}\hat\delta\|_1+\frac{\lambda}{n}u\|\hat  \Gamma_{\tau 0}\hat\delta_{T_\mtau}\|_1-\frac{\lambda}{n}\ell\|\hat  \Gamma_{\tau 0}\hat\delta_{T_\mtau^c}\|_1\\
& \leq 2\{ c_f+ c_r\} \{\En[\hat f_i^2(x_i'\hat\delta)^2]\}^{1/2}+\frac{\lambda}{n}\left(u+\frac{1}{c}\right)\|\hat  \Gamma_{\tau 0}\hat\delta_{T_\mtau}\|_1-\frac{\lambda}{n}\left(\ell-\frac{1}{c}\right)\|\hat  \Gamma_{\tau 0}\hat\delta_{T_\mtau^c}\|_1\\
\end{array}\end{equation}}
Let $\tilde \cc = \frac{cu+1}{c\ell-1}\|\hat\Gamma_{\tau 0}\|_\infty\|\hat\Gamma_{\tau 0}^{-1}\|_\infty$. If $\hat \delta \not\in \Delta_{\tilde \cc}$ we have $ \left(u+\frac{1}{c}\right)\|\hat  \Gamma_{\tau 0}\hat\delta_{T_\mtau}\|_1\leq \left(\ell-\frac{1}{c}\right)\|\hat  \Gamma_{\tau 0}\hat\delta_{T_\mtau^c}\|_1$ so that
$$ \{\En[\hat f_i^2(x_i'\hat\delta)^2]\}^{1/2} \leq 2\{ c_f+ c_r\}.$$

Otherwise assume $\hat\delta \in \Delta_{\tilde \cc}$. In this case (\ref{LassoNewEq}) yields
{\small $$\begin{array}{rl}\En[\hat f_i^2(x_i'\hat\delta)^2]
& \leq 2\{ c_f+ c_r\} \{\En[\hat f_i^2(x_i'\hat\delta)^2]\}^{1/2}+\frac{\lambda}{n}\left(u+\frac{1}{c}\right)\|\hat  \Gamma_{\tau 0}\hat\delta_{T_\mtau}\|_1-\frac{\lambda}{n}\left(\ell-\frac{1}{c}\right)\|\hat  \Gamma_{\tau 0}\hat\delta_{T_\mtau^c}\|_1\\
& \leq 2\{c_f+ c_r\} \{\En[\hat f_i^2(x_i'\hat\delta)^2]\}^{1/2}+\frac{\lambda}{n}\left(u+\frac{1}{c}\right)\sqrt{s}\{\En[\hat f_i^2(x_i'\hat\delta)^2]\}^{1/2}/\hat\kappa_{\tilde \cc}\\
\end{array}$$}
which implies
$$  \{\En[\hat f_i^2(x_i'\hat\delta)^2]\}^{1/2} \leq 2 \{c_f+ c_r\} + \frac{\lambda\sqrt{s}}{n\hat\kappa_{\tilde \cc}}\left(u+\frac{1}{c}\right)$$

To establish the $\ell_1$-bound, first assume that $\hat\delta \in \Delta_{2\tilde \cc}$. In that case
$$ \|\hat\delta\|_1 \leq (1+2\tilde\cc)\|\hat\delta_{T_\mtau}\|_1 \leq \sqrt{s}\{\En[\hat f_i^2(x_i'\hat\delta)^2]\}^{1/2}/\hat \kappa_{2\tilde\cc} \leq  2 \frac{\sqrt{s}\{c_f+c_r\}}{\hat \kappa_{2\tilde\cc}} + \frac{\lambda s }{n\hat\kappa_{\tilde \cc}\hat \kappa_{2\tilde\cc}}\left(u+\frac{1}{c}\right).$$

Otherwise note that $\hat\delta \not\in \Delta_{2\tilde \cc}$ implies that $ \left(u+\frac{1}{c}\right)\|\hat  \Gamma_{\tau 0}\hat\delta_{T_\mtau}\|_1\leq \frac{1}{2}\cdot\left(\ell-\frac{1}{c}\right)\|\hat  \Gamma_{\tau 0}\hat\delta_{T_\mtau^c}\|_1$ so that (\ref{LassoNewEq}) gives
$$\frac{1}{2}\frac{\lambda}{n}\cdot\left(\ell-\frac{1}{c}\right)\|\hat  \Gamma_{\tau 0}\hat\delta_{T_\mtau^c}\|_1 \leq \{\En[\hat f_i^2(x_i'\hat\delta)^2]\}^{1/2} \left( 2\{c_f+c_r\}-\{\En[\hat f_i^2(x_i'\hat\delta)^2]\}^{1/2}\right) \leq \{c_f+c_r\}^2.$$
Therefore
$$\|\hat\delta\|_1 \leq \left(1+\frac{1}{2\tilde\cc}\right)\|\hat\delta_{T_\mtau^c}\|_1 \leq \left(1+\frac{1}{2\tilde\cc}\right)\|\hat\Gamma_{\tau 0}^{-1}\|_\infty\|\hat\Gamma_{\tau 0}\hat\delta_{T_\mtau^c}\|_1\leq \left(1+\frac{1}{2\tilde\cc}\right)\frac{2c\|\hat\Gamma_{\tau 0}^{-1}\|_\infty}{\ell c-1}\frac{n}{\lambda}\{c_f+c_r\}^2$$
{\bf $\square$}

{\bf Proof.}{\bf \ (Proof of Lemma \ref{Thm:EstLassoPostRate})}
Note that $\|\hat f\|_\infty^2$ and $\|\widehat \Gamma_{0}^{-1}\|_\infty$ are uniformly bounded with probability going to one. Under the assumption on the design, for $\mathcal{M}$ defined in Lemma \ref{Thm:Sparsity} we have that
$\min_{m \in \mathcal{M}}\semax{m\wedge n}$ is uniformly bounded. Thus by Lemma \ref{Thm:Sparsity} with probability $1-\gamma-o(1)$ we have
$$ \hat s_m \lesssim   \[\frac{n \{c_f+ c_r\}}{\lambda} + \sqrt{s}\]^2.$$

The bound then follows from  Lemma \ref{Thm:2StepMain}.
{\bf $\square$}

\subsection{Technical Results for Post-Lasso with Estimated Weights}

\begin{lemma}[Performance of the Post-Lasso]\label{Thm:2StepMain}
Under Conditions WL, let $\widehat T_\mtau$ denote the support selected by $\widehat \theta_\tau$, and $\widetilde \theta_\tau$ be
the Post-Lasso estimator based on $\widehat T_\mtau$. Then we have for $\hat s_\mtau = |\widehat T_\mtau|$, with probability $1-o(1)$
$$
\begin{array}{l}
\displaystyle  \| \hat f_i(x_i'\theta_\tau + r_{\mtau i} - x_i'\widetilde \theta_\tau)\|_{2,n} \lesssim \sqrt{\frac{\semax{\hat s_\mtau}}{\semin{\hat s_\mtau}}}\frac{c_f}{\min_{i\leq n} \hat f_i} +\frac{\sqrt{\hat s_\mtau }\sqrt{ \log p }}{\sqrt{n \ \semin{\hat s_\mtau}}\min_{i\leq n} \hat f_i}\\  + \min_{\supp(\theta)\subseteq\widehat T_\mtau}  \|\hat f_i(x_i'\theta_\tau + r_{\mtau i} - x_i'\theta)\|_{2,n}\\
\end{array}
$$
Moreover, if in addition $\lambda$ satisfies (\ref{Eq:reg}), and $\ell \widehat \Gamma_{\tau 0} \leq \widehat \Gamma_\tau \leq u \widehat \Gamma_{\tau 0}$ with $u \geq 1 \geq \ell > 1/c$ in the first stage for Lasso, then we have with probability $1-\gamma-o(1)$
$$\min_{\supp(\theta)\subseteq\widehat T_\mtau}  \|\hat f_i(x_i'\theta_\tau + r_{\mtau i}- x_i'\theta)\|_{2,n} \leq 3\{ c_f + c_r\}+  \(u + \frac{1}{c}\) \frac{\lambda \sqrt{s}}{n \kappa_{\tilde \cc}\min_{i\leq n}\hat f_i}+3\bar f C\sqrt{s/n}.$$

\end{lemma}
{\bf Proof.}{\bf \ (Proof of Lemma \ref{Thm:2StepMain})}
Let $F = \diag(f)$, $\widehat F = \diag(\hat f)$, $X=[ x_1;\ldots; x_n]'$, $m_\tau=X\theta_\tau+r_{\mtau}$, and for a set of indices $S \subset \{1,\ldots,p\}$ we define the projection matrix on the columns associated with the indices in $S$ as ${P}_{S} = FX[S](FX[S]'FX[S])^{-1} FX[S]'$ and $\widehat{P}_{S} = \widehat FX[S](X[S]'\widehat F'\widehat FX[S])^{-1} \widehat FX[S]'$. Since $f_id_i = f_im_{\tau i} + v_i$ we have that $\hat f_i d_i = \hat f_i m_{\tau i} + v_i \hat f_i/f_i$ and we have $$\widehat Fm_\tau - \widehat FX\widetilde\theta_\tau = ( I - \widehat{P}_{\hat T_\mtau})\widehat Fm_\tau - \widehat {P}_{\hat T_\mtau}\widehat F F^{-1} v$$ where $I$ is the identity operator. Therefore
\begin{equation}\label{eqPL}
\| \widehat Fm_\tau - \widehat FX\widetilde\theta_\tau \| \leq \|( I - \widehat{P}_{\hat T_\mtau})\widehat Fm_\tau \| +\|\widehat {P}_{\hat T_\mtau}\widehat F F^{-1} v\|.
\end{equation}
Since $\|\widehat FX[\hat T_\mtau]/\sqrt{n}( X[\hat T_\mtau]'\widehat F'\widehat FX[\hat T_\mtau]/n)^{-1} \| \leq \|\widehat F^{-1}\|_\infty\sqrt{1/\semin{\hat s_\mtau}}$,  the last term in (\ref{eqPL}) satisfies
$$\begin{array}{rl}
\|\widehat{P}_{\hat T_\mtau}\widehat F F^{-1}v \|  & \leq  \|\widehat F^{-1}\|_\infty \sqrt{1/\semin{\hat s_\mtau}}  \ \|X[\hat T_\mtau]'\widehat F^2F^{-1}v/\sqrt{n}\| \\
 & \leq \|\widehat F^{-1}\|_\infty \sqrt{1/\semin{\hat s_\mtau}} \left\{  \|X[\hat T_\mtau]'\{\widehat F^2-F^2\}F^{-1}v/\sqrt{n}\| +   \|X[\hat T_\mtau]'Fv/\sqrt{n}\| \right\} \\
 & \leq \|\widehat F^{-1}\|_\infty \sqrt{1/\semin{\hat s_\mtau}}  \left\{  \|X[\hat T_\mtau]'\{\widehat F^2-F^2\}F^{-1}v/\sqrt{n}\| +   \sqrt{\hat s_\mtau}\|X'Fv/\sqrt{n}\|_\infty \right\}.\end{array}$$
Condition WL(iii) implies that
$$ \|X[\hat T_\mtau]'\{\widehat F^2-F^2\}F^{-1}v/\sqrt{n}\| \leq \sup_{\|\alpha\|_0\leq \hat s_\mtau, \|\alpha\|\leq 1}|\alpha'X[\hat T_\mtau]'\{\widehat F^2-F^2\}F^{-1}v/\sqrt{n}| \leq \sqrt{n}\sqrt{\semax{\hat s_\mtau}} c_f.$$
Under Condition WL(iv), by Lemma \ref{Thm:ChoiceLambda} we have with probability $1-o(1)$
$$  \|X'Fv/\sqrt{n}\|_\infty \lesssim_P \sqrt{\log (pn)} \max_{1\leq j\leq p} \sqrt{\En[f_i^2x_{ij}^2v_{i}^2]}.$$
Moreover,  Condition WL(iv) also implies $\max_{1\leq j\leq p} \sqrt{\En[f_i^2x_{ij}^2v_{i}^2]} \lesssim 1$ with probability $1-o(1)$ since $\max_{1\leq j\leq p} |(\En-\barEp)[f_i^2x_{ij}^2v_{i}^2]|\leq \delta_n$ with probability $1-\Delta_n$, and $\max_{1\leq j\leq p}\barEp[f_i^2x_{ij}^2v_{i}^2] \leq \max_{1\leq j\leq p}\{\barEp[f_i^3x_{ij}^3v_{i}^3]\}^{2/3} \lesssim 1$.

The last statement follows from noting that the Lasso solution provides an upper bound to the approximation of the best model based on $\widehat T_\mtau$, and the application of Lemma \ref{Thm:RateEstimatedLasso}.{\bf $\square$}

\begin{lemma}[Empirical pre-sparsity for Lasso]\label{Lemma:SparsityLASSO}
Let $\widehat T_\mtau$ denote the support selected by the Lasso estimator, $\hat s_\mtau = |\widehat T_\mtau|$, assume $\lambda/n \geq c\|\En[\widehat \Gamma_{\tau 0}^{-1}f_ix_iv_i]\|_{\infty}$, and $\ell \widehat \Gamma_{\tau 0} \leq \widehat \Gamma_\tau \leq u \widehat \Gamma_{\tau 0}$ with $u \geq 1 \geq \ell > 1/c$. Then, for $c_0=(uc+1)/(\ell c - 1)$ and $\tilde \cc = (uc+1)/(\ell c - 1)\|\hat\Gamma_{\tau 0}\|_\infty\|\hat\Gamma_{\tau 0}^{-1}\|_\infty$ we have
$$ \sqrt{\hat s_\mtau} \leq 2\sqrt{\semax{\hat s_\mtau}}(1+3\|\hat f\|_\infty)\|\widehat \Gamma_{0}^{-1}\|_\infty c_0\[\frac{n \{ c_f+c_r\}}{\lambda} + \frac{\sqrt{s}\|\hat \Gamma_{\tau 0}\|_\infty}{\kappa_{\tilde \cc}\min_{i\leq n}\widehat f_i}\].$$

\end{lemma}
{\bf Proof.}{\bf \ (Proof of Lemma \ref{Lemma:SparsityLASSO})} Let $\widehat F = \diag(\hat f)$, $R_\mtau = (r_{\mtau 1},\ldots, r_{\mtau n})'$, and $X=[ x_1;\ldots; x_n]'$. We have from the optimality conditions that the Lasso estimator $\widehat \theta_\tau $ satisfies  $$2\En[ \widehat \Gamma_{j}^{-1}\hat f_{i}^2x_i(d_i-x_i'\hat\theta_\tau)] = \sign(\hat\theta_{\tau j})\lambda/n \ \text{ for each } \ j \in \widehat T_\mtau.  $$

Therefore, noting that $\|\widehat \Gamma^{-1}\widehat \Gamma^{0}\|_\infty \leq 1/\ell$, we have
{\small \begin{eqnarray*}
   &  &\sqrt{\hat s_\mtau} \lambda    =   2\| (\widehat \Gamma^{-1} X'\widehat F^2(D - X\hat \theta_\tau))_{\widehat T_\mtau} \| \\
            &  & \leq   2\| (\widehat \Gamma^{-1}X'FV)_{\widehat T_\mtau} \| + 2\| (\widehat \Gamma^{-1}X'(\widehat F^2-F^2)F^{-1}V )_{\widehat T_\mtau} \| + 2\| (\widehat \Gamma^{-1}X'\widehat F^2R_\mtau)_{\widehat T_\mtau} \| \\
            & &+ 2\| (\widehat \Gamma^{-1}X'\widehat F^2X(\theta_\tau-\hat \theta_\tau))_{\widehat T_\mtau} \| \\
            & &  \leq  \sqrt{\hat s_\mtau}\ \|\widehat \Gamma^{-1}\widehat \Gamma_{0}\|_\infty\|\widehat \Gamma_{\tau 0}^{-1}X' F'V\|_{\infty} +  2n\sqrt{\semax{\hat s_\mtau}}\| \widehat \Gamma^{-1}\|_\infty\{c_f +\|\hat F\|_\infty c_r\}+ \\
             & & 2n \sqrt{ \semax{\hat s_\mtau}} \|\widehat F\|_\infty\|\widehat \Gamma^{-1}\|_\infty\|\hat f_i x_i'(\hat \theta_\tau - \theta_\tau) \|_{2,n},\\
            & & \leq \sqrt{\hat s_\mtau}\ (1/\ell) \ n\|\widehat \Gamma_{\tau 0}^{-1}X'FV\|_{\infty} + 2n \sqrt{ \semax{\hat s_\mtau}} \frac{\|\widehat \Gamma^{-1}_0\|_\infty}{\ell}(c_f + \|\widehat F\|_\infty c_r+\|\widehat F\|_\infty\|\hat f_ix_i'(\hat \theta_\tau- \theta_\tau) \|_{2,n}),
\end{eqnarray*}
}\!where we used that
$$\begin{array}{rcl}
&& \| (X'\widehat F^2(\theta_\tau-\hat \theta_\tau))_{\widehat T_\mtau} \| \\
&& \leq \sup_{\|\delta\|_0\leq \hat s_\mtau, \|\delta\|\leq 1}|
\delta' X'\widehat F^2 X(\theta_\tau-\hat \theta_\tau)| \leq
\sup_{\|\delta\|_0\leq \hat s_\mtau, \|\delta\|\leq 1}\| \delta'X'\widehat F'\|\|\widehat F X(\theta_\tau-\hat \theta_\tau)\| \\
&& \leq  \sup_{\|\delta\|_0\leq \hat
s_\mtau, \|\delta\|\leq 1}\{ \delta'X'\widehat F^2 X\delta\}^{1/2}\|\widehat F X(\theta_\tau-\hat \theta_\tau)\| \leq
n\sqrt{\semax{\hat s_\mtau}}\|\widehat f_i\|_\infty\|\widehat f_i x_i'(\theta_\tau-\hat \theta_\tau)\|_{2,n}, \\
&& \| (X'(\widehat F^2-F^2)F^{-1}V)_{\widehat T_\mtau} \|  \leq \sup_{\|\delta\|_0\leq \hat s_\mtau, \|\delta\|\leq 1}|\delta'X'(\widehat F^2-F^2)F^{-1}V|\\
&& \leq \sup_{\|\delta\|_0\leq \hat s_\mtau, \|\delta\|\leq 1}\|X\delta\| \ \|(\widehat F^2-F^2)F^{-1}V\| \leq n\sqrt{\semax{\hat s_\mtau}} c_f\\
\end{array}$$
Since $\lambda/c \geq \|\widehat \Gamma_{\tau 0}^{-1}X' FV\|_{\infty}$, and by Lemma \ref{Thm:RateEstimatedLasso}, $\|\widehat f_ix_i'(\widehat\theta_\tau-\theta_\tau)\|_{2,n} \leq 2\{c_f + c_r\}+\(u + \frac{1}{c}\) \frac{\lambda \sqrt{s}\|\hat \Gamma_{\tau 0}\|_\infty}{n \kappa_{\tilde \cc}\min_{i\leq n}\widehat f_i}$ we have
$$ \sqrt{\hat s_\mtau} \leq \frac{2\sqrt{\semax{\hat s_\mtau}}\frac{\|\widehat \Gamma_{0}^{-1}\|_\infty}{\ell}\[\frac{n  c_f}{\lambda}(1+2\|\hat F\|_\infty) + \frac{n c_r}{\lambda}3\|\hat F\|_\infty +\|\hat F\|_\infty\(u+\frac{1}{c}\)\frac{\sqrt{s}\|\hat \Gamma_{\tau 0}\|_\infty}{\kappa_{\tilde \cc}\min_{i\leq n}\widehat f_i}\]}{\(1-\frac{1}{c\ell}\)}.$$

The result follows by noting that $(u+[1/c])/(1-1/[\ell c]) = c_0 \ell$ by definition of $c_0$.
{\bf $\square$}

\begin{lemma}[Sub-linearity of maximal sparse eigenvalues]
\label{Lemma:SparseEigenvalueIMP}Let $M$ be a semi-definite positive matrix. For any integer $k \geq 0$ and constant $\ell \geq 1$ we have $ \semax{\ceil{\ell k}}(M) \leq  \lceil \ell \rceil \semax{k}(M).$
\end{lemma}

\begin{lemma}[Sparsity for Estimated Lasso under data-driven penalty]\label{Thm:Sparsity}
Consider the Lasso estimator $\widehat \theta_\tau$, let $\widehat s_\mtau = |\widehat T_\mtau|$, and assume that $\lambda/n \geq c\|\En[\widehat \Gamma_{\tau 0}^{-1}f_ix_iv_i]\|_{\infty}$. Consider the set $$\mathcal{M}=\left\{ m \in \mathbb{N}:
m > 8\semax{m}(1+3\|\hat f\|_\infty)^2\|\widehat \Gamma_{0}^{-1}\|_\infty^2 c_0^2\[\frac{n \{ c_f+ c_r\}}{\lambda} + \frac{\sqrt{s}\|\hat \Gamma_{\tau 0}\|_\infty}{\kappa_{\tilde \cc}\min_{i\leq n}\widehat f_i}\]^2 \right\}.$$ Then,
$$ \hat s_\mtau \leq  4 \left( \min_{m \in \mathcal{M}}\semax{m}\right)(1+3\|\hat f\|_\infty)^2\|\widehat \Gamma_{0}^{-1}\|_\infty^2 c_0^2\[\frac{n \{c_f+ c_r\}}{\lambda} + \frac{\sqrt{s}\|\hat \Gamma_{\tau 0}\|_\infty}{\kappa_{\tilde \cc}\min_{i\leq n}\widehat f_i}\]^2 .$$
\end{lemma}
{\bf Proof.}{\bf \ (Proof of Lemma \ref{Thm:Sparsity})}
Let $L_n = 2(1+3\|\hat f\|_\infty)\|\widehat \Gamma_{0}^{-1}\|_\infty c_0\[\frac{n \{c_f+c_r\}}{\lambda} + \frac{\sqrt{s}\|\hat \Gamma_{\tau 0}\|_\infty}{\kappa_{\tilde \cc}\min_{i\leq n}\widehat f_i}\].$
Rewriting the conclusion in Lemma \ref{Lemma:SparsityLASSO} we have
 \begin{equation}\label{Eq:Sparsity}\hat s_\mtau \leq \semax{\hat s_\mtau}L_n^2.\end{equation}
Consider any $M \in \mathcal{M}$, and suppose $\widehat s_\mtau > M$. Therefore by the sublinearity of the maximum sparse eigenvalue (see Lemma \ref{Lemma:SparseEigenvalueIMP})
$$ \hat s_\mtau  \leq  \ceil{\frac{\hat s_\mtau}{M}}\semax{M}L_n^2.$$ Thus, since $\ceil{k}\leq 2k$ for any $k\geq 1$  we have
$$ M \leq   2\semax{M}L_n^2$$
which violates the condition that $M \in \mathcal{M}$. Therefore, we have $\widehat s_\mtau \leq M$.

In turn, applying (\ref{Eq:Sparsity}) once more with $\widehat s_\mtau \leq M$ we obtain
 $$ \hat s_\mtau \leq   \semax{M}L_n^2.$$

The result follows by minimizing the bound over $M \in \mathcal{M}$.
{\bf $\square$}

\section{Proofs for Section \ref{Sec:EstIVQR} of Supplementary Material}

In this section we denote the nuisance parameters as $\tilde h = (\tilde g, \tilde \z)$, where $\tilde g$ is a function of variable $z\in \mathcal{Z}$, and $\tilde \z$ is a function on $(d,z)\mapsto \tilde \z(d,z)$. We define the score with $(\tilde \alpha,\tilde h)$ as
$$\psi_{\tilde \alpha,\tilde h}(y_i,d_i,z_i) = (\tau - 1\{y_i \leq \tilde g(z_i)+d_i\tilde \alpha\})\tilde \z(d_i,x_i).$$
For notational convenience we write $\tilde \z_i=\tilde \z(d_i,z_i)$ and $\tilde g_i = \tilde g(z_i)$, $h_0 = (g_{\tau},\z_0)$ and $\hat h = (\hat g,\hat \z)$.
 For a fixed $\tilde \alpha\in\RR$, $\tilde g:\mathcal{Z}\to \RR$, and $\tilde \z:\mathcal{D}\times\mathcal{Z} \to \RR$ we define
$$\left. \Gamma(\tilde\alpha,\tilde h) := \barEp[ \psi_{\alpha, h}(y_i,d_i,z_i)]\right|_{\alpha = \tilde \alpha, h = \tilde h} $$
The partial derivative of $\Gamma$ with respect to $\alpha$ at $(\tilde\alpha,\tilde h)$ is denoted by $\Gamma_\alpha(\tilde \alpha,\tilde h)$ and the directional derivative with respect to $[\hat h-h_0]$ at $(\tilde\alpha,\tilde h)$ is denote as
$$ \Gamma_h(\tilde\alpha,\tilde h)[\hat h- h_0] = \lim_{t\to 0} \frac{\Gamma(\tilde\alpha,\tilde h+t[\hat h- h_0])-\Gamma(\tilde\alpha,\tilde h)}{t}.$$

{\bf Proof.}{\bf \ (Proof of Lemma \ref{Thm:EstIVQR})}
The asymptotic normality results is established in Steps 1-4 which assume Condition IQR(i-iii). The additional results (derived on Steps 5 and 6) also assumed Condition IQR(iv).

\noindent Step 1. {\rm (Normality result)} We have the following identity
{\small \begin{equation}\label{Def:IdentityHelp}
\begin{array}{rl}
\En[\psi_{\check\alpha_\tau,\hat h}(y_i,d_i,z_i)] & = \En[\psi_{\alpha_\tau,h_0}(y_i,d_i,z_i)]+\En[\psi_{\check\alpha_\tau,\hat h}(y_i,d_i,z_i)-\psi_{\alpha_\tau, h_0}(y_i,d_i,z_i)]\\
 &=  \En[\psi_{\alpha_\tau,h_0}(y_i,d_i,z_i)]+\underbrace{\Gamma(\check\alpha_\tau,\hat h)}_{(I)}\\
 & +\underbrace{n^{-1/2}\Gn(\psi_{\check\alpha_\tau,\hat h}-\psi_{\check\alpha_\tau,h_0})}_{(II)} +\underbrace{n^{-1/2}\Gn(\psi_{\check\alpha_\tau,h_0}-\psi_{\alpha_\tau,h_0})}_{(III)}\\
\end{array}
\end{equation}}

\noindent By the second relation in (\ref{Eq:HLhatalpha}), Condition IQR(iii), the left hand side of the display above satisfies we have $|\En[\psi_{\check\alpha_\tau,\hat h}(y_i,d_i,z_i)] |\lesssim \delta_n n^{-1/2}$  with probability at least $1-\Delta_n$. Since $\hat h \in \overline{\mathcal{F}}$ with probability at least $1-\Delta_n$ by Condition IQR(iii), with the same probability we have $|(II)| \lesssim \delta_n n^{-1/2}.$

We now proceed to bound term $(III)$. By Condition IQR(iii) we have with probability at least $1-\Delta_n$ that $|\check \alpha_\tau - \alpha_\tau|\leq \delta_n$. Observe that
\begin{align*}
(\psi_{\alpha,h_0}-\psi_{\alpha_\tau,h_0})(y_{i},d_{i},z_{i}) &=( 1\{y_{i} \leq g_{\tau i} + d_{i} \alpha_\tau\} - 1\{y_i\leq g_{\tau_i} +d_i \alpha \}) \z_{0i} \\
&=(1\{\epsilon_{i} \leq 0\} - 1\{\epsilon_{i} \leq d_{i}(\alpha - \alpha_\tau)\} ) \z_{0i},
\end{align*}
so that $| (\psi_{\alpha,h_{0}}-\psi_{\alpha_\tau,h_{0}})(y_{i},d_{i},z_{i}) | \leq 1\{| \epsilon_{i} | \leq \delta_{n}|d_{i}|\}| \z_{0i} |$ whenever $| \alpha - \alpha_\tau | \leq \delta_{n}$. Since the class of functions $\{ (y,d,z) \mapsto (\psi_{\alpha,h_0}-\psi_{\alpha_\tau,h_0})(y,d,z) : | \alpha - \alpha_\tau | \leq \delta_{n} \}$ is a VC subgraph class with VC index bounded by some constant independent of $n$,
using (a version of) Theorem 2.14.1 in \cite{vdV-W}, we have
\begin{equation*}
\sup_{| \alpha - \alpha_\tau| \leq \delta_{n}}|  \Gn (\psi_{\alpha,h_0}-\psi_{\alpha_\tau,h_0}) | \lesssim_{P} (\barEp [ 1\{| \epsilon_{i} | \leq \delta_{n}|d_{i}|\} \z_{0i}^{2} ])^{1/2}  \lesssim_{P} \delta_{n}^{1/2}.
\end{equation*}
This implies that $| III | \lesssim \delta_{n}^{1/3} n^{-1/2}$ with probability $1-o(1)$.

Therefore we have $ 0 = \En[\psi_{\alpha_\tau,h_0}(y_i,d_i,z_i)] + (II) + O_P(\delta_n^{1/2}n^{-1/2}) + O_P(\delta_n)|\check\alpha_\tau-\alpha_\tau|.$ Step 2 below establishes that $(II)= -\barEp[f_id_i\z_{0i}](\check\alpha_\tau-\alpha_\tau)| +O_P(\delta_n n^{-1/2} + \delta_n|\check\alpha_\tau-\alpha_\tau|).$ Combining these relations we have
\begin{equation}\label{FinalIV} \barEp[f_id_i\z_{0i}](\check\alpha_\tau-\alpha_\tau) = \En[\psi_{\alpha_\tau,h_0}(y_i,d_i,z_i)] + O_P(\delta_n^{1/2}n^{-1/2}) + O_P(\delta_n)|\check\alpha_\tau-\alpha_\tau|.\end{equation}
Note that $\mathbb{U}_n(\tau)= \{\barEp[\psi_{\alpha_\tau,h_0}^2(y_i,d_i,z_i)]\}^{-1/2}\sqrt{n}\En[\psi_{\alpha_\tau,h_0}(y_i,d_i,z_i)]$ and $\barEp[\psi_{\alpha_\tau,h_0}^2(y_i,d_i,z_i)]=\tau(1-\tau)\barEp[\z_{0i}^2]$ so that the first representation result follows from (\ref{FinalIV}).
\noindent Since $\barEp[\psi_{\alpha_\tau,h_0}(y_i,d_i,z_i)]=0$ and $\barEp[\z_{0i}^3]\leq C$, by the Lyapunov CLT we have $$\sqrt{n}\En[\psi_{\alpha_\tau,h_0}(y_i,d_i,z_i)] = \sqrt{n}\En[\psi_{\alpha_\tau,h_0}(y_i,d_i,z_i)] \rightsquigarrow N(0,\barEp[\tau(1-\tau)\z_{0i}^2])$$
and $\mathbb{U}_n(\tau)\rightsquigarrow N(0,1)$ follows by noting that $|\barEp[f_id_i\z_{0i}]|\geq c>0$.

~\\
\noindent Step 2.{\rm  (Bounding $\Gamma(\alpha,\hat h)$ for $|\alpha-\alpha_\tau|\leq\delta_n$)}
For any (fixed function) $\hat h \in \overline{\mathcal{F}}$, we have
\begin{equation}\label{Eq:BiasIQR}
\begin{array}{rl}
\Gamma(\alpha,\hat h) & = \Gamma(\alpha,h_0) + \Gamma(\alpha,\hat h)- \Gamma(\alpha,h_0)\\
& = \Gamma(\alpha, h_0) + \{ \Gamma(\alpha,\hat h) - \Gamma(\alpha,h_0) - \Gamma_h(\alpha,h_0)[\hat h - h_0]\}  + \Gamma_h(\alpha,h_0)[\hat h - h_0].\\
\end{array}
\end{equation}
Because $ \Gamma(\alpha_\tau,h_0) = 0$, by Taylor expansion there is some $\tilde \alpha \in [\alpha_\tau, \alpha]$ such that $$\Gamma(\alpha,h_0) = \Gamma(\alpha_\tau,h_0) + \Gamma_\alpha(\tilde \alpha,h_0)( \alpha - \alpha_\tau) = \left\{ \Gamma_\alpha(\alpha_\tau,h_0) + \eta_n \right\} ( \alpha - \alpha_\tau)$$
where $|\eta_n| \leq \delta_n \barEp[|d_i^2\z_{0i}|]\leq \delta_n C$ by relation (\ref{EqGamma1}) in Step 4.

Combining the argument above with relations (\ref{EqGamma20}), (\ref{BoundGamma2first}) and (\ref{BoundGamma2third}) in Step 3 below we have
\begin{equation}\label{EqBias}
\begin{array}{rl}
\Gamma(\alpha,\hat h) & =  \Gamma_h(\alpha_\tau,h_0)[\hat h - h_0] + \Gamma(\alpha_\tau,h_0) + \{\Gamma_\alpha(\alpha_\tau,h_0)\\
 & + O(\delta_n \barEp[|d_i^2\z_{0i}|])\}( \alpha - \alpha_\tau) + O(\delta_nn^{-1/2})\\
& = \Gamma_\alpha(\alpha_\tau,h_0) ( \alpha - \alpha_\tau) + O(\delta_n|\alpha-\alpha_\tau| \barEp[|d_i^2\z_{0i}|] + \delta_nn^{-1/2})\\
\end{array}
\end{equation}

\noindent Step 3. {\rm (Relations for $\Gamma_h$)} The directional derivative $\Gamma_h$ with respect the direction $\hat h - h_0$ at a point $\tilde h = (\tilde g,\tilde z)$ is given by $$\begin{array}{rl}
 \Gamma_h(\alpha, \tilde h)[\hat h - h_0]  &  = -\barEp[f_{\epsilon_i\mid d_i,z_i}(d_i(\alpha-\alpha_\tau)+\tilde g_i - g_{\tau i})\tilde \z_{0i}\{\hat g_i - g_{\tau i}\}]  \\
   & + \barEp[ ( \tau - 1\{ y_i \leq \tilde g_i + d_i\alpha\}) \{\hat \z_i-\z_{0i}\}]\\
 \end{array}$$  Note that when $\Gamma_h$ is evaluated at $(\alpha_\tau, h_0)$ we have with probability $1-\Delta_n$ \begin{equation}\label{EqGamma20}|\Gamma_h(\alpha_\tau,h_0)[\hat h - h_0]| = |-\barEp[f_i\z_{0i}\{\hat g_i- g_{\tau i}\}]| \leq \delta_n n^{-1/2}\end{equation} by $\hat h \in \overline{\mathcal{F}}$ with probability at least $1-\Delta_n$, and by $P(y_i \leq g_{\tau i} + d_i\alpha_\tau\mid d_i, z_i) = \tau$.
The expression for $\Gamma_h$ also leads to the following bound
\begin{equation}\label{BoundGamma2first}\begin{array}{l} \left| \Gamma_h(\alpha,h_0)[\hat h - h_0]\right. - \left.  \Gamma_h(\alpha_\tau, h_0)[\hat h - h_0]\right| = \\
= |\barEp[ \{f_{\epsilon_i\mid d_i,z_i}(0)-f_{\epsilon_i\mid d_i,z_i}(d_i(\alpha-\alpha_\tau))\}\z_{0i}\{\hat g_i - g_{\tau i}\}]\\ + \barEp[\{F_i(0)-F_i(d_i(\alpha-\alpha_\tau))\}\{\hat\z_i-\z_{0i}\}]|\\
\leq \barEp[|\alpha-\alpha_\tau| \ \bar f' |d_i  \z_{0i}| \ |\hat g_i - g_{\tau i}|] + \barEp[ \bar f |(\alpha-\alpha_\tau)d_i| \  |\hat \z_i-\z_{0i}|]\\
 \leq\bar f' |\alpha-\alpha_\tau| \{ \barEp[|\hat g_i - g_{\tau i}|^2] \  \barEp[\z_{0i}^2d_i^2]\}^{1/2}
 +\bar f|\alpha-\alpha_\tau|  \{\barEp[(\hat \z_i-\z_{0i})^2] \barEp[d_i^2]\}^{1/2}\\
 \lesssim_P |\alpha-\alpha_\tau|\delta_n\\
   \end{array}\end{equation}

The second directional derivative $\Gamma_{hh}$ at $\tilde h = (\tilde g, \tilde \z)$ with respect to the direction $\hat h - h_0$, provided $\hat h \in \overline{\mathcal{F}}$, can be bounded by
\begin{equation}\label{BoundGamma2second}\begin{array}{rl}
 \left|\Gamma_{hh}(\alpha, \tilde h)[\hat h - h_0, \hat h - h_0]  \right| & =  \left| -\barEp[f_{\epsilon_i\mid d_i,z_i}'(d_i(\alpha-\alpha_\tau)+\tilde g_i-g_{\tau i})\tilde \z_i\{\hat g_i - g_{\tau i}\}^2]  \right. \\
  & \left. + 2\En[ f_{\epsilon_i\mid d_i, z_i}(d_i(\alpha-\alpha_\tau)+\tilde g_i - g_{\tau i})\{\hat g_i - g_{\tau i}\}\{\hat \z_i - \z_{0i}\}]\right|\\
& \leq  \bar f'\barEp[|\z_{0i}|\{\hat g_i - g_{\tau i}\}^2]  + \bar f'\barEp[|\hat\z_i - \z_{0i}|\{\hat g_i - g_{\tau i}\}^2] \\
  &  + 2\bar f\barEp[ |\hat g_i - g_{\tau i}|| \ |\hat \z_i - \z_{0i}|]\\
& \leq \delta_n n^{-1/2}\\
 \end{array}\end{equation}
since $\tilde h \in [ h_0, \hat h]$, $|\tilde \z(d_i,z_i)|\leq |\z_{0}(d_i,z_i)|+|\hat \z(d_i,z_i)-\z_{0}(d_i,z_i)|$, and the last bound follows from $\hat h \in \overline{\mathcal{F}}$.

Therefore, provided that $\hat h \in \overline{\mathcal{F}}$,  we have
\begin{equation}\label{BoundGamma2third}\begin{array}{rl}
 \left|\Gamma(\alpha,\hat h) - \Gamma(\alpha,h_0)  - \Gamma_h(\alpha,h_0)\left[\hat h-h_0\right]\right| & \leq {\displaystyle \sup_{\tilde h \in [ h_0, \hat h ]}} \left| \Gamma_{hh}(\alpha, \tilde h)\left[\hat h - h_0, \hat h-h_0\right] \right|\\
 & \lesssim \delta_n n^{-1/2}.
 \end{array} \end{equation}

~\\
\noindent Step 4. {\rm (Relations for $\Gamma_\alpha$)} By definition of $\Gamma$, its derivative with respect to $\alpha$ at $(\alpha,\tilde h)$ is
$$ \Gamma_\alpha(\alpha,\tilde h) = -\barEp[ f_{\epsilon_i\mid d_i,z_i}(d_i(\alpha-\alpha_\tau)+\tilde g_i-g_{\tau i})d_i\tilde \z_i]. $$
Therefore, evaluating $\Gamma_\alpha(\alpha,\tilde h)$ at $\alpha=\alpha_\tau$ and $\tilde h = h_0$, since for $ f_{\epsilon_i\mid d_i,z_i}(0)=f_i$  we have
\begin{equation}\label{EqGamma10} \Gamma_\alpha(\alpha_\tau,h_0) = -\barEp[ f_i d_i \z_{0i}].\end{equation} Moreover, $\Gamma_\alpha$ also satisfies
 \begin{equation}\label{EqGamma1}
 \begin{array}{rl}
 \left| \Gamma_\alpha(\alpha,h_0) - \Gamma_\alpha(\alpha_\tau,h_0)\right| & = \left| \barEp[ f_{\epsilon_i\mid d_i,z_i}(d_i(\alpha-\alpha_\tau)\mid d_i,z_i) \z_{0i} d_i] - \barEp[ f_i \z_{0i} d_i]\right|\\
 & \leq |\alpha-\alpha_\tau| \bar f'\barEp[ | d_i^2\z_{0i}|] \\
 & \leq |\alpha-\alpha_\tau| \bar f'\{\barEp[ d_i^4]\barEp[\z_{0i}^2]\}^{1/2} \leq C'|\alpha-\alpha_\tau|\end{array}\end{equation}
since $\barEp[d_i^4] \vee \barEp[\z_P{0i}^4] \leq C$ and $\bar f'<C$ by Condition IQR(i).

\noindent Step 5. {\rm (Estimation of Variance)}
First note that
\begin{equation}\label{Ed:FirstVarianceTerm}\begin{array}{rl}
& |\En[\hat f_id_i\hat \z_i]-\barEp[f_id_i\z_{0i}] | \\
& = | \En[\hat f_id_i\hat \z_i]-\En[f_id_i\z_{0i}] |+| \En[ f_id_i \z_{0i}]-\barEp[f_id_i\z_{0i}] |\\
& \leq | \En[(\hat f_i-f_i)d_i\hat \z_i]|+|\En[f_id_i(\hat \z_i-\z_{0i})] |+ |\En[ f_id_i \z_{0i}]-\barEp[f_id_i\z_{0i}]|\\
& \leq | \En[(\hat f_i-f_i)d_i(\hat \z_i-\z_{0i})]|+| \En[(\hat f_i-f_i)d_i\z_{0i}]|\\
& +\|f_id_i\|_{2,n}\|\hat \z_i-\z_{0i}\|_{2,n} + |\ \En[ f_id_i \z_{0i}]-\barEp[f_id_i\z_{0i}] |\\
& \lesssim_P \|(\hat f_i-f_i)d_i\|_{2,n}\|\hat \z_i-\z_{0i}\|_{2,n}+\|\hat f_i-f_i\|_{2,n}\|d_i\z_{0i}\|_{2,n}\\
&+\|f_id_i\|_{2,n}\|\hat \z_i-\z_{0i}\|_{2,n} + | \En[ f_id_i \z_{0i}]-\barEp[f_id_i\z_{0i}] |\\
& \lesssim_P \delta_n\end{array}\end{equation}
because $f_i, \hat f_i \leq C$, $\barEp[d_i^4] \leq C$, $\barEp[\z_{0i}^4]\leq C$ by Condition IQR(i) and Condition IQR(iv).


Next we proceed to control the other term of the variance. We have
\begin{equation}\label{Eq:Denominator}\begin{array}{rl}
& |\ \|\psi_{\check\alpha_\tau,\hat h}(y_i,d_i,z_i)\|_{2,n}  -  \| \psi_{\alpha_\tau,h_0}(y_i,d_i,z_i)\|_{2,n}  | \leq \|\psi_{\check\alpha_\tau,\hat h}(y_i,d_i,z_i) -  \psi_{\alpha_\tau,h_0}(y_i,d_i,z_i)\|_{2,n} \\
& \leq \|\psi_{\check\alpha_\tau,\hat h}(y_i,d_i,z_i) - (\tau-1\{y_i\leq d_i\check\alpha_\tau+\tilde g_i\})\z_{0i}\|_{2,n} \\
&+ \|(\tau-1\{y_i\leq d_i\check\alpha_\tau+\tilde g_i\})\z_{0i} - \psi_{\alpha_\tau,h_0}(y_i,d_i,z_i)\|_{2,n} \\
& \leq \|\hat \z_i-\z_{0i}\|_{2,n} +\|(1\{y_i\leq d_i\alpha_\tau+ g_{\tau i}\}-1\{y_i\leq d_i\check\alpha_\tau+\tilde g_i\})\z_{0i}\|_{2,n}\\
& \leq \|\hat \z_i-\z_{0i}\|_{2,n} +\|\z_{0i}^2\|_{2,n}^{1/2}\|1\{|\epsilon_i|\leq |d_i(\alpha_\tau-\check\alpha_\tau)+ g_{\tau i}-\tilde g_i|\}\|_{2,n}^{1/2}\\
& \lesssim_P \delta_n\end{array}\end{equation}
 by IQR(ii) and IQR(iv). Also, $|\En[\psi_{\alpha_\tau,h_0}^2(y_i,d_i,z_i)] - \barEp[\psi_{\alpha_\tau,h_0}^2(y_i,d_i,z_i)]|\lesssim_P \delta_n$ by independence and bounded moment conditions in Condition IQR(ii).

Step 6. (Main Step for $\chi^2$) Note that the denominator of $L_n(\alpha_\tau)$ was analyzed in relation (\ref{Eq:Denominator}) of Step 5. Next consider the numerator of $L_n(\alpha_\tau)$. Since $\Gamma(\alpha_\tau,h_0)=\barEp[\psi_{\alpha_\tau, h_0}(y_i,d_i,z_i)]=0$ we have $$\En[ \psi_{\alpha_\tau,\hat h}(y_i,d_i,z_i)] = (\En-\barEp)[ \psi_{\alpha_\tau,\hat h}(y_i,d_i,z_i) - \psi_{\alpha_\tau, h_0}(y_i,d_i,z_i)] + \Gamma(\alpha_\tau,\hat h)+\En[\psi_{\alpha_\tau, h_0}(y_i,d_i,z_i)].$$
By $\hat h \in \overline{\mathcal{F}}$ with probability $1-\Delta_n$  and (\ref{EqBias}) with $\alpha=\alpha_\tau$, it follows that with the same probability
$$ |(\En-\barEp)[ \psi_{\alpha_\tau,\hat h}(y_i,d_i,z_i) - \psi_{\alpha_\tau, h_0}(y_i,d_i,z_i)]| \leq \delta_nn^{-1/2} \ \ \mbox{and} \ \ |\Gamma(\alpha_\tau,\hat h)| \lesssim \delta_n n^{-1/2}.$$
The identity $nA_n^2 = nB_n^2 + n(A_n-B_n)^2+2nB_n(A_n-B_n)$ for $A_n =\En[  \psi_{\alpha_\tau,\hat h}(y_i,d_i,x_i)]$  and  $B_n = \En[  \psi_{\alpha_\tau, h_0}(y_i,d_i,x_i)] \lesssim_P \{\tau(1-\tau)\barEp[\z_{0i}^2]\}^{1/2}n^{-1/2}$ yields
\begin{eqnarray*}
 nL_n(\alpha_\tau) & = & \frac{n|\En[  \psi_{\alpha_\tau,\hat h}(y_i,d_i,z_i)]|^2}{\En[ \psi_{\alpha_\tau,\hat h}^2(y_i,d_i,z_i)]} \\
 & =  & \frac{n|\En[  \psi_{\alpha_\tau,h_0}(y_i,d_i,z_i)]|^2+O_P(\delta_n)}{\barEp[\tau(1-\tau)\z_{0i}^2]+O_P(\delta_n)}= \frac{n|\En[ \psi_{\alpha_\tau, h_0}(y_i,d_i,z_i)]|^2}{\barEp[\tau(1-\tau)\z_{0i}^2]} +O_P(\delta_n)
  \end{eqnarray*}
since $\tau(1-\tau)\barEp[\z_{0i}^2]$ is bounded away from zero because $\underline{C} \leq |\barEp[f_id_i\z_{0i}]|=|\barEp[v_i\z_{0i}]|\leq  \{ \barEp[v_i^2] \barEp[\z_{0i}^2]\}^{1/2}$ and $\barEp[v_i^2]$ is bounded above uniformly. Therefore, the result then follows since $\sqrt{n}\En[  \psi_{\alpha_\tau,h_0}(y_i,d_i,z_i)]\rightsquigarrow N(0,\tau(1-\tau)\barEp[\z_{0i}^2])$.

{\bf $\square$}

\section{Rates of convergence for $\hat f$}

Let $\hat Q(u\mid \tilde x) = \tilde x'\hat\eta_u$ for $u=\tau-h,\tau+h$. Using a Taylor expansion for the conditional quantile function $Q(\cdot\mid \tilde x)$,  assuming that $\sup_{|\tilde \tau-\tau|\leq h}|Q'''(\tilde \tau\mid \tilde x)|\leq C$ we have
$$ | \widehat Q'(\tau\mid \tilde x)-Q'(\tau\mid \tilde x)| \leq \frac{|Q(\tau+h\mid \tilde x)-\tilde x'\hat\eta_{\tau+h}|+|Q(\tau-h\mid \tilde x)-\tilde x'\hat\eta_{\tau-h}|}{h} + Ch^2.$$
In turn, to estimate $f_i$, the conditional density at $Q(\tau\mid \tilde x)$, we set $\hat f_i = 1/ \widehat Q'(\tau\mid \tilde x_i)$ which leads to
\begin{equation}\label{Eq:Errorhatf}|f_i - \hat f_i | = \frac{| \widehat Q'(\tau\mid \tilde x_i)-Q'(\tau\mid \tilde x_i)|}{\widehat Q'(\tau\mid \tilde x_i)Q'(\tau\mid \tilde x_i)}= (\hat f_i f_i)\cdot | \widehat Q'(\tau\mid \tilde x_i)-Q'(\tau\mid \tilde x_i)|.\end{equation}

\begin{lemma}[Bound Rates for Density Estimator]\label{lemma:boundr2nrinftyn}
Let $\tilde x = (d,x)$, suppose that $c\leq f_i \leq C$, $\sup_\epsilon f'_{\epsilon_i\mid \tilde x_i}(\epsilon\mid \tilde x_i) \leq \bar f'\leq C$, $i=1,\ldots,n$, uniformly in $n$. Assume further that with probability $1-\Delta_n$ we have for $u=\tau-h, \tau+h$ that
$$ \|\tilde x_i'(\hat\eta_u-\eta_u)+r_{ui}\|_{2,n}\leq \frac{C}{\kappa_\cc}\sqrt{\frac{s\log (p\vee n)}{n}}, \ \ \|\hat\eta_u-\eta_u\|_{1}\leq \frac{C}{\kappa_\cc^2}\sqrt{\frac{s^2\log (p\vee n)}{n}} \ \ \ $$ $$\mbox{and} \ \ |\hat\eta_{u1} - \eta_{u1}| \leq \frac{C}{\kappa_\cc}\sqrt{\frac{s\log (p\vee n)}{n}}.$$  Then if ${\displaystyle \sup_{|\tilde \tau-\tau|\leq h}}|Q'''(\tilde \tau\mid \tilde x)|\leq C$,  ${\displaystyle \max_{i\leq n}}\|x_i\|_\infty\sqrt{s^2\log (p\vee n)} + \max_{i\leq n}|d_i| \sqrt{s\log(p\vee n)}\leq \delta_nh\kappa_\cc^2\sqrt{n}$ and $ {\displaystyle \max_{u=\tau+h,\tau-h}}\|r_{ui}\|_{\infty} \leq h\delta_n$  we have
$$ \|f_i - \hat f_i\|_{2,n} \lesssim_P \frac{1}{h\kappa_{\cc}}\sqrt{\frac{s\log (n\vee p)}{n}} + h^2, \ \ \mbox{and} $$ $$ \max_{i\leq n}|f_i - \hat f_i| \lesssim_P \max_{u=\tau+h,\tau-h}\frac{\|r_{ui}\|_{\infty}}{h}+\frac{\max_{i\leq n}\| x_i\|_\infty}{h\kappa_{\cc}^2}\sqrt{\frac{s^2\log (n\vee p)}{n}} $$ $$ + \frac{\max_{i\leq n}| d_i|_\infty}{h\kappa_{\cc}}\sqrt{\frac{s\log (n\vee p)}{n}} +h^2.$$\end{lemma}
{\bf Proof.}
Letting $(\delta^u_{\alpha};\delta^u_{\beta}) = \eta_{u} - \hat \eta_{u}$ and $\tilde x_i = (d_i,x_i')'$ we have that
 $$\begin{array}{rl}
  |\hat f_i - f_i| & \leq |f_i \hat f_i \frac{\tilde x_i'(\eta_{\tau + h} - \hat \eta_{\tau+h})+r_{\gtau +h,i} - \tilde x_i'(\eta_{\tau - h} - \hat \eta_{\tau-h})-r_{\gtau-h,i}}{2h}| + C h^2 \\
  & = h^{-1}(f_i \hat f_i)  | x_i'\delta^{\tau+h}_\beta + d_i\delta^{\tau+h}_\alpha+r_{\gtau +h,i} - x_i'\delta^{\tau-h}_\beta - d_i\delta^{\tau-h}_\alpha-r_{\gtau-h,i}\}| + C h^2 \\
  & \leq h^{-1}(f_i \hat f_i)  \left\{ K_x\|\eta_{\tau+h}\|_1 + K_x\|\eta_{\tau-h}\|_1+ |d_i| \cdot|\delta^{\tau+h}_\alpha|+|d_i| \cdot|\delta^{\tau-h}_\alpha|\right.\\\
  & \left.+ |r_{\gtau +h,i} -r_{\gtau-h,i}|\right\} + C h^2. \\ \end{array} $$
The result follows because for sequences $d_n\to 0, c_n\to 0$ we have $|\hat f_i-f_i|\leq |\hat f_i f_i|c_n + d_n $ implies that $\hat f_i (1-f_ic_n) \leq  f_i + d_n$. Since $f_i$ is bounded, $f_ic_n\to 0$ which implies that $\hat f_i$ is bounded. Therefore, $|\hat f_i-f_i|\lesssim c_n + d_n$. We take $d_n = Ch^2 \to 0$ and  $$c_n = h^{-1}\left\{ K_x\|\eta_{\tau+h}\|_1 + K_x\|\eta_{\tau-h}\|_1+ |d_i| \cdot|\delta^{\tau+h}_\alpha|+|d_i| \cdot|\delta^{\tau-h}_\alpha|+ |r_{\gtau +h,i} -r_{\gtau-h,i}|\right\} \to_P 0$$ by the growth condition.

Moreover, we have
{\small $$ \|(\hat f_i - f_i)/f_i\|_{2,n}  \lesssim \frac{\|\hat f_i \tilde x_i'(\hat \eta_{\tau + h} - \eta_{\tau + h})+\hat f_ir_{\gtau+h,i}\|_{2,n} + \|\hat f_i \tilde x_i'(\hat \eta_{\tau - h} - \eta_{\tau - h})+\hat f_ir_{\gtau+h,i}\|_{2,n}}{h} + Ch^2.$$}
By the previous result $\hat f_i$ is uniformly bounded from above with high probability. Thus, the  result follows by the assumed prediction norm rate $\|\tilde x_i'(\hat\eta_u-\eta_u)+r_{ui}\|_{2,n}$.

{\bf $\square$}


\end{document}